\documentclass[11pt,reqno]{amsart}
\usepackage{amscd,amssymb,amsmath,amsthm}
\usepackage{mathtools}
\usepackage[english]{babel}
\usepackage[T1]{fontenc}
\usepackage[arrow,matrix]{xy}
\usepackage{graphicx,subfigure,tikz,tikz-3dplot}

\usetikzlibrary{positioning,matrix,arrows,calc}
\usepackage{cite}
\usepackage{dsfont}
\usepackage{enumitem}
\usepackage{bbm}
\usepackage{hyperref}

\hypersetup{
	colorlinks=true,
	linkcolor=blue}
\usepackage{array}
\usepackage{xargs}                      
\usepackage[colorinlistoftodos,prependcaption,textsize=tiny]{todonotes}

\newcommand\ddfrac[2]{\frac{\displaystyle #1}{\displaystyle #2}}

\oddsidemargin=0.3in \evensidemargin=0.3in \theoremstyle{plain}
\newtheorem{theorem}{Theorem}[section]
\newtheorem{lemma}[theorem]{Lemma}
\newtheorem{proposition}[theorem]{Proposition}

\theoremstyle{definition}
\newtheorem{remark}[theorem]{Remark}

\newtheorem{example}[theorem]{Example}

\numberwithin{equation}{section} 

\newcommand{\T}{\mathbb{T}}
\newcommand{\N}{\mathbb{N}}
\newcommand{\R}{\mathbb{R}}
\newcounter{hypocounter}
\renewcommand\thehypocounter{(H\arabic{hypocounter})} 
\newenvironment{hypothesis}{
	\refstepcounter{hypocounter}
	\begin{flushright}
		\begin{tabular}{m{0.85\linewidth}r} 
		}{&\hfill{\bf \thehypocounter}\end{tabular}\end{flushright}}

\begin{document}


\title[Intermittent two-point dynamics at the transition to chaos]{Intermittent two-point dynamics at the transition to chaos for random circle endomorphisms }

\author{V.P.H. Goverse, A.J. Homburg, J.S.W. Lamb}

\address{V.P.H. Goverse\\  Department of Mathematics,
	Imperial College London,
	180 Queen's Gate,
	London SW7 2AZ,
	United Kingdom}
\email{vincent.goverse21@imperial.ac.uk}	

\address{A.J. Homburg\\ KdV Institute for Mathematics, University of Amsterdam, Science park 107, 1098 XG Amsterdam, Netherlands\newline Mathematics Institute, Leiden University, 
Einsteinweg 55,
2333 CC Leiden, Netherlands}
\email{a.j.homburg@uva.nl}

\address{J.S.W. Lamb\\  Department of Mathematics,
	Imperial College London,
	180 Queen's Gate,
	London SW7 2AZ,
	United Kingdom\newline
	International Research Center for Neurointelligence (IRCN), The University of Tokyo, Tokyo 
113-0033, Japan\newline
Centre for Applied Mathematics and Bioinformatics, Department of Mathematics and Natural Sciences, Gulf University for Science and Technology, Halwally, 32093 Kuwait.
	}
\email{jsw.lamb@imperial.ac.uk}


\begin{abstract}
We establish the existence of intermittent two-point dynamics and infinite stationary measures for a class of random circle endomorphisms with zero Lyapunov exponent, as a dynamical characterisation of the transition from synchronisation (negative Lyapunov exponent) to chaos (positive Lyapunov exponent).  
\end{abstract}

\maketitle

\section{Introduction}
%
%
%
In dynamical systems theory, the phenomenon of chaos (with hallmark sensitive dependence on initial conditions) has been an important motivation and focal point of research. In particular, the question of when and why chaotic dynamics emerges from more predictable motion, for instance in a parametrized family of dynamical systems, has been a central question in bifurcation theory. 

While several routes to chaos have been identified for deterministic dynamical systems (see for instance \cite{MR629208}), the corresponding transition in random dynamical systems (dynamical systems driven by a signal with certain probabilistic characteristics) remains much less understood. 

In this paper, we address this problem for a class of random circle endomorphisms, adopting notions of order and chaos in terms of the sign of the Lyapunov exponent, with negative Lyapunov exponent implying synchronisation (almost sure convergence of the distance between different trajectories with different initial conditions) and positive Lyapunov exponent implying chaos (including sensitive dependence on initial conditions). We are led to address the question how this transition happens, and in particular to identify dynamical aspects that accompany the change of sign of the Lyapunov exponent. We focus in this respect on the so-called two-point motion from the topological and (invariant) measure theoretical point of view.

In late 1980s, Baxendale and Stroock \cite{MR968817,MR1144097} studied the dynamics of stochastic differential equations, characterising stationary measures for the two-point motion in different Lyapunov exponent regimes, establishing in particular the existence of an infinite ergodic invariant measure at the zero Lyapunov transition point between synchronisation (stationary probability measure on the diagonal) and chaos (stationary probability measures on and off the diagonal). 

In the early 1990s, Pikovsky \cite{MR1165510} and Yu, Ott and Chen \cite{PhysRevLett.65.2935} studied the transition to chaos in the discrete time setting (random maps), leading to numerical evidence supporting several heuristic conjectures concerning intermittency.

Recently, Homburg and Kalle \cite{https://doi.org/10.48550/arxiv.2207.09987} obtained explicit results about stationary measures for certain random affine iterated function systems on the circle. They also point out the fact that the infinite stationary measure of the two-point motion at the transition corresponds to intermittent dynamics.

This leads to the natural question whether intermittency and associated infinite ergodic stationary measure of the two-point motion are generic features of the transition to chaos in random dynamical systems.

In this paper, we answer this question in the affirmative, in the specific setting of circle endomorphisms of degree two with additive noise.
Importantly, we develop techniques to analyse the two-point dynamics near the diagonal, expanding on results developed by Baxendale and Stroock \cite{MR968817,MR1144097} for stochastic differential equations. 
In particular, we analyse the spectral properties and actions of annealed Koopman operators for one- and two-point motions,
allowing us to derive quantitative estimates on various escape times of the quenched process, which are key 
to our results. In Section~\ref{s:strategy}, we present a summary of our techniques.



We anticipate that the techniques introduced in this paper will turn out to be fundamental to settle our question in general. 

\subsection{Main results}
Consider a smooth monotone circle endomorphism $T: \mathbb{T} \to \mathbb{T}$ of degree $k>1$, where $\mathbb T = \mathbb{R}/\mathbb{Z}$ denotes the circle endowed with the topology induced by the arc-length metric 
$d$, and the parameter family of maps defined as
\[T_a(x):=T(x+a\pmod{1}).\] 

We consider random iteration of maps from this family 

\[
x_n = T^n_\omega(x_0):= T_{\omega_{n-1}} \circ \cdots \circ T_{\omega_1} \circ T_{\omega_0} (x_0)
\]
where $\omega := (\omega_i)_{i\in\mathbb{N}}$ and $\omega_i$ is drawn randomly (i.i.d.) from
$[-\vartheta,\vartheta]$ with uniform measure $\mathrm{Leb}/(2\vartheta)$.
We denote the corresponding sequence space
\[
\Sigma_\vartheta \coloneqq [-\vartheta,\vartheta]^\mathbb{N}.
\]
with corresponding product measure $\mathbb{P}$.

Our results require certain mild hypotheses, Hypothesis~\ref{h:endo}-\ref{h:openaftertwo}, which we proceed to sketch with reference to Section~\ref{s:setting} for details. 

We assume that $T$ and $\vartheta$ are such that the random dynamical system has a unique stationary measure $\mu$ with smooth and everywhere positive density. 
As a result, the random dynamical system has a single Lyapunov exponent
\begin{equation*}
\lambda  =  \frac{1}{2\vartheta}
\int_\mathbb{T}\int_{[-\vartheta,\vartheta]} \ln \left(DT_\omega(x)\right)\, d\mathrm{Leb}(\omega) d\mu (x),
\end{equation*}
which can be negative, zero, or positive, depending on $T$ and $\vartheta$. 
When we consider the case where the  random dynamical system depends smoothly on an additional parameter, the Lyapunov exponent also varies smoothly with this parameter (due to our hypotheses). Then, a transition to chaos corresponds to the Lyapunov exponent  traversing zero from below.

We consider the two-point motion to gain insights into the dynamics:
\[
(x_n,y_n) = (T^{(2)}_\omega)^n (x_0,y_0) \coloneqq   \left( T^n_\omega  (x_0) , T^n_\omega  (y_0)
\right).
\]
Stationary measures for the two-point random dynamical system provide information well beyond stationary measure for (the one-point dynamics of) $T_\omega$.
Note that the stationary measure for $T_\omega$ is also a stationary measure of the two-point motion, supported on the diagonal 
\[
\Delta \coloneqq \left\{ (x,x) \in \mathbb{T}^2 \; ; \; x \in \mathbb{T} \right\}.
\]
We are primarily interested in stationary measures with support outside of $\Delta$, providing detail on the comparison between orbits with different initial conditions. The $\varepsilon$-neighbourhood of the diagonal is denoted as 
\[
\Delta_\varepsilon  \coloneqq \left\{ (x,y) \in \mathbb{T}^2 \; ; \; d(x,y) < \varepsilon\right\}.
\]

Our main result concerns the following trichotomy in phenomenology from a topological and measure theoretic point of view:

\begin{theorem}[Topological random dynamics]\label{thm:main1}
Let $T_\omega$ be a random dynamical system satisfying hypotheses
\ref{h:endo} to \ref{h:openaftertwo},
$\lambda$ be its Lyapunov exponent, and let $T^{(2)}_\omega$ be the corresponding two-point random dynamical system. Then,

	\begin{enumerate}
		\item \label{i:thm1:syn}\textbf{Synchronisation:} 	if $\lambda<0$, 
		for all $x,y \in \mathbb{T}$,
        \[
		\lim_{n\to\infty} d\left(T^n_\omega (x),T^n_\omega(y)\right) = 0,~\mathbb{P}-\mathrm{a.s.}
		\]
		
		\item \label{i:thm1:iter}\textbf{Intermittency:} if $\lambda=0$, 				
		for all $(x,y) \in \mathbb{T}^2\setminus \Delta$, 
		\[
			\lim_{n\to\infty} \frac{1}{n}\sum_{i=0}^{n-1}    d\left(T^i_\omega (x) ,T^i_\omega(y)\right)= 0
		~\textrm{and}~ 
			\limsup_{n\to\infty}d\left(T^n_\omega (x), T^n_\omega(y)\right) > 0,~\mathbb{P}-\mathrm{a.s.}
		\]

		\item \label{i:thm1:chao}\textbf{Chaos:} if $\lambda > 0$, for all $(x,y) \in \mathbb{T}^2\setminus \Delta$, 
\[
\lim_{n\to\infty}\frac{1}{n}\sum_{i=0}^{n-1}   d \left(T^i_\omega (x), T^i_\omega(y)\right) > 0
~\textrm{and}~
\liminf_{n\to\infty}d\left(T^n_\omega (x) ,T^n_\omega(y)\right) = 0,~\mathbb{P}-\mathrm{a.s.}
\]	
\end{enumerate}
\end{theorem}

This result mirrors those in \cite{https://doi.org/10.48550/arxiv.2207.09987}, obtained for special random affine circle maps. 
The most interesting aspect concerns the existence of intermittency in part \eqref{i:thm1:iter}. The existence of synchronisation in the presence of negative Lyapunov exponent has already been well-studied, see e.g. \cite{MR3820004}, and the properties highlighted under part \eqref{i:thm1:chao} also align with existing insights. 
The characterization of intermittency in part \eqref{i:thm1:iter} is reminiscent of synchronisation on average \cite{MR3519583} and that of chaos in part \eqref{i:thm1:chao} of Li-Yorke chaos \cite{MR385028,MR1900136}. 
These characterisations are not exhaustive. For instance, recently the positive Lyapunov exponent regime has been associated with the existence of so-called random horse-shoes \cite{lamb_horseshoes_2025,lamb2025nonuniformexpansionadditive}.

The proof of Theorem~\ref{thm:main1}, which uses techniques introduced in Sections~\ref{s:setting} and~\ref{s:koopman}, is given in Section~\ref{s:trd}. 
The results for $\lambda<0,\lambda =0$ and $\lambda >0$ are presented separately in   Propositions~\ref{prop:synchron},~\ref{prop:SynchronOnAverage} and \ref{prop:chaos}, respectively. 
	
\begin{theorem}[Stationary measures]\label{thm:main2}
Let $T_\omega$ be as in Theorem~\ref{thm:main1}, and let $\mu$ denote its unique stationary (probability) measure. 

Then, 
	\begin{enumerate}
		
		\item \label{i:main2:1}If $\lambda < 0$, $\mu$ on $\Delta$ is the unique stationary measure for $T^{(2)}_\omega$.
		\item \label{i:main2:2}If $\lambda=0$,  $\mu$ on $\Delta$ is the unique stationary probability measure for $T^{(2)}_\omega$. In addition,
        $T^{(2)}_\omega$ admits an infinite  stationary (Radon) measure $\mu^{(2)}$ on $\mathbb{T}^2\setminus \Delta$, which has full support.  Moreover, for each such measure
there exist $\alpha, \beta \in (0,\infty)$ such that
		\begin{equation*}
			\alpha \leq 	\liminf_{\varepsilon \to 0} \frac{\mu^{(2)}(\mathbb{T}^2\setminus\Delta_\varepsilon)}{-\ln(\varepsilon)} \leq \limsup_{\varepsilon \to 0} \frac{ \mu^{(2)}(\mathbb{T}^2\setminus\Delta_\varepsilon)}{-\ln(\varepsilon)} \leq \beta.
		\end{equation*}
\item \label{i:main2:3} If $\lambda >0$, in addition to the unique stationary measure $\mu$ on $\Delta$, $T^{(2)}_\omega$ admits a  stationary probability measure 
$\mu^{(2)}$ on $\mathbb{T}^2\setminus \Delta$, which has full support.

Moreover, if the non-zero root $\gamma<0$ of the moment Lyapunov function\footnote{The moment Lyapunov function will be defined below, see Eq.~\eqref{momlf} and Lemma~\ref{lem:second0}. }  associated with $T_\omega$  
is also larger than $-\frac{1}{2}$, then for each such measure
there exists $\alpha, \beta \in (0,\infty)$, such that
		\begin{equation*}
			\alpha \leq \liminf_{\varepsilon \to 0} \frac{\mu^{(2)}(\Delta_\varepsilon)}{\varepsilon^{-\gamma}} \leq \limsup_{\varepsilon \to 0} \frac{\mu^{(2)}(\Delta_\varepsilon)}{\varepsilon^{-\gamma}} \leq \beta.
		\end{equation*}
	\end{enumerate}
\end{theorem}

As in Theorem~\ref{thm:main1}, the most interesting aspect of this result is the intermittent case in part \eqref{i:main2:2}, where we find that ergodic invariant measures off the diagonal are infinite, together with the estimate on how such measures grow near the diagonal. 

{ {
For an ergodic stationary measure in case of a positive Lyapunov exponent, the estimates on $\mu^{(2)}(\Delta_\varepsilon)$ determine bounds on the frequency with which orbits of the random map are within distance $\varepsilon$ of each other, by Birkhoff's ergodic theorem.
}}
The asymptotics of the stationary invariant measure near the diagonal in the positive Lyapunov exponent scenario (in part \eqref{i:main2:3}) relies on the condition that $\gamma \in (-\frac{1}{2},0)$. If  $\gamma<-\frac{1}{2}$, then the asymptotics may be different due to
points being mapped more frequently close to the diagonal, see Proposition~\ref{prop:probmeas}. 


{  The construction of stationary measures off the diagonal for the two-point motion in parts \eqref{i:main2:2} and \eqref{i:main2:3} makes use of an inducing scheme with randomized return times, in order to apply a Krylov-Bogolyubov type argument. It should be noted that such reasoning does not provide uniqueness of stationary measures}.

The heart of the proof of the above theorem is presented in Section~\ref{s:stat}, with arguments relying on estimates from Sections~\ref{s:koopman},~\ref{s:trd}.
The existence of the stationary measure on $\Delta$ follows from Proposition~\ref{prop:uniquestatmeasure}.
For $\lambda <0$, Theorem~\ref{thm:main1} \eqref{i:thm1:syn} directly implies the unique stationary measure in Theorem~\ref{thm:main2} \eqref{i:main2:1}.
%
The results for $\lambda =0$ and $\lambda > 0$ follow from Proposition~\ref{prop:infmeas}, Proposition~\ref{prop:bounds} and Proposition~\ref{prop:probmeas}. 

We proceed to illustrate our results in an example.

\begin{example}\label{ex:T_nu} 
\begin{figure}[ht!]
    \centering
    \includegraphics[width=0.7\linewidth]{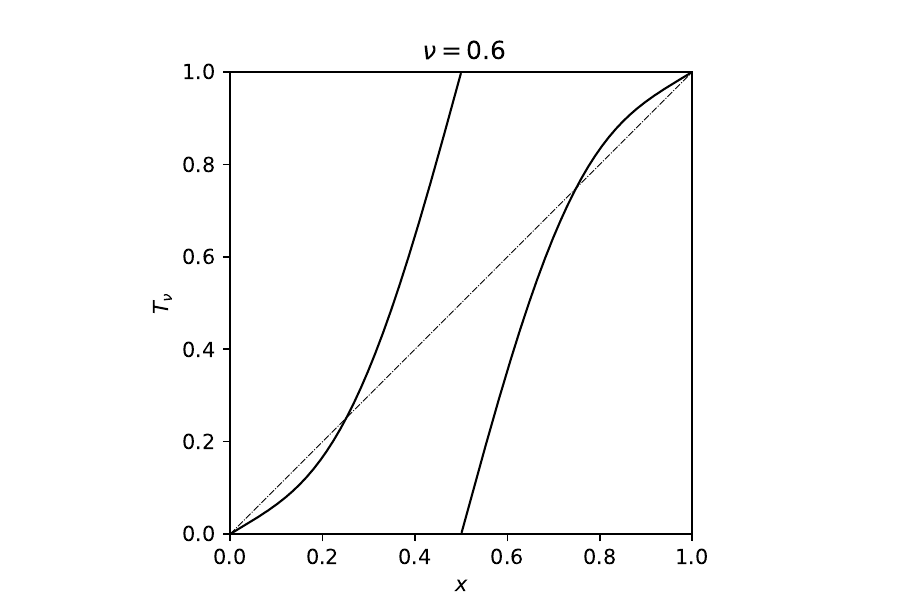}
    \caption{The graph of $T_{\nu}$ \eqref{eq:chatgpt} at $\nu=0.6$. }
    \label{fig:fa}
\end{figure}
We consider a one-parameter family of random circle endomorphism 
\[
x_n=T_{\nu,\omega}^n(x_0)
:= T_{\nu,\omega_{n-1}} \circ \cdots \circ T_{\nu,\omega_1} \circ T_{\nu,\omega_0} (x_0)
\]
with $\omega_i$  drawn i.i.d.\ from the uniform distribution on a subinterval of the circle \([-\vartheta,\vartheta]\), and
\begin{equation}\label{eq:RDS_ex}
    T_{\nu,a}(x) := T_\nu(x +a \pmod 1),
\end{equation}
with
\begin{equation}\label{eq:chatgpt}
    T_\nu(x) := \int_0^x \nu + 140 (2-\nu) t^3(1-t)^3 dt \pmod{1}.
\end{equation}
Note that $T_\nu(0)=0$ and $DT_\nu(0) = \nu$. 
In Figure \ref{fig:fa} we sketch the graph of $T_\nu$ for $\nu = 0.6$. 
{  Figure~\ref{fig:all_subfigs} depicts results of numerical experiments for this random family.  
We remark that it is possible to determine Lyapunov exponents and the Lyapunov moment functions rigorously 
 \cite{galatolo2025efficientcomputationstationarymeasures,MR4879452}. }

For parameter ranges discussed in this example $T_{\nu,\omega}$ satisfies the Hypotheses \ref{h:endo} to \ref{h:openaftertwo}, see Example \ref{ex:white-noise-continued}.

We first consider the behaviour of the Lyapunov exponent at a fixed value of $\nu= 0.6$ with varying $\vartheta$. In Figure~\ref{fig:bndlambda}, we observe noise-induced chaos, as the Lyapunov exponent increases monotonically from $-0.4$ to $0.45$ with increasing $\theta\in[0.1,0.5]$.
In this interval, the noise is 
large enough to ensure the existence of a unique stationary measure for $T_{0.6,\omega}$ on \(\mathbb{T}\). 

\begin{figure}[ht!]
    \centering
    
    \subfigure[Numerical approximation of the Lyapunov exponent for fixed $\nu = 0.6$ and varying $\vartheta$. The Lyapunov exponent increases monotonically  with the noise level, from negative to positive (noise-induced chaos).]{
        \includegraphics[width=0.45\textwidth]{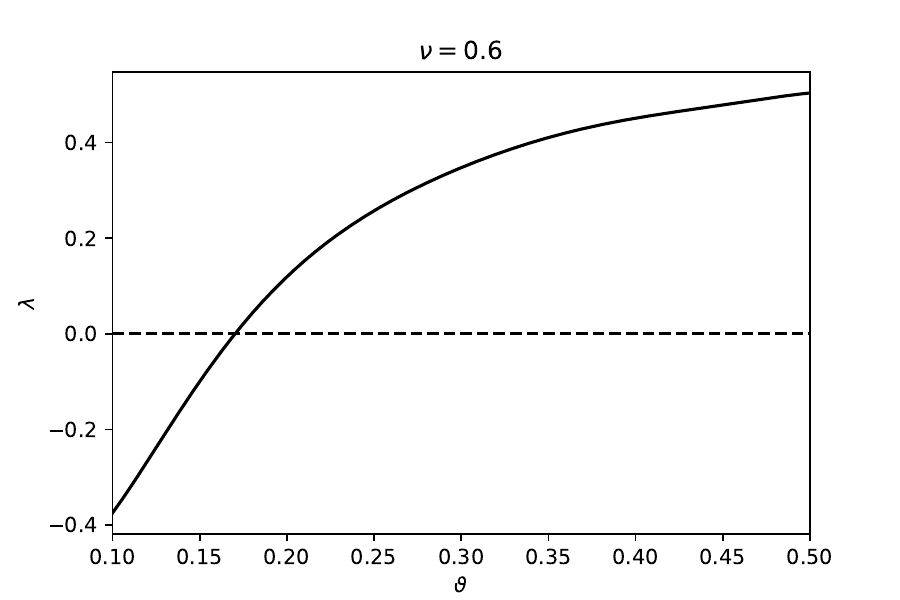} 
        \label{fig:bndlambda}
    }
    \quad
    \subfigure[Normalized distribution of distance between two orbits ($10^6$ iterations) with different initial conditions and the same noise realisation, consistent with Theorem~\ref{thm:main2} \eqref{i:main2:3}.]{
        \includegraphics[width=0.45\textwidth]{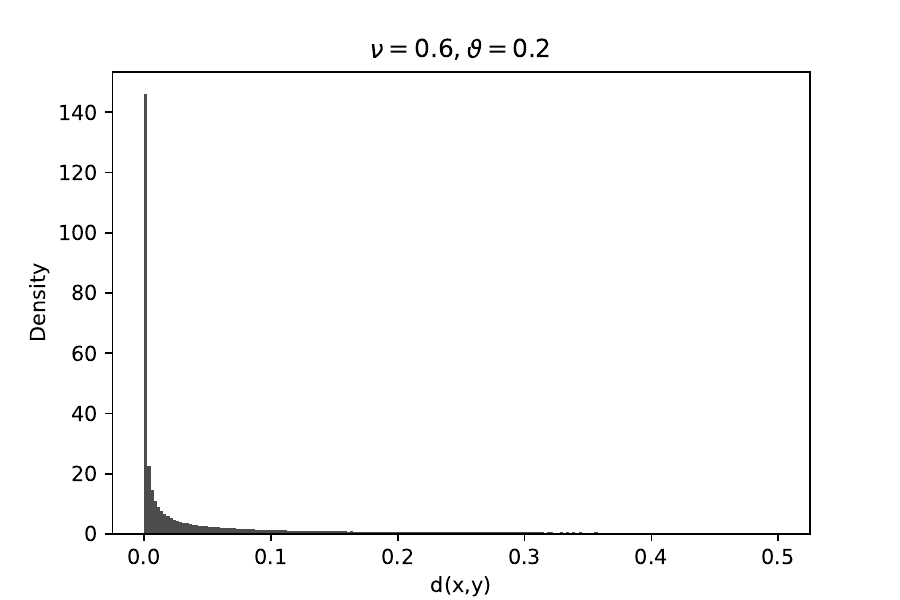} 
        \label{fig:distrL>}
    }

    \vspace{1em}

    \subfigure[Example of a time series of the distance between two orbits for $\vartheta = 0.17$, when $\lambda\approx 0$, close to intermittency. ]{
        \includegraphics[width=0.45\textwidth]{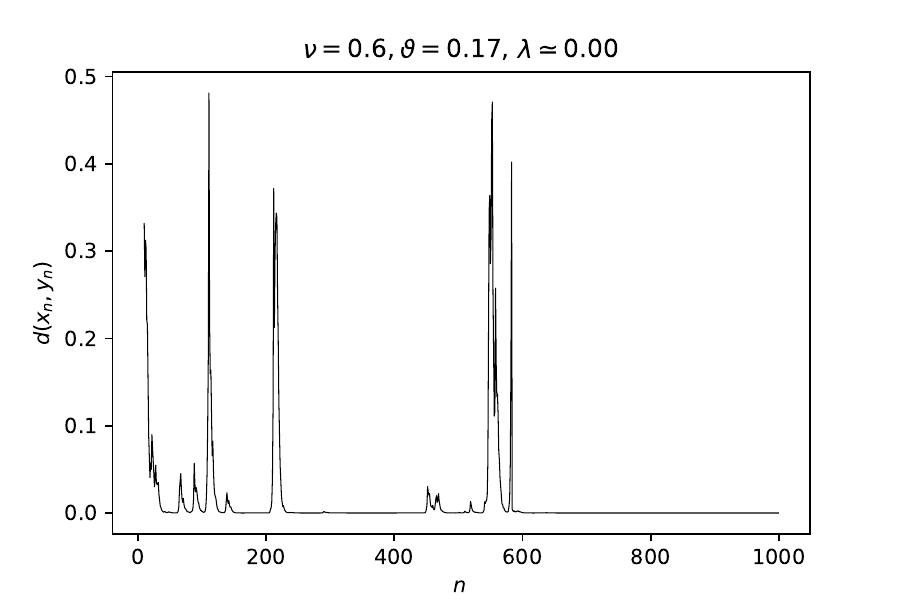} 
        \label{fig:timeL0}
    }
    \quad
    \subfigure[Example of a time series of the distance between two orbits for $\vartheta = 0.2$, when $\lambda\approx 0.11$, in the chaotic regime.]{
        \includegraphics[width=0.45\textwidth]{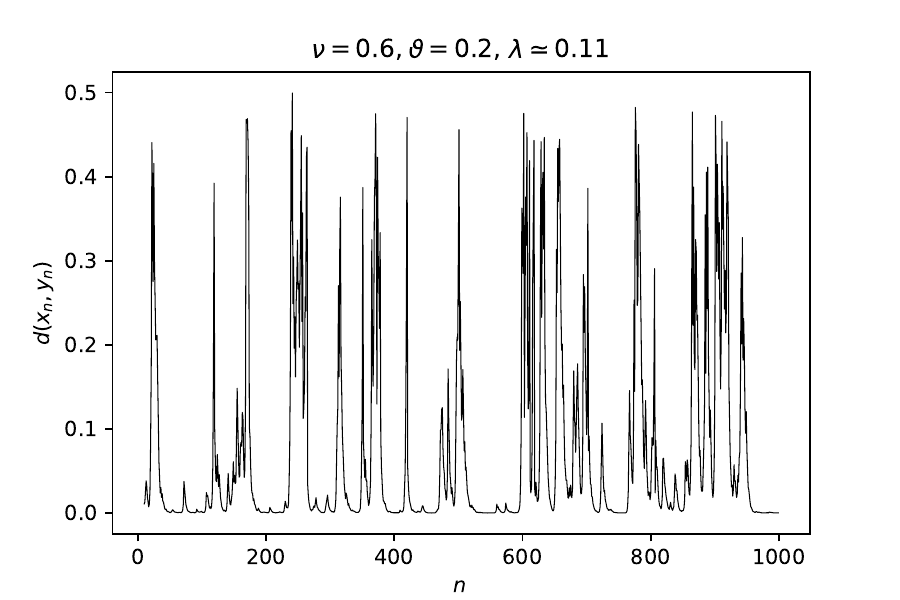}   
        \label{fig:timeL>}
    }
    \caption{Illustration of aspects of the dynamics in Example~\ref{ex:T_nu} regarding the family of random maps $T_{\nu,\omega}$ with $\nu = 0.6$.}
    \label{fig:all_subfigs}
\end{figure}

The dynamics of the two-point motions is illustrated by examples of time series for two different values of
\(\vartheta\) in Figure~\ref{fig:timeL0} ($\theta=0.17$, close to intermittency) and Figure~\ref{fig:timeL>} ($\theta=0.2$, chaos), displaying qualitative differences consistent with 
Theorem~\ref{thm:main1}.

Figure~\ref{fig:distrL>} shows the distribution of the two-point distance for an example time-series, displaying inverse-power-law behavior near zero, as asserted in
Theorem~\ref{thm:main2}. 

Rather than varying the noise level for fixed parameter $\nu$, one may also consider varying $\nu$ at fixed noise level. Fixing the noise level at $\theta=0.5$ (full noise), yields one-point motion orbits to be random i.i.d. sample from the stationary measure \(\mu = \left(T_\nu\right)_* (\mathrm{Leb})\), see Figure~\ref{fig:statmeas>}.

We include this example as it highlights the difference between the one- and two-point motion. In particular, at full noise, the one-point dynamics is essentially a full shift while the dynamics may be synchronising or chaotic.
Indeed, varying $\nu$ between $0.025$ and $0.2$ we observe (in Figure~\ref{fig:whitelambda} ) a monotonically increasing Lyapunov exponent going from negative (synchronisation) to positive (chaos).



\begin{figure}[ht!]
    \centering
    
 \subfigure[Numerical approximation of the Lyapunov exponent for fixed $\vartheta = 0.5$ and varying $\nu$. The Lyapunov exponent increases monotonically with $\nu$, from negative to positive.]{
        \includegraphics[width=0.45\textwidth]{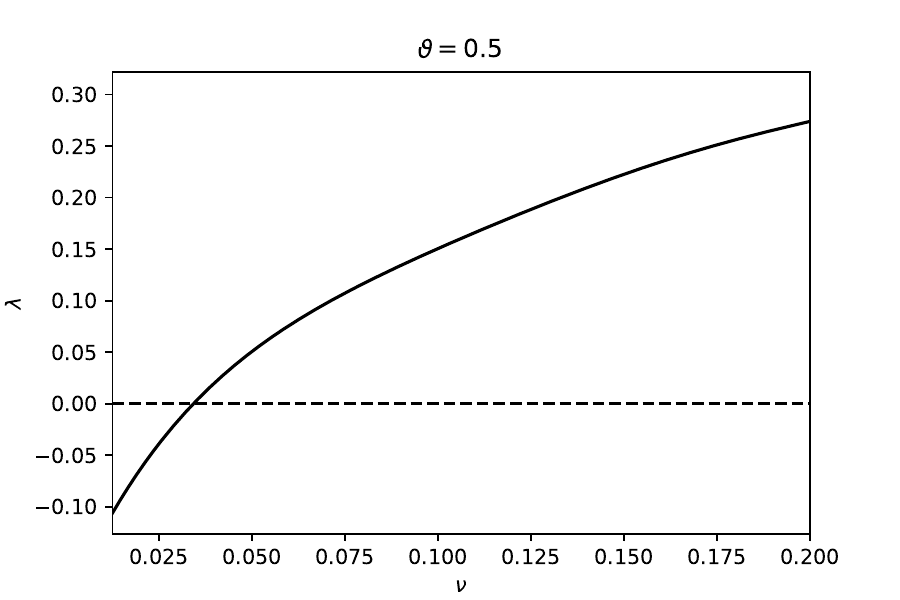} 
        \label{fig:whitelambda}
    }
    \quad
    \subfigure[Numerical approximation of the stationary density, with $\nu=0.6$ and $\vartheta = 0.5$.]{
        \includegraphics[width=0.45\textwidth]{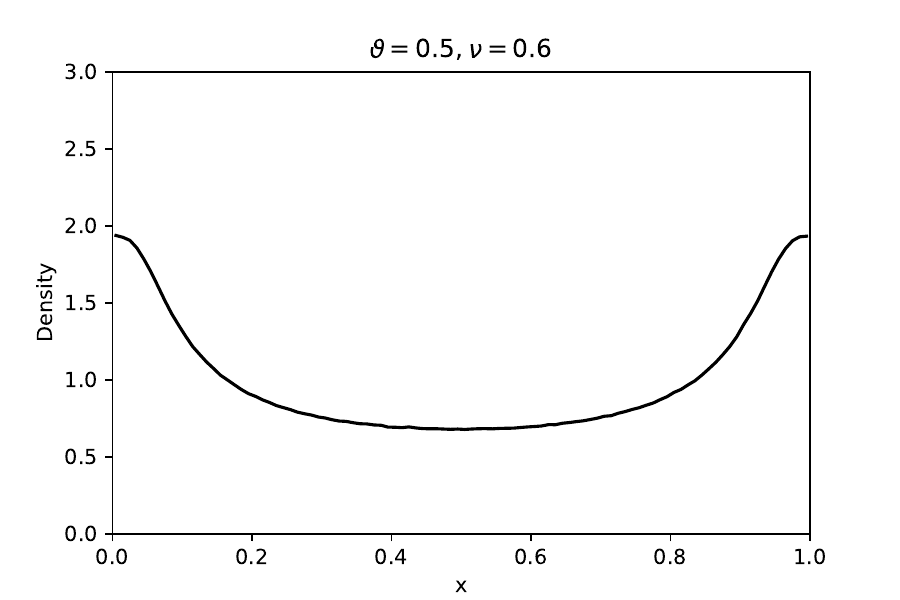} 
        \label{fig:statmeas>}
    }

\caption{Illustration of aspects of the dynamics of Example~\ref{ex:T_nu} regarding the family of random maps $T_{\nu,\omega}$, with $\vartheta = 0.5$ and varying $\nu$. }
\end{figure}
We revisit the full noise case in the context of this example,
in Example \ref{ex:white-noise-cont1} and Example \ref{ex:white-noise-cont2},
where this assumption leads to simplified statements.

\end{example}

\subsection{Context}

In this paper we have established a link between intermittency (of the two-point motion) and the transition to chaos in random dynamical systems. In the deterministic setting, intermittency has been recognized as a hallmark of only one of several routes to chaos \cite{MR576270}.


{ The link between intermittency and infinite ergodic measures has been previously observed for the one-point motion, concerning the behaviour of single trajectories.
First, this was established for deterministic dynamics near non-hyperbolic fixed points \cite{MR576460,MR1695915,MR599464}. Later, this was extended to intermittent dynamics near unstable invariant manifolds, also know as on-off intermittency  \cite{PhysRevE.49.1140,PhysRevLett.70.279}.
In random maps on bounded state spaces with common neutral fixed points, intermittency and infinite ergodic measures have been established in the presence of neutral fixed points  \cite{BB16,MR4024530} and on-average neutral fixed points \cite{MR722776,MR4706371}. 
See \cite{MR3600645,MR3826118,MR2033186}
for case studies of random logistic maps with zero Lyapunov exponents. 
Random dynamical systems on unbounded state spaces with zero Lyapunov exponent may also yield infinite stationary measures \cite{MR1428518,MR4434534,MR3098067,MR2954260,MR3342668}. 
For specific results on synchronisation on average (related to our notion of intermittency) in random maps, see also \cite{MR3519583,MR4142337}.


While the majority of research in ergodic theory of dynamical systems focuses on one-point motions, two-point motions are central to this paper. Crucially, our results on intermittency and infinite ergodic measures of the two-point motion of random maps with zero Lyapunov exponents do not imply intermittency and infinite ergodic measures for the one-point motion. In particular, the one-point motion of the random maps discussed in this paper do not display intermittency nor infinite ergodic measures.
}

The importance of two-point motions, or more general $n$-point motions,  has been recognized in the study of stochastic differential equations, see for instance \cite{MR1070361}, and in particular by Ledrappier \& Young \cite{MR968818,MR953818} and Baxendale \& Stroock \cite{MR968817,MR1144097} focusing on different Lyapunov exponent regimes.
More recently, their importance has been recognized in stochastic fluid dynamics \cite{MR4242626,MR4385127,MR4415776,MR4404792,PierreHumbert}.

Synchronisation in random dynamical systems constitutes the most elementary type of dynamical behaviour, and has been studied extensively  \cite{MR3663624,MR3862799,MR4228028}. In particular, it has been established in the negative Lyapunov exponent regime that synchronisation is equivalent to a mild contractability condition \cite{MR3820004}. Specifically, random circle homeomorphism have been discussed in \cite{MR756386,MR2117507,MR2425065,MR3719548}.

{ 
There is also an interest in iterated function systems \cite{MR977274} where the random parameter is chosen from a finite set, see also  \cite{ https://doi.org/10.48550/arxiv.2207.09987} for a result in the direction of this paper.
Our methods do not cover this setting, as we use absolute continuity of the noise to obtain spectral gap for a Koopman operator on a space of continuous functions. A possible strategy for    
discrete noise models might be to consider smaller function spaces. 


There are various problems beyond this paper that deserve further attention.
First of all, we expect our results to extend to higher-dimensional settings and allow for critical points (see \cite{galatolo2025efficientcomputationstationarymeasures} for a recent numerical study in this direction).
Furthermore, it is natural to study parametrized families of random dynamical systems and consider the dependence on the parameter of, in particular,  Lyapunov moment function and stationary measures of the two-point motion near zero Lyapunov exponent (see \cite{chalhoub2025analyticdependencelyapunovmoment} for a recent study in the case of random matrix products).
Finally, it would be of interest to obtain finer statistical information about the two-point dynamics in the presence of infinite stationary measures.}

\subsection{Organization of the paper}
This paper is divided into six sections. In Section~\ref{s:setting}, we discuss the underlying hypotheses of this paper and provide a summary of the strategy of the proof. In Section~\ref{s:koopman}, the technical machinery involving Koopman operators is developed, which is required in the subsequent sections. Section~\ref{s:trd} is dedicated to the proof of Theorem~\ref{thm:main1}, where the three different scenarios with zero, positive, and negative Lyapunov exponent are treated, respectively in Subsections~\ref{s:zero}, \ref{s:pos}, and \ref{s:neg}. The construction of stationary measures and the derivation of their properties, leading to Theorem~\ref{thm:main2}, are contained in Section~\ref{s:stat}.

\section{Setting and strategy of the proof}\label{s:setting}
This section introduces the hypotheses on the system and includes a summary of the strategy of the proof. 
We start with a circle endomorphism $T: \T\to \T$ with the following property.
    \vspace{0.5cm}
	\begin{hypothesis}\label{h:endo}
		$T$ is a circle endomorphism in $\mathcal{C}^2(\mathbb T)$, of degree two with derivative positively bounded from above and below; there exist positive real numbers $ a_1< a_2$, such that \[a_1<DT (x)<a_2\] for all $x \in \mathbb T$.
	\end{hypothesis}
    \vspace{0.5cm}
    Note that every point has exactly two inverse images under $T$. Degree two is not essential in this paper, higher degree works the same with minor modifications.
Define
\begin{equation}\label{eq:Rmin}
    R_{\min} \coloneqq \max \left\{ r \; ; \;  a_1 d(x,y) \leq d\left(T(x),T(y)\right) \leq a_2 d(x,y)  
 \text{ for all } x,y \in \T \text{ with } d(x,y)\leq r
\right\}.
\end{equation}
In particular $T(x) \neq T(y)$ if $d(x,y) \leq R_{\min}$.

	For $\omega \in \mathbb{T}$ write 
	\begin{align*} 
		T_\omega (x) &= T(x + \omega \pmod 1) 
	\end{align*}
	for the circle endomorphism obtained by composing $T$ with a translation.
	We  take a random process $\omega_i$, $i \in \mathbb{N}$, where the $\omega_i$ are independently drawn from a uniform distribution on 
 \[
\Omega_\vartheta \coloneqq [-\vartheta,\vartheta].
\] 
	This yields a random dynamical system
	\begin{align}\label{eq:rds}
		x_{n+1} &= T_{\omega_n} (x_n) 
	\end{align}
	for an initial point $x_0 \in \mathbb{T}$. 
    
	Let the sequence space 
\[
\Sigma_\vartheta \coloneqq \Omega_\vartheta^\mathbb{N}
\]
 be endowed with the product topology.
	We write $\omega  = (\omega_i)_{i\in\mathbb{N}}$ for points in ${\Sigma_\vartheta}$.
	Cylinders are sets 
	$[A_0,A_1\ldots,A_k] = \{ \omega \in {\Sigma_\vartheta} \; ; \; \omega_i \in A_i, 0\leq i \leq k\}$
	for Borel sets $A_i \subset \T$. Cylinders are the basis for the product topology.   
	Given the normalized Lebesgue measure $\mathrm{Leb}/(2\vartheta)$ on $\Omega_\vartheta$, write $\mathbb{P}$ for the corresponding product measure on ${\Sigma_\vartheta}$. 
For a function $X$ on $\Sigma_\vartheta$ we use common notation such as
\begin{align*}
\mathbb{E} \left[X \right] &\coloneqq \int_{\Sigma_\vartheta}  X(\omega)\, d\mathbb{P}(\omega).
\end{align*}
By identifying $\Omega_\vartheta$ with the cylinder $[\Omega_\vartheta]$ we can also use $\mathbb{P}$ for normalized Lebesgue measure $\mathrm{Leb}/(2\vartheta)$ on $\Omega_\vartheta$.
So, with a slight abuse of notation,
if a function $\omega \mapsto X(\omega)$ depends on a single symbol $\omega \in \Omega_\vartheta$, we write
$\mathbb{E} \left[X \right] = \int_{\Omega_\vartheta} X(\omega)\, d\mathbb{P}(\omega)$.
Recall that a measure $\mu$ on $\mathbb{T}$ is called a stationary measure if
\[
\int_{\Sigma_\vartheta}  \mu \left( \left( T_\omega\right)^{-1} (A) \right) \, d\mathbb{P} (\omega) =  \mu(A),
\]
for any (Borel) measurable set $A \subset \T$.

	The skew product map
	$\Theta: {\Sigma_\vartheta} \times \mathbb{T} \to  {\Sigma_\vartheta} \times \mathbb{T}$ 
	is defined by
	\[
	\Theta(\omega,x) \coloneqq (\sigma \omega , T_{\omega_0} (x)).
	\]
	Here $\sigma$ is the left shift operator $\sigma \omega \coloneqq (\omega_{i+1})_{i\in\mathbb{N}}$.
	With a slight abuse of notation we write
	\[
	T^n_\omega (x) \coloneqq T_{\omega_{n-1}} \circ \cdots \circ T_{\omega_0} (x)
	\]
	for iterates.
    
We compares two different trajectories by studying the random dynamical system. 
For $\omega \in \Sigma_\vartheta$, the two-point map $(x,y) \mapsto T^{(2)}_\omega (x,y)$ on $\mathbb{T}^2$ is the product
\[
(x,y) \mapsto (T_\omega (x) , T_\omega(y)). 
\]
	This yields the random dynamical system
	\begin{align}\label{eq:rds2}
		(x_{n+1},y_{n+1}) &=   T^{(2)}_{\omega_n} (x_n,y_n).
	\end{align}
	The two-point skew product map
	$\Theta^{(2)} : {\Sigma_\vartheta} \times \mathbb{T}^2 \to   {\Sigma_\vartheta} \times \mathbb{T}^2$
	is denoted by
	\[
	\Theta^{(2)} (\omega,x,y) = (\sigma\omega , T^{(2)}_\omega (x,y)).
	\]
A measure $\mu^{(2)}$ on $\mathbb{T}^2$ is a stationary measure of the random dynamical system $T^{(2)}_\omega$ on $\T^2$ if
\[
\int_{\Sigma_\vartheta}  \mu^{(2)} \left( \left( T^{(2)}_\omega\right)^{-1} (A) \right) \, d\mathbb{P} (\omega) =  \mu^{(2)}(A),
\]
for any (Borel) measurable set $A \subset \T^2$.

\subsection{Hypotheses}

We focus on random circle endomorphisms whose trajectories are not confined to subintervals
of the circle but spread over the entire circle.
 
	\vspace{0.5cm}
	\begin{hypothesis}\label{h:minorisation}
There is $k>0$ so that for any $x,y \in \mathbb{T}$, there is $\omega \in {\Sigma_\vartheta}$ so that
$T^k_\omega (x) = y$.
	\end{hypothesis}
		\vspace{0.5cm}

{ This hypothesis  in combination with absolute continuous noise guarantees the existence of a unique absolutely continuous stationary measure of full support for the one-point motion, but also has further applications that are used throughout the paper.}


\begin{proposition}\label{prop:uniquestatmeasure}
Suppose the random dynamical system described by \eqref{eq:rds}
with $\omega_n$ i.i.d. picked from a uniform distribution for $[-\vartheta,\vartheta]$, 
adheres to Hypotheses~\ref{h:endo}, \ref{h:minorisation}.

Then the random dynamical system admits an absolutely continuous stationary measure $\mu$ with full support and smooth density.
\end{proposition}
The proof is omitted and be obtained from standard arguments using Perron-Frobenius operators.  A similar setup is in \cite{MR1233850}, to which we refer for details.



%

The following assumption provides a local source for contraction.
This is needed to achieve a negative or zero Lyapunov exponent.

\vspace{0.5cm}
\begin{hypothesis}\label{h:periodic}
	The map $T$ has a hyperbolic attracting periodic orbit.
\end{hypothesis}
\vspace{0.5cm}

We must
avoid that points in $\T^2 \setminus \Delta$ 
are mapped into $\Delta$ by $T^{(2)}_a$ with positive probability, and more generally we must
control and bound the probability that points are mapped very close to the diagonal by the two-point maps $T^{(2)}_a$. 
The following hypothesis provides us with an assumption to prevent this.

	\vspace{0.5cm}
	\begin{hypothesis}\label{h:genericnoise}
For every $(x,y) \in \mathbb{T}^2\setminus \Delta$, 
the curve $a \mapsto (T_a (x), T_a (y))$ in $\T^2$ intersects the diagonal $\Delta$ transversely or with
at most a single quadratic tangency.
	\end{hypothesis}
	\vspace{0.5cm}

{ {
	Hypothesis \ref{h:genericnoise} is a generic condition.
		For convenience we formulate the following lemma with the parameter $a$ in $T_a$ from the whole circle $\mathbb{T}$; taking $a$ from a larger interval only shrinks the set of circle endomorphisms that satisfy Hypothesis~\ref{h:genericnoise}.
		
		\begin{lemma}\label{lem:H4isgeneric}
			There is an open and dense set of degree two increasing circle endomorphisms $T: \mathbb{T} \to \mathbb{T}$ in the $\mathcal C^2$ topology so that the corresponding random system
			$T_a$ with $a \in \mathbb{T}$  satisfies Hypothesis~\ref{h:genericnoise}. 
		\end{lemma}
		
		\begin{proof}
			To study intersections of the curve $a \mapsto (T_a (x), T_a (y))$ with $\Delta$ we only need to consider nearby points $T(x+a), T(y+a)$.  We may thus write $T(x+a) - T(y+a) = 0$ instead of $d(T(x+a),T(y+a))=0$.
			Write  $T(x+a) = T(y+a)$ as $T(b+s) = T(b)$ with $b = x+a$ and $b+s = y+a$, that is, $s = y-x$.
			We are thus looking for zeros of the family of functions 
			\[
			f_s (b) \coloneqq T(b+s) - T(b),
			\] 
			with $b \in \mathbb{T}$.
			Note that zero's of $f_s$ for $b+s \neq b$  can only occur for $s$ away from 
			zero.
			The hypothesis can be rephrased as stating that the graph of the function $f_s$ has only transverse intersections with zero, or at most one quadratic tangency to the graph of the zero function for isolated values of $s$.
			As in classical bifurcation theory applied to bifurcations of equilibria in one-parameter families \cite{MR279827,MR339280,MR727698}, 
			for an open and dense set of degree two increasing circle endomorphisms, a zero is either transverse or with a quadratic tangency, and there is at most one zero with a quadratic tangency (a quadratic tangency corresponds to a saddle-node bifurcation, for generic families one finds at most one saddle-node bifurcation at a parameter value).
		\end{proof}
 }}       
 
The integrability statement in the following lemma is a consequence of Hypothesis~\ref{h:genericnoise}, and results in an estimate that is required in later arguments (Proposition~\ref{prop:chaos}, Proposition~\ref{prop:bounds} and Proposition~\ref{prop:probmeas}). 

\begin{lemma}\label{lem:logbnd}
    For all $R>0$ there exists $C_1>0$ with 
	\begin{align}\label{eq:resulthypo}
		\mathbb{E} \left[  -\ln(d(T_\omega^{(2)}(x,y))) \right] < C_1,
	\end{align}
	for all $(x,y) \in \T^2 \setminus \Delta_{R}$. 
\end{lemma}

\begin{proof}
 Let $w(x,y)$ denote the signed distance of $x,y\in\mathbb{T}$, defined as $w(x,y) =  \left( (y - x + \frac{1}{2}) \pmod1 \right) - \frac{1}{2}$. Consider the real valued function \[a \mapsto f(a) \coloneqq w\left( T^{(2)}_a (x,y) \right)\] on $\T$.
Hypothesis~\ref{h:genericnoise} implies that (see also the proof of Lemma~\ref{lem:H4isgeneric})
for $(x,y) \in \mathbb{T}^2 \setminus \Delta_\delta$, there exist $\delta>0, C>0$ so that 
\begin{align*} 
	\left| D^2 f (a) \right| &\geq C
\end{align*}
whenever $|f(a)| < \delta$ and $\left| Df(a) \right| < \delta$.

If $I \subset \T$ is an interval so that for $a \in I$, $|f(a)| < \delta$
and $\left| D f(a) \right| \geq \delta$,
then $\int_I -\ln  \left( |f(a)| \right)\, da$ is bounded by some $C_\delta >0$. 

To prove the lemma it remains to consider $f$ on an interval $J \subset \T$ so that $|f(a)| < 2 \delta$ for $a \in J$ and $J$ contains a point $d \in J$ with $Df(d) = 0$.
If $f(d) = 0$  and $Df(d) = 0$, we have a bound $0 < c (a-d)^2 < |f(a)|$ for some $c > 0$. The same bound applies if $Df(d) = 0$ and $f$ is never zero on $J$.
The possibility that is left is where $f$ has two zeros $f(d_1) = f(d_2) = 0$ for nearby points $d_1,d_2$ on different sides of $d$.
Then a bound $0 < c |(a-d_1) (a-d_2)| < |f(a)|$ for some $c > 0$ holds.
In all cases $c$ is uniformly bounded away from $0$.
As the logarithms of $c (a-d)^2$ and $ c |(a-d_1) (a-d_2)|$ are integrable, we find that
$\int_J -\ln  \left( |f(a)| \right)\, da$ is bounded by some $C_\delta >0$. 
Noting that $-\ln ( |f(a)| )$ is bounded if
$|f(a)| > \delta$, proves the lemma.  
\end{proof}

We make a final assumption that will be used to prove full support of stationary measures for the two-point motion. 
The random two-point system maps a point $(x,y)$ to a curve $a \mapsto (T_a (x),T_a(y))$ in $\mathbb{T}^2$. 
For the composition of two iterates, we have
\begin{align}
H_{a_0,a_1}(x,y) &\coloneqq \mathrm{det}\, \left( \frac{\partial}{\partial (a_0,a_1)}  T^{(2)}_{a_1} \circ T^{(2)}_{a_0}  (x,y)     \right)\nonumber
\\
&= DT_{a_1} ( T_{a_0}(x)) DT_{a_1} ( T_{a_0}(y)) \left( DT_{a_0} (x) - DT_{a_0} (y)  \right).\label{eq:H}
\end{align}
Consider now the following hypothesis.
\vspace{0.5cm}
	\begin{hypothesis}\label{h:openaftertwo}
		For each $(x,y) \in \mathbb{T}^2 \setminus \Delta$, there are $a_0,a_1 \in (-\vartheta,\vartheta)$ so that $H_{a_0,a_1}(x,y) \neq 0$.
	\end{hypothesis}
	\vspace{0.5cm}

This condition implies that for each $(x,y) \in \mathbb{T}^2 \setminus \Delta$, the image of 
\[
(a_0,a_1) \mapsto T^{(2)}_{a_1} \circ T^{(2)}_{a_0}  (x,y),
\] 
with $a_0,a_1 \in [-\vartheta,\vartheta]$, contains an open set.

Recall that a map $F: X \to X$ on a topological space $X$ is called topologically exact if for any open $U \subset X$, there is $n \in \mathbb{N}$ so that $F^n (U) = X$.

\begin{lemma}\label{lem:2pointexact}
The skew product systems  $\Theta: {\Sigma_\vartheta} \times \mathbb{T} \to  {\Sigma_\vartheta} \times \mathbb{T}$ and $\Theta^{(2)}: {\Sigma_\vartheta} \times \mathbb{T}^2 \to  {\Sigma_\vartheta} \times \mathbb{T}^2$ are topologically exact.
\end{lemma}

\begin{proof}
We first consider $\Theta$ and we start with the following statement:
for any open interval $I \subset \T$ there is $\omega \in \Sigma_\vartheta$ and $k\in\mathbb{N}$ so that
$T_\omega^k (I) = \T$. The reasoning will also show that there is a cylinder $D \subset \Sigma_\vartheta$ that contains $\omega$, so that $T_\zeta^k (I) = \T$ for any $\zeta \in D$.

Given the interval $I$,  take $x \in I$. 
Let $y \in \mathbb{T}$ be a periodic point in an uncountable transitive  hyperbolic repelling set $\Lambda$ of a map $T_a$ \cite{MR806250}. 
Write $k$ for the period of $y$.
{ {
We claim that given an open neighborhood $V$ of $y$, there is an iterate
$T^{kn}_a$ that maps $V$ onto all of $\mathbb{T}$. Suppose not. If $V$ is small then $T_a^k$ is a diffeomorphism between $V$ and $T_a^k(V)$. Keep iterating until either some iterate $T_a^{kn}$ no longer defines a diffeomorphism between $T_a^{k(n-1)} (V)$ and $T_a^{kn}(V)$ (then $T_a^{kn} (V)$ covers $\mathbb{T}$, or this never happens. In the latter case,  $\cup_n T^{kn}_a (V)$ is an interval and $T_a^k$ is a diffeomorphism on it. This is not possible since $y$ is accumulated by repelling periodic points of different periods.   
}}

By Hypothesis~\ref{h:minorisation} there is $\varsigma \in \Sigma_\vartheta$ and $m \in \mathbb{N}$ so that $T_\varsigma^m (x) = y$. Hence there is $n \in \mathbb{N}$ so that $T^{kn}_a \circ T^m_\varsigma$ maps $I$ onto $\mathbb{T}$. By continuity the same applies to $T^{m+kn}_\nu$ for $\nu$ in a suitable cylinder $D$
that contains  $\varsigma_0\cdots \varsigma_{m-1} a \cdots a$ (with $nk$ times $a$). This proves the statement with which we started.

To conclude the proof of topological exactness of $\Theta$, let $U\subset \Sigma_\vartheta \times \mathbb{T}$ be an open set.
By shrinking $U$ we may assume that $U$ is a  product set	$U = C \times J$ of a cylinder	$C$ and an open interval $J$. 
There is an iterate $\Theta^j (U)$ that contains an open set 
$\Sigma_\vartheta \times I$ for an open interval $I \subset \T$.
Now use the above statement to establish that a further iterate covers $\Sigma_\vartheta \times \T$.

Next we consider $\Theta^{(2)}$.
	Take an open set $U \subset \Sigma_\vartheta \times \mathbb{T}^2$. By shrinking $U$ we may assume that $U$ is a  product set
	$U = C \times I \times J$ of a cylinder
	$C$ and open intervals $I,J$. 

	By the above reasoning, there is $l \in \mathbb{N}$ and $\omega \in  C$  so that $\left(T^{(2)}_\omega\right)^l (I \times J)$ contains $\mathbb{T} \times K$ for an open interval $K \subset  \mathbb{T}$. Applying the above reasoning again, there is $l \in \mathbb{N}$ so that
	$\left(T^{(2)}_\omega\right)^l (I\times J)$ equals $\mathbb{T}^2$. The  proof of topological exactness of $\Theta^{(2)}$ is concluded as above.
\end{proof}

Hypotheses~\ref{h:periodic},~\ref{h:openaftertwo} and Lemma~\ref{lem:2pointexact} show that $\mathbb P$-almost surely  all orbits of the two point motion get arbitrarily close to the diagonal $\Delta$. The following lemma formalizes this.

\begin{lemma}    \label{lem:proximal}
For all $\varepsilon>0$, there exists a $N\in\N$, $C>0$ such that for all $(x,y) \in \T^2$, we have 
\[
\mathbb{P} \left(  \left\{  \omega \in \Sigma_\vartheta \; ; \;  d(T_\omega^i(x),T^i_\omega(y))<\varepsilon \text { for some } 0\leq i \leq N \right\}   \right) > C.
\]
\end{lemma}

\begin{proof}
Take $(x,y) \in \mathbb{T}^2 \setminus \Delta_\varepsilon$. By Hypothesis~\ref{h:openaftertwo} we find that
$\omega \mapsto \left(T^{(2)}_\omega\right)^2 (x,y)$ contains an open product set $I_x\times I_y \subset \T^2$. By compactness we find $\delta>0$ so that  $I_x$ and $I_y$ have diameter at least $\delta$, uniformly in $x,y \in    \mathbb{T}^2 \setminus \Delta_\varepsilon$.
The argument of Lemma~\ref{lem:2pointexact} shows that for any $(x,y) \in  \mathbb{T}^2 \setminus \Delta_\varepsilon$ there is an open neighborhood $U$ of $(x,y)$ and a positive integer $N$
so that    $\left(T^{(2)}_\omega \right)^N(x',y')$ covers $\T^2$ for any $(x',y') \in U$. 
By compactness $N$ is bounded uniformly 
in $(x,y) \in    \mathbb{T}^2 \setminus \Delta_\varepsilon$. This implies the lemma.
\end{proof}

\begin{example} \label{ex:white-noise-continued}
We continue from Example~\ref{ex:T_nu}, giving concise justifications that the required hypotheses hold. Hypotheses~\ref{h:endo} and \ref{h:periodic} are trivially satisfied. {  Hypothesis~\ref{h:minorisation} is obvious when $\vartheta = 0.5$, and can be checked via a simulation or some naive bounds. The main obstruction to Hypothesis~\ref{h:minorisation} would be 
that the noise is too small to let an orbit escape from a neighborhood of the stable fixed point.}

Regarding Hypothesis~\ref{h:genericnoise}, observe that there is exactly one point $y \in \T$ that has two preimages $x_1, x_2 \in \T$ with
\[
DT(x_1) = DT(x_2).
\]
At this point $y$, the second derivatives $DT(x_1)$ and $DT(x_2)$ 
have opposite signs.
The curvature of the curve $a \mapsto (T_a (x_1), T_a (x_2))$ at the point of tangency with $\Delta$
is therefore nonzero, and the tangency is quadratic.

Finally, by Equation \eqref{eq:H}, and the fact that for all $(x,y) \in \T^2\setminus \Delta$, there exists an $a_0 \in [-\vartheta,\vartheta]$ such that $DT_{a_0}(x) - DT_{a_0}(y) \neq 0$, Hypothesis~\ref{h:openaftertwo} follows.
  \hfill $\blacksquare$
 \end{example}

\subsection{Strategy of the proof}\label{s:strategy}

This section gives a helicopter view on the reasoning in Sections~\ref{s:koopman} and \ref{s:trd} that leads to Theorem~\ref{thm:main1} on dynamics of the random two-point maps.
For the study of orbits of the random two-point maps it is crucial to understand the duration of trajectories of two nearby points staying close to each other.
In terms of the two-point motion this is the duration of trajectories of the two-point motion staying close to the diagonal.
Suppose $x_0,y_0 \in \mathbb{T}$ are close, so that $d_0 = d(x_0,y_0)$ is close to zero. With $x_n = T^n_\omega (x_0)$, $y_n = T^n_\omega (y_0)$ and $d_n = d(x_n,y_n)$, we find that as long as $d_n$ is small,
\begin{align}\label{eq:approxDT}
d_{n+1} &\approx DT_{\omega_n} (x_n) d_n.
\end{align}
Taking logarithms $u_n = \ln (d_n)$ this reads
\begin{align}\label{eq:approxlnDT}
u_{n+1} &\approx u_n + \ln \left(DT_{\omega_n}(x_n)  \right),
\end{align}
which is a random walk driven by the one-point motion $x_{n+1} = T_{\omega_n} (x_n)$.

In the special case where $x_n$ is identically and independently distributed (see Example~\ref{ex:T_nu}, with $\vartheta = 0.5$), 
the approximation is a random walk with i.i.d. steps. But this is not true in general.
We derive estimates for the duration of passages near the diagonal by adapting reasoning 
in \cite{MR968817,MR1144097} for continuous time settings to our discrete time setting. We base our analysis on properties of the Koopman operator for the two-point system.

For the one-point motion, the annealed Koopman operator $P$
acting on a real valued function $\phi$ on $\T$ 
is defined as
\begin{align*}
	P \phi(x) &= \mathbb{E} \left[ \phi(T_\omega (x) ) \right].
\end{align*}
Here $\mathbb{E}$ stands for an expectation over $\omega \in \Sigma_\vartheta$.
Analogously the annealed two-point Koopman operator $P_{(2)}$ 
acting on  a real valued function $\phi$ on $\T^2\setminus\Delta$
is defined as
\[
P_{(2)}\phi (x,y) = \mathbb{E} \left[\phi(T_\omega^{(2)}(x,y))\right].
\]

To estimate stopping times for trajectories of the two-point motion from strips $\Delta_\delta \setminus \Delta_\varepsilon$ near the diagonal (so with $0 < \varepsilon < \delta$ small) we construct sub- and supermartingales for the stopped dynamics. The key statements are Proposition~\ref{prop:phieta} and Proposition~\ref{prop:W} below. 
The constructions rely on an analysis of $P_{(2)}$, which in turn is facilitated by the approximations \eqref{eq:approxDT} and \eqref{eq:approxlnDT} and a study of the corresponding linearized Koopman operator, 	defined for continuous real valued functions  $\phi$ on ${\T\times\R^+}$ by
\[
TP \phi(x,u) = \mathbb{E} \left[\phi(T_\omega(x), DT_\omega (x)u)\right].
\]
Formulas for the mentioned sub- and supermartingales are obtained from a study of the twisted Koopman operator  ${P}_q$ whose action on a real valued function $\psi$ on $\mathbb{T}$
is defined as  
\[
		{P}_q \psi(x) = \mathbb{E}
		\left[\frac{\psi\left(T_\omega(x)\right)}{(D T_\omega (x))^q}\right].
\]
This in turn connects to  the moment Lyapunov exponent
\begin{equation*} 
	\Lambda(q) = \lim_{n\to\infty} \frac{1}{n}\ln  \left(\mathbb{E} \left[ D T_\omega^n (x)^q \right]\right).  
\end{equation*}
The moment Lyapunov exponent  plays a central role in our analysis, similar to \cite{MR968817,MR1144097}.
Originally introduced by Molchanov  \cite{molcanov_structure_1978} and subsequently developed in particular by Arnold \cite{ArnoldStabFormula}, it can be viewed as a generalization of the Lyapunov exponent \cite{MR805125}. The relationship between the moment Lyapunov exponent and the Lyapunov exponent was first described in \cite{MR800244}. 

{ 
The above comments on methodology refer to a study of the two-point motion near the diagonal.  Due to the non-injectivity of $T$, trajectories of distinct points $(x,y) \in \mathbb{T}^2 \setminus \Delta$ may land on the diagonal $\Delta$ in finite time. 
Our analysis circumvents this issue in two ways. 
First, the local estimates for the Koopman operators are restricted to strips $\Delta_{R} \setminus \Delta$ where $R < R_{min}$. Within this radius, the map acts injectively, ensuring that distinct points no further than $R$ apart do not coalesce in a single iterate, allowing for well-defined expansion and contraction estimates. 
Second, regarding the global dynamics and measures construction in Section \ref{s:stat}, the set of points that map directly onto $\Delta$ (and thus leave $\mathbb{T}^2 \setminus \Delta$) is controlled by Hypothesis~\ref{h:genericnoise}. This ensures that the probability of landing on or close to the diagonal from a large distance is negligible for our growth rate estimates. 
}

\section{Koopman operators}\label{s:koopman}

This section develops theory of Koopman operators needed for our analysis on random dynamics.
The contents of this section are crucial for the arguments in the following sections, although the statements in the lemmas and propositions in the following sections can be read without reference to this section. 

Below we define in particular the annealed Koopman operator and the linearized Koopman operator for the random one-point maps, and the annealed Koopman operator for the random two-point maps. 
{  What we refer to as the annealed Koopman operator is also known, especially in the stochastics literature, as the Markov semigroup.}
We consider the twisted Koopman operator and the moment Lyapunov exponent and use it to obtain eigenfunctions for the linearized Koopman operator.
The approximation of the two-point random maps applied to nearby points by the linearized random map
allows us to obtain, from these eigenfunctions, functions on which the two-point Koopman operator acts in a
desired way: these functions appear in the construction of sub- and supermartingales for the random two-point maps considered near the diagonal.
 
\subsection{Twisted Koopman operator and moment Lyapunov function}\label{ss:twisted}

Write $\mathcal{C}(\mathbb{T},\mathbb{R})$ for the space of real valued continuous functions on $\mathbb{T}$.  Denote the $\mathcal C^0$-norm as $\|\cdot\|$.
The \textit{(annealed) Koopman operator} $P : \mathcal{C}(\mathbb{T},\mathbb{R})\to \mathcal{C}(\mathbb{T},\mathbb{R})$
is defined as
\begin{align*}
	P \phi(x) &\coloneqq \mathbb{E} \left[ \phi(T_\omega (x) ) \right].
\end{align*}

The \textit{twisted Koopman operator},  ${P}_q: \mathcal{C}(\T,\mathbb R) \to\mathcal{C}(\T,\mathbb R)$ with $q \in \mathbb R$, is defined for  continuous functions $\psi: \T \to \mathbb R$ by
\[
{P}_q \psi(x) \coloneqq \mathbb{E}
\left[\frac{\psi\left(T_\omega(x)\right)}{(D T_\omega (x))^q}\right].
\]
Note that  $P = {P}_0$. 

The following lemma establishes basic properties of $P_q$.
Similar statements in other settings are in \cite{StayHomble,ArnoldStabFormula}.

\begin{lemma}\label{lem:Pqcompact}
	For $q \in \mathbb{R}$, the operator $P_q$  on $\mathcal{C}(\mathbb{T},\mathbb{R})$ is { strictly positive} and compact.
\end{lemma}

\begin{proof}
	It is clear that $P_q \psi \geq 0$ if $\psi \geq 0$, that is, that $P_q$ is positive. 
	Hypothesis~\ref{h:minorisation} yields the following property.
	There exists $n \in \mathbb{N}$, so that
	for any $\psi\in \mathcal{C}(\mathbb{T},\mathbb{R})$ with $\psi \geq 0$ and $\psi \ne 0$,
	$P_q^n \psi > 0$ everywhere.
	This means that $P_q$ is { strictly positive}.
	
	Let $\mathcal{B}$ be the unit ball in  $\mathcal{C}(\mathbb{T},\mathbb{R})$.
	For $\psi \in \mathcal{B}$,  note that (for readability skipping the $\pmod 1$ from the arguments) 
	\[
	P_q \psi (x) = \frac{1}{2\vartheta} \int_{-\vartheta}^\vartheta  D T(x+\omega)^{-q} \psi (T(x+\omega))\, d\omega = \frac{1}{2\vartheta} \int_{x-\vartheta}^{x+\vartheta} DT (y)^{-q} \psi (T(y))\, dy
	\]
	is a continuously differentiable function of $x \in \mathbb{T}$.
	 As $\| P_q \psi\| \leq C$ for a constant $C$ (for $q>0$ we can take $C = a_2^{q}/\vartheta$, for $q<0$ we can take $C =  a_1^{q}/\vartheta$ with $a_1,a_2$ from Hypothesis~\ref{h:endo}), $P_q$ is a bounded operator.
	Now
	\begin{align*}
		\left| D P_q \psi (x)  \right| &= \frac{1}{2\vartheta} \left|  DT (x+\vartheta)^{-q} \psi (T(x+\vartheta)) - DT (x-\vartheta)^{-q} \psi (T(x-\vartheta))  \right|
		\\
		&\leq C \|\psi\|
	\end{align*}
	for the same constant $C$.
	It follows that $P_q \psi$ for $\psi \in \mathcal{B}$ is an equicontinuous family of functions. 
	By the Arzela–Ascoli theorem we get that $P_q$ is a compact operator.
\end{proof}

Using this lemma we get the following results on the dominant eigenvalue and corresponding eigenvector 
of $P_q$.

We make use of the $q$th moment Lyapunov exponent $\Lambda(q)$ (see \cite{ArnoldStabFormula,StayHomble,MR1671091})
defined as
\begin{equation}\label{momlf}
	\Lambda(q) \coloneqq \lim_{n\to\infty} \frac{1}{n}\ln  \left(\mathbb{E} \left[ D T_\omega^n (x)^q \right]\right).  
\end{equation}
We also write moment Lyapunov function, especially when discussing its dependence on $q$.
We will find that the limit
does not depend on $x$, and that it exists as an analytic function of $q$.
Denote $\mathbbm{1}_\mathbb{T}:\T\to \R$ as the constant function equal to 1 on the circle.

\begin{proposition}\label{prop:specnew}
	For $q \in \mathbb{R}$, ${P}_{q}$ has a dominant simple eigenvalue $e^{\Lambda(-q)}$.
	The rest of the spectrum of $P_q$ is contained in a disk of radius less then $e^{\Lambda(-q)}$.
	Write $\phi_q \in \mathcal{C} (\T,\mathbb R)$ for the dominant  eigenfunction of ${P}_{q}$, so
	\begin{equation}
		{P}_{q}\phi_{q} = e^{\Lambda(-q)} \phi_{q}.
		\label{eq:twistedeig}
	\end{equation}
	Then 
	\begin{enumerate}
		\item   $\phi_0 = \mathbbm{1}_\mathbb{T}$ and $e^{\Lambda (0)} = 1$,
		\item   $\phi_q$ is  a positive function,
		\item   $\Lambda(q)$ and $\phi_{q}$ depend analytically on $q$,
		\item   $\Lambda$ is convex,
            \item $\Lambda(q) \geq \lambda q$.

	\end{enumerate}				
\end{proposition}

\begin{proof} 
	The proof follows ideas  as in \cite{StayHomble,ArnoldStabFormula}.
	The spectral properties of $P_q$ follow from Lemma~\ref{lem:Pqcompact} and the Krein-Rutman theorem (see for instance \cite[Section~19.5]{MR787404}).  
	Write $r(q) = \sigma\left(   P_{q} \right)$ for the spectral radius of $P_q$. By the Krein-Rutman theorem, this equals the dominant eigenvalue
	of $P_q$. {  The dominant eigenvalue is simple, because the operator is strictly positive.}
	We have 
	\[ P^n_q\mathbbm{1}_\mathbb{T} (x) =  \langle k_q , \mathbbm{1}_\mathbb{T} \rangle r(q)^n \phi_q (x)  + o(r(q)^n),
	\]
	as $n\to\infty$, uniformly in $x$, where $k_q\in \mathcal C(\T,\R)^* $ is a probability measure, see \cite{ArnoldStabFormula}. Now we rescale $\phi_q$, such that $\langle k_q , \mathbbm{1}_\mathbb{T} \rangle =1$ and using this, calculate
	\begin{align*}
		\lim_{n\to\infty} \frac{1}{n} \ln \left( \mathbb{E} \left[ \left(DT_\omega^n (x) \right)^{-q}  \right] \right)
		&= \lim_{n\to\infty}\frac{1}{n} \ln \left(P^n_q\mathbbm{1}_\mathbb{T} (x) \right)
		\\
		&=  \lim_{n\to\infty}\ln \left(  \left(P^n_q\mathbbm{1}_\mathbb{T} (x) \right)^{1/n}  \right)
		\\
		&= \lim_{n\to\infty} \ln \left( \left(r(q)^n \phi_q (x)  + o(r(q)^n) \right)^{1/n} \right)
		\\
		&=  \ln (r(q)).
	\end{align*}
	We find that $\lim_{n\to\infty} \frac{1}{n} \ln \left(\mathbb{E} \left[ \left(DT_\omega^n (x) \right)^{-q}  \right]\right)$ does not depend on $x$.
	The remaining properties of $\Lambda(-q)$ and $\phi_q$ follow from \cite{StayHomble,ArnoldStabFormula}.
\end{proof}

The following lemma connects the moment Lyapunov exponent and the Lyapunov exponent.

\begin{lemma}\label{lem:momentderivatives}
	The first derivative of the moment Lyapunov function at $q=0$ is equal to the Lyapunov exponent
	\begin{align}\label{eq:Lambda'(0)}
		\Lambda' (0) &= \lim_{n\to\infty} \frac{1}{n} \mathbb{E} \left[ \ln  \left( D T_\omega^n (x)  \right) \right]  = \lambda.
	\end{align}
\end{lemma}

\begin{proof}
    This follows from the fact that $\Lambda(0) =0$, $\Lambda(q) \geq \lambda q$ and the analyticity of $\Lambda$.
    For the complete argument, see  \cite{ArnoldStabFormula,StayHomble}.
\end{proof}


	
	

We have established that $\Lambda$ is a  convex function that vanishes at $0$. The following lemma shows the existence of a second zero if $\lambda \neq 0$. This second zero plays a 
prominent role in our analysis and also appears in the statements of the main theorem on stationary measures.

\begin{lemma}\label{lem:second0}
	    If  $\lambda \neq 0$ then there is a unique $\gamma \neq 0$, with opposite sign, such that $\Lambda(\gamma) = 0$. (We set $\gamma = 0$ if $\lambda = 0$.) 
\end{lemma}

\begin{proof}
{ {
			Assume that $\lambda > 0$, which implies by \eqref{eq:Lambda'(0)} that $\Lambda'(0) > 0$.	
		By Hypothesis~\ref{h:periodic} there is a contracting periodic point $x_c$ of $T$. Let $k_c$ denote its period.
		By continuity of the map, there exist $\varepsilon > 0$ and $\delta > 0$ such that the set
		\[
		\mathcal{H}^c = \big\{ \omega \in \Sigma_\vartheta \; ; \; x \in B_\varepsilon(x_c) \text{ implies }   T^{k_c}_\omega (x) \in B_\varepsilon (x_c),\,
	 \ln(DT_\omega^{k_c}(x)) < -\delta \big\}
		\]
		has positive measure $\mathbb{P} (	\mathcal{H}^c ) > 0$.
		
		By Hypothesis \ref{h:minorisation}, we know that for any $x \in \T$, there is a positive probability of reaching $B_\varepsilon(x_c)$ in $k$ steps. For any $m = k + nk_c$ sufficiently large, we have
		\begin{align*}
			\mathbb{E} \left[ (DT_\omega^m(x))^{-q} \right]
			&\geq \mathbb{P}(x_k \in B_\varepsilon(x_c))\,   e^{-n \delta q}\,  \mathbb{P}(\mathcal{H}^c)^n.
		\end{align*}	
		Taking the logarithm and applying the limit $\lim_{n \to \infty} \frac{\ln(\cdot)}{n}$, it follows that $\Lambda(q) \to \infty$ as $q \to -\infty$.	
		By continuity of $\Lambda(q)$ and the fact that $\Lambda'(0) > 0$, there is a unique $\gamma < 0$ such that $\Lambda(\gamma) = 0$. 
		
The argument in the case where $\lambda < 0$  is analogous, using an expanding periodic orbit instead of the contracting periodic orbit.  The existence of an expanding periodic orbit $x_c$ of $T$ follows from \cite{MR806250}.
Let $k_c$ denote its period. 
For $\omega_0,\ldots,\omega_{k_c-1}$ near $0$, the maps $T^{k_c}_\omega$ have hyperbolic repelling fixed points near $x_c$. Since $\frac{\partial}{\partial \omega} T_\omega > 0$ everywhere, these hyperbolic repelling fixed points form a set that contains an open neighborhood of $x_c$. 
Using this, there exist $\varepsilon>0$, $\delta>0$ and $C>0$, such that for each $x \in B_\varepsilon (x_c)$ the set
\[
\mathcal{H}^c_{x} = \big\{ \omega \in \Sigma_\vartheta \; ; \;  T^{k_c}_\omega (x) \in B_\varepsilon (x_c),\, \ln(DT_\omega^{k_c}(x)) > \delta \big\}
\]
has strictly positive measure $\mathbb{P} (\mathcal{H}^c_{x}) > C$.  The proof can now be finished as above.
}}			
    
\end{proof}

We write $\partial_q \phi_q$ for the derivative of $\phi_q$ with respect to $q$, and likewise
$\partial^2_q \phi_q$ for the second order derivative with respect to $q$.
Recall from \eqref{eq:Lambda'(0)} that $\lambda = \Lambda'(0)$ and 
let
\begin{align*}
V &\coloneq \Lambda''(0).
\end{align*}

\begin{lemma}\label{lem:analyticeigen}
We have the following two equalities for $x \in \T$:
\begin{equation}
 \mathbb{E}\left[ \ln(D T_\omega (x) ) \right]	-\lambda = \left(P_0 - I\right){\partial_q \phi_0} (x) \label{eq:firstmom}
\end{equation}
and, if $\lambda = 0$, 
\begin{align}
	V = \left(P_0 - I\right){\partial_q^2 \phi_0} (x)-  2 \mathbb{E}\left[ {\partial_q \phi_0} (T_\omega(x)) \ln(D T_\omega (x)) \right] +   \mathbb{E}\left[ \ln^2(DT_\omega (x))\right]. 
	\label{eq:secondmom}
\end{align}
\end{lemma}

\begin{proof}
Differentiating \eqref{eq:twistedeig} once with respect to $q$, yields for  $x \in \T$,
\begin{align}
	e^{\Lambda(-q)}{\partial_q \phi_q} (x) - e^{\Lambda(-q)}\Lambda'(-q)\phi_q(x)  &= \mathbb{E}\left[\partial_q\left( \frac{\phi_q\left(T_\omega(x)\right)}{D T_\omega (x) ^q}\right)\right]\nonumber
	\\
	&= ({P}_{q}{\partial_q \phi_q} )(x) - \mathbb{E}\left[ \frac{\phi_q\left(T_\omega(x)\right) \ln(DT_\omega (x))}{D T_\omega (x)^q}\right]. \label{eq:dp}
\end{align}		
{  Note that the derivative of $\phi_q(x), $ with respect to $q$, exists because of Proposition \ref{prop:specnew}. }

Equation \eqref{eq:dp} evaluated at $q = 0$ yields
\begin{equation*}
	{\partial_q \phi_0} (x) - \Lambda'(0)  = ({P}_{0}{\partial_q \phi_0} )(x) - \mathbb{E} \left[ \ln(DT_\omega (x)) \right].
\end{equation*}
Rewriting this proves \eqref{eq:firstmom}.

Next, differentiating \eqref{eq:dp} with respect to $q$ yields
\begin{multline}
	e^{\Lambda(-q)}\left({\partial^2_q \phi_q} - 2\Lambda'(-q) {\partial_q \phi_q} + \left(\Lambda'(-q)^2 +\Lambda''(-q) \right)\phi_q \right)(x)
	\\  \label{eq:d2p}
	= {P}_q{\partial_q^2 \phi_q} (x) - 2\mathbb{E} \left[ \frac{{\partial_q \phi_q} \left(T_\omega(x) \right)\ln(DT_\omega (x))}{D T_\omega (x)^q}\right]
	+ \mathbb{E}\left[\frac{\phi_q\left(T_\omega(x) \right)\ln^2(D T_\omega (x))}{D T_\omega (x)^q}\right].
\end{multline}
Evaluating \eqref{eq:d2p} for $q=0$ yields
\begin{multline}\label{eq:d2peval}
	 - 2\Lambda'(0) {\partial_q \phi_0} (x) + \Lambda'(0)^2  + \Lambda''(0) 
	\\
= \left( {P}_0 - I \right) {\partial_q^2 \phi_0} (x) - 2\mathbb{E} \left[ {\partial_q \phi_0} \left(T_\omega(x) \right)\ln(DT_\omega (x))  \right]
	+ \mathbb{E}\left[\ln^2(D T_\omega (x))\right].
\end{multline}
%
%
Rewriting, and plugging in $\lambda=0$, we obtain \eqref{eq:secondmom}.
\end{proof}

It is immediate from the convexity of $\Lambda$ that $V \geq 0$. Our set-up implies that in fact $V>0$, which we require in our analysis of the case with a zero Lyapunov exponent $\lambda$ (see Lemma~\ref{lem:martingale} below).

\begin{lemma}
We have $V>0$.
\end{lemma}

\begin{proof}
Following \cite{MR1671091} we establish that $V = 0$ implies $\Lambda (q) = \lambda q$, which will lead to a contradiction.

Rewriting \eqref{eq:d2peval}, we obtain
\begin{multline}\label{eq:V+}
 V + \left(P_0 - I\right)  \left(  (\partial_q  \phi_0 )^2 - \partial^2_q \phi_0 \right) (x)
\\
=
\mathbb{E}\left[  \left( \partial_q \phi_0 (T_\omega(x))  - \ln(DT_\omega(x)) \right)^2 \right]
 -  \left(\partial_q \phi_0 (x) - \lambda\right)^2
\\
=
\mathbb{E}\left[  \left( \partial_q \phi_0 (T_\omega(x))  - \ln(DT_\omega(x)) \right)^2 \right]
 -  \left( \mathbb{E}\left[  \partial_q \phi_0 (T_\omega(x))  - \ln(DT_\omega(x))   \right]\right)^2. 
\end{multline}
The last step uses \eqref{eq:firstmom}. We conclude that
\[
V + \left(P_0 - I\right)  \left(  (\partial_q  \phi_0 )^2 - \partial^2_q \phi_0 \right) (x) \geq 0.
\]
Now suppose $V=0$, for the sake of contradiction.  From $ \left(P_0 - I\right)  \left(  (\partial_q  \phi_0 )^2 - \partial^2_q \phi_0 \right) (x) \geq 0$ we conclude that $(\partial_q  \phi_0 )^2 - \partial^2_q \phi_0$ is constant
(to see this, note that if a continuous function $\chi:\mathbb{T} \to \mathbb{R}$ takes a maximum at $x$, then the property $\mathbb{E} \left[ \chi (T_\omega(x)) \right] \geq \chi (x)$ with Hypothesis~\ref{h:minorisation} implies that $\chi$ is maximal at every point of $\mathbb{T}$).  
So in fact 
\[
\left(P_0 - I\right)  \left(  (\partial_q  \phi_0 )^2 - \partial^2_q \phi_0 \right) (x) = 0
\]
and \eqref{eq:V+} implies that
\[
 \partial_q \phi_0 (T_\omega(x))  - \ln(DT_\omega(x)) =  \partial_q \phi_0 (x) - \lambda
\]
for all $x$ and almost all $\omega$.
Then also
\[
 \partial_q \phi_0 (T^n_\omega(x))  - \ln(DT^n_\omega(x)) =  \partial_q \phi_0 (x) - n\lambda.
\]
As $\Lambda(q) = \lim_{n\to\infty} \frac{1}{n} \ln \left(  \mathbb{E} \left[ \left(DT^n_\omega(x) \right)^q \right] \right)$ we find $\Lambda(q) = q \lambda$. Lemma~\ref{lem:second0} shows that this is not the case in our set-up.
\end{proof}

\subsection{Linearized Koopman operator}\label{ss:linkoop}

We introduce the \textit{linearized Koopman operator}, 	defined for continuous real valued functions  $\phi$ on ${\T\times\R^+}$ by
\[
TP \phi(x,u) \coloneqq \mathbb{E} \left[\phi(T_\omega(x), DT_\omega (x)u)\right].
\]

Recall from Proposition~\ref{prop:specnew} that $\phi_q$ is the dominant eigenfunction of ${P}_q$. 
Define $\tilde{W}_q:{\T\times\R^+}\to \R$ by
\begin{align}\label{eq:tildeWq}
\tilde{W}_q(x,u) &= u^{-q}\phi_q\left(x\right).
\end{align}
We reserve the \emph{tilde} notation for functions on $\T\times \R^+$.

\begin{lemma}\label{lem:TPW}
The function $\tilde{W}_q$ is an eigenfunction of $TP$, with eigenvalue ${e^{\Lambda(-q)}}$:
\begin{equation}\label{eq:TPW}
	TP  \tilde{W}_q  (x,u) = e^{\Lambda(-q)} \tilde{W}_q  (x,u).
\end{equation}
\end{lemma}

\begin{proof}
This follows from a straightforward computation. For $(x,u) \in {\T\times\R^+}$,
\begin{align*}
	TP \tilde{W}_q (x,u)&=  \mathbb{E} \left[\tilde{W}_q(T_\omega(x),D T_\omega (x) u)\right] \\
	&= \mathbb{E} \left[(DT_\omega(x) u)^{-q}\phi_q\left(T_\omega (x)\right)\right] \\
	&= u^{-q}\mathbb{E} \left[DT_\omega(x) ^{-q}\phi_q\left(T_\omega (x)\right)\right]\\
	&=  u^{-q} {P}_q\phi_q(x)\\
	& = e^{\Lambda(-q)}  u^{-q} \phi_q(x)\\
	&= e^{\Lambda(-q)}\tilde{W}_q(x,u).
\end{align*}
\end{proof}

\begin{remark}
Consider a random sequence 
\[
(x_{n+1}, u_{n+1}) = \left(T_{\sigma^n\omega} (x_n), DT_{\sigma^n\omega} (x_n) u_n\right),
\]
with $x_0,u_0 \in \T\times\R^+$.
 Take $\gamma$ as in Lemma~\ref{lem:second0} such that $e^{\Lambda(\gamma)}=1$. 
Then
\[
(TP - I)  \tilde{W}_\gamma  (x,u) = 0,
\]
which shows that $\tilde{W}_\gamma(x_n,u_n)$ is a martingale. \hfill $\blacksquare$
\end{remark}

\begin{example}\label{ex:white-noise-cont1}
In the setting of Example~\ref{ex:T_nu} with $\vartheta = 0.5$, we obtain
\[
\tilde W_q(x,u) = \|\phi_q\|\,u^{-q},
\]
which is independent of $x$, {  since $P_q\psi$ is constant for every $\psi \in \mathcal C(\T,\R)$.} \hfill $\blacksquare$
\end{example}

\begin{lemma}\label{lem:TPphieta}
There exist $K>0$ and two continuous  functions 
\[ \tilde{\phi}, \tilde{\eta} \in \mathcal{C}^0 ({\T\times\R^+},\mathbb R), \] 
such that the following holds for $(x,u) \in {\T\times\R^+}$.
For $\tilde{\phi}$ we have				
\begin{align}
	\vert \tilde{\phi}(x,u) - \ln(u) \vert &\leq K,  \label{eq:TPphi_log}
	\\
	\left( TP-I\right)\tilde{\phi} (x,u) 
	&=\lambda.   \label{eq:TPphimartingale}
\end{align}
Assume $\lambda = 0$. Then for $\tilde{\eta}$ we have
\begin{align}
	\vert \tilde{\eta}(x,u) - \ln^2(u) \vert &\leq K|\ln(u)|, \label{eq:TPeta_log}
	\\
	\left( TP-I\right)\tilde{\eta}(x,u)  &= 
	V.  \label{eq:TPetamartingale}
\end{align}
\end{lemma}

{  Concerning notation in the following proof and further below, $\pi_1$ 
is the function $\pi_1 (x,u) = x$ on $\T\times\R^+$,
$\pi_2$ stands for the function $\pi_2 (x,u) = u$ on $\T\times\R^+$. }

\begin{proof}[Proof of Lemma~\ref{lem:TPphieta}]
We give a constructive proof for Lemma~\ref{lem:TPphieta}, where we use Proposition~\ref{prop:specnew} and Lemma~\ref{lem:analyticeigen}.
We first find the function $\tilde{\phi}$ for which \eqref{eq:TPphi_log} and \eqref{eq:TPphimartingale} holds. Subsequently we find $\tilde{\eta}$ for which \eqref{eq:TPeta_log} and \eqref{eq:TPetamartingale} holds. 

For $ (x,u) \in {\T\times\R^+}$, 
let 
\begin{align}\label{eq:tildephi}
	\tilde{\phi}(x,u) &= \ln(u) -{\partial_q \phi_0} (x). 
\end{align}
Note that 
\begin{align*}
	\tilde{\phi}\in \mathcal{C}^0 ({\T\times\R^+},\mathbb R).
\end{align*} 
Furthermore, 
for all $(x,u) \in {\T\times\R^+}$, we have 
\begin{align*}
	\left(TP-I\right) \tilde{\phi} (x,u) &= 
	\left(TP-I\right)\ln(\pi_2)(x,u)  -\left(P_0-I\right) \left({\partial_q \phi_0} \right)(x) 
	\\  
	&= \mathbb{E} \left[\ln( DT_\omega(x) u )\right] - \ln(u) - \mathbb{E} \left[\ln( DT_\omega(x)  )\right] + \lambda 
	\\ 
	&=\lambda. 
\end{align*}
In the first line we use that ${\partial_q \phi_0} (x)$ only depends on $x$.
For the second equality we apply \eqref{eq:firstmom} in Lemma~\ref{lem:analyticeigen}.

For the second part of the proof, we take $\lambda=0$ and let 
\begin{align*}
	\tilde{\eta}(x,u) = \ln^2(u) -2\ln(u){\partial_q \phi_0} (x) +{\partial^2_q \phi_0} (x).
\end{align*}
And again note that 
\begin{align*}
	\tilde{\eta}\in \mathcal{C}^0 ({\T\times\R^+},\mathbb R).
\end{align*} 
From a straightforward computation, we obtain, for $(x,u)\in {\T\times\R^+}$ 
\begin{multline*}
	\left(TP-I\right)\eta (x,u)  = 
	\\  \left(TP-I\right) \ln^2(\pi_2)(x,u)  
	- 2 \left(TP-I\right) \ln(\pi_2){\partial_q \phi_0} (\pi_1)(x,u)
	+ \left(TP-I\right)  {\partial^2_q \phi_0} (\pi_1)(x,u). 
\end{multline*}
We analyse the terms on the right hand side separately. For the first term, we have
\begin{align}
	\left(TP-I\right) \left(\ln^2(\pi_2)\right)(x,u) 
	&= \mathbb{E} \left[\ln^2(DT_\omega(x) u)\right] -  \ln^2(u) \nonumber
	\\ 
	&= \mathbb{E} \left[\left(\ln(DT_\omega(x)) + \ln (u)\right)^2\right] -  \ln^2(u) \nonumber
	\\
	&=  \mathbb{E} \left[\ln^2(DT_\omega(x)) \right] + 2 \ln(u) \mathbb{E} \left[\ln(DT_\omega(x)) \right].          \label{eq:eta1}
\end{align}

For the second term we have
\begin{align}
    2\left(TP-I\right) &\ln(\pi_2){\partial_q \phi_0} (\pi_1)(x,u) 
	= 2\mathbb{E} \left[\ln(DT_\omega(x)u){\partial_q \phi_0} (T_\omega(x)) \right] - 2\ln(u){\partial_q \phi_0} (x) \nonumber
	\\
	&= 2\mathbb{E} \left[\ln(DT_\omega(x)){\partial_q \phi_0} (T_\omega(x)) \right] +2 \ln(u)\left(\mathbb{E} \left[ {\partial_q \phi_0} (T_\omega(x))\right] -  {\partial_q \phi_0} (x) \right) \nonumber
	\\
	&= 2\mathbb{E} \left[\ln(DT_\omega(x)){\partial_q \phi_0} (T_\omega(x)) \right] +2 \ln(u) (P-I) {\partial_q \phi_0} (x) \nonumber
	\\
	&= 2\mathbb{E} \left[\ln(DT_\omega(x)){\partial_q \phi_0} (T_\omega(x)) \right]+ 2\ln(u)\mathbb{E} \left[\ln(DT_\omega(x)) \right], \label{eq:eta2} 
\end{align}
where in the third to fourth line we make use of the properties of  ${\partial_q \phi_0} $, as described by \eqref{eq:firstmom} in  Lemma~\ref{lem:analyticeigen}.
For the third term, we have
\begin{align}
	\left(TP-I\right) {\partial^2_q \phi_0} (\pi_1)(x,u) &=V +2\mathbb{E}\left[\ln(DT_\omega(x) ){\partial_q \phi_0} \left(T_\omega(x)\right)\right]- \mathbb{E} \left[\ln^2(DT_\omega(x) )\right]\label{eq:eta3}. 
\end{align}
Here we use \eqref{eq:secondmom} in Lemma~\ref{lem:analyticeigen}. 
Combining \eqref{eq:eta1}, \eqref{eq:eta2} and \eqref{eq:eta3}, we conclude, for $(x,u)\in {\T\times\R^+}$, 
\begin{align*}
	\left(TP-I\right)\tilde\eta (x,u) =  V .
\end{align*}

Finally,  \eqref{eq:TPphi_log} and \eqref{eq:TPeta_log} follow from the analyticity in $q$ of $\phi_q(x)$(see Proposition~\ref{prop:specnew}).
\end{proof}

\begin{remark}
Consider a random sequence 
\[
(x_{n+1}, u_{n+1}) = \left(T_{\sigma^n\omega} (x_n), DT_{\sigma^n\omega} (x_n) u_n\right),
\]
with $(x_0,u_0) \in \T\times\R^+$.
 Then \eqref{eq:TPphimartingale}  expresses that the function $\tilde{\phi}(x_n,u_n) - n\lambda$ is a martingale. Assuming $\lambda=0$,  \eqref{eq:TPetamartingale}  expresses that the function $\tilde{\eta}(x_n,u_n) - nV$ is a martingale. \hfill $\blacksquare$
\end{remark}

\begin{example}\label{ex:white-noise-cont2}
In the case of Example~\ref{ex:T_nu}, with $\vartheta = 0.5$, we get $\tilde{\phi} (x,u) = \ln(u) + \|\partial_q\phi_0\|$ and $\tilde{\eta}(x,u) = \ln^2(u)-2\|\partial_q\phi_0\|\ln(u) + \|\partial_q^2\phi_0\|$. Both functions are independent of $x$,  {  as $P_q\psi$ is constant for every $\psi \in \mathcal C(\T,\R)$.} \hfill $\blacksquare$
		\end{example}
		
\subsection{Two-point Koopman operator}\label{ss:insideD}

The {\textit{(annealed) two-point Koopman operator} $P_{(2)}$
is defined for a real valued function $\phi$ on $\T^2\setminus\Delta$ by
\[
P_{(2)}\phi (x,y) \coloneqq \mathbb{E} \left[\phi(T_\omega^{(2)}(x,y))\right].
\] 

Define $\phi:\T^2 \setminus \Delta \to \mathbb{R}$ by  
$$
\phi(x,y) \coloneqq \ln(d(x,y)) - {\partial_q \phi_0} (x). 
$$ 
With $\tilde{\phi}$ from Lemma~\ref{lem:TPphieta} (see \eqref{eq:tildephi}) we have,
for $(x,y) \in \T^2 \setminus \Delta$,
\begin{equation}
	\phi(x,y) = \tilde{\phi}(x,d(x,y)). \label{eq:phi-tilde}
\end{equation}

Let $\eta:\T^2 \setminus \Delta \to \mathbb{R}$ be defined as   
$$
\eta(x,y) \coloneqq \ln^2(d(x,y)) -2{\partial_q \phi_0} (x)\ln(d(x,y)) + {\partial_q \phi_0} (x). 
$$ 
Note that for $(x,y) \in \T^2 \setminus \Delta$,
\begin{equation}
	\eta(x,y) = \tilde{\eta}(x,d(x,y)). \label{eq:eta-tilde}
\end{equation}		
Define $W_q: \T^2\setminus\Delta \to \R$ by 
\begin{equation}\label{eq:Wdef}
	W_q(x,y) \coloneqq d(x,y)^{-q}\phi_q(x).
\end{equation}
With $\tilde{W}_q$ from \eqref{eq:tildeWq} we have for $(x,y) \in \T^2\setminus\Delta$,
\begin{equation}
	W_q(x,y) = \tilde{W}_q(x,d(x,y)). \label{eq:W-tilde}
\end{equation}

 The next lemma compares the action of $TP$ and $P_{(2)}$.

\begin{lemma}\label{lem:TP-P2}
	There exist $R,B>0$, such that we have the following bounds, for  $(x,y)\in \Delta_R\setminus\Delta$,
    \begin{align}
		\left\vert TP\tilde\phi(x,d(x,y)) - P_{(2)}\phi(x,y)) \right\vert &\leq B d(x,y)\label{eq:(TP-P)phi}, \\
		\left\vert TP\tilde\eta(x,d(x,y)) - P_{(2)} \eta(x,y)\right\vert &\leq  B d(x,y)|\ln(d(x,y))|, \label{eq:(TP-P)eta} 
	\end{align}
	and for $q \in [-|\gamma| -1 , |\gamma| +1 ]$,	
        \begin{equation}
            \left|TP\tilde{W}_q(x,d(x,y))-P_{(2)}W_q(x,y)\right| \leq B d(x,y)^{-q+1}.\label{eq:TPW-PW}
        \end{equation}
\end{lemma}

\begin{proof}  
	{ {We will work out the estimates for \eqref{eq:(TP-P)phi}, then the computation for \eqref{eq:(TP-P)eta} is analogous.}}  We will sketch the proof for \eqref{eq:TPW-PW}.
	
	Recall that $w(x,y)$ denotes the signed distance between $x$ and $y$, for nearby points $x, y \in \mathbb{T}$.
	Take $R >0$ small enough so that $T_\omega$ is injective on all intervals of length $R$.
	For $(x,y) \in \Delta_{R}\setminus \Delta$,
	\begin{align}
		TP \tilde{\phi}(x,|w(x,y)|)&= \mathbb{E} \left[\ln(|DT_\omega(x) w(x,y)| )\right]  - P_0{\partial_q \phi_0} (x) \nonumber\\
		&= \ln(|w(x,y)|) + \mathbb{E} \left[\ln(DT_\omega(x)) \right]  -\mathbb{E} \left[{\partial_q \phi_0} \left(T_\omega(x)\right)\right].\label{eq:TP}
	\end{align}

	To determine $P_{(2)}\phi$, we use a Taylor expansion for $T_\omega$. 
	We obtain,
	\begin{equation*}
		w(T_\omega^{(2)}(x,y)) = DT_\omega(x)w(x,y) +\frac{1}{2} D^2 T_\omega (\xi)w(x,y)^2
		\label{eq:taytay},
	\end{equation*}
	for a $\xi$ in the interval between $x$ and $x + w(x,y)$.
	Therefore, for $(x,y) \in \Delta_{R}\setminus \Delta$ with $R>0$ small enough, 
	\begin{multline}
		P_{(2)} \phi (x,y) = \mathbb{E} \left[\ln\left(\left| w(T_\omega^{(2)}(x,y))\right|\right)\right] -P_0{\partial_q \phi_0} (x) 
		\\
		= \mathbb{E} \left[\ln\left(\left| DT_\omega(x) w(x,y) + \frac{D^2T_\omega(\xi)  w(x,y)^2}{2} \right|\right)\right]  - P_0{\partial_q \phi_0} (x) 
		\\
		= \ln(|w(x,y)|) + \mathbb{E} \left[\ln\left(DT_\omega(x)\right)  \right] 
		+ \mathbb{E} \left[\ln\left( \left| 1 +\frac{D^2T_\omega(\xi)  w(x,y)}{2 DT_\omega (x)} \right|\right) \right]  -
		\mathbb{E} [{\partial_q \phi_0} (T_\omega(x))].
		\label{eq:P}
	\end{multline}
	
	Combining \eqref{eq:TP} and \eqref{eq:P} yields, for $(x,y) \in \Delta_{R}\setminus \Delta$,
	\begin{align}
		\left| P_{(2)} \phi (x,y) - TP \phi (x,|w(x,y)|) \right| &\leq \left|
		\mathbb{E} \left[\ln\left( \left| 1 +\frac{D^2T_\omega(\xi)  w(x,y)}{2 DT_\omega (x)} \right|\right) \right] \right|
		\label{eq:2Bbdwithlog}.
	\end{align}
	The right hand side of \eqref{eq:2Bbdwithlog} can be bounded using the standard inequalities $x/(1+x)\leq \ln(1+x) \leq x$ for $x>-1$.
	For $R$ small we have  $\left|  \frac{D^2 T_\omega(\xi) w(x,y)}{2DT_\omega(x)}\right| < 1$ for $(x,y) \in \Delta_{R}\setminus \Delta$.
	Then for the upper bound we obtain
	\begin{align*}
		\mathbb{E} \left[\ln\left(1 + \frac{D^2 T_\omega(\xi)  w(x,y)}{2DT_\omega(x)}\right)\right] &\leq  \frac{|w(x,y)| \Vert D^2T_\omega\Vert }{2 a_1}.
	\end{align*}	
	Similarly for the lower bound, when we take $R$ small enough so that $|w(x,y)| \Vert D^2 T_\omega \Vert < a_1$ for $(x,y) \in \Delta_{R}\setminus \Delta$, we obtain
	\begin{align*}
		\mathbb{E} \left[ \ln\left(1 + \frac{D^2T_\omega(\xi) w(x,y)}{2DT_\omega(x)} \right)\right]  
		&\geq  \mathbb{E} \left[ \ddfrac{\frac{D^2T_\omega(\xi) w(x,y)}{2DT_\omega(x)}}{1 + \frac{D^2T_\omega(\xi) w(x,y)}{2DT_\omega(x)}} \right] 
		\\
		&\geq  \mathbb{E} \left[ \dfrac{D^2T_\omega(\xi) w(x,y)}{2DT_\omega(x) + D^2T_\omega(\xi) w(x,y)}\right] 
		\\
		&\geq -|w(x,y)|\Vert D^2T_\omega\Vert  \mathbb{E} \left[ \frac{1}{2DT_\omega(x) + D^2T_\omega(\xi) w(x,y)}\right] 
		\\
		&\geq \frac{-|w(x,y)|\Vert D^2T_\omega\Vert}{a_1}.
	\end{align*}
		Setting $B = \frac{\Vert D^2T_\omega \Vert}{a_1}$ finishes the estimates for  \eqref{eq:(TP-P)phi}.

    { 
    To prove \eqref{eq:(TP-P)eta}, we again perform an explicit calculation.
Note that $\eta = \ln^2(d(x,y)) - 2 \partial_q \phi_0 (x)\ln(d(x,y)) + \partial_q\phi_0(x)$ is made up from three terms. When bounding the difference $\left\vert TP\tilde\eta(x,d(x,y)) - P_{(2)} \eta(x,y)\right\vert$, the contribution from the last two terms is treated as above. 
    %
	It therefore suffices to bound the contribution from the first term $\ln^2(d(x,y))$.
	
	For $(x,y) \in \Delta_{R}\setminus \Delta$, factoring the Taylor expansion as in the previous step, we have
	\begin{align*}
		P_{(2)} \ln^2(d) (x,y) 
		&= \mathbb{E} \left[\ln^2\left(\left| DT_\omega(x)w(x,y) + \frac{D^2T_\omega(\xi) w(x,y)^2}{2} \right|\right)\right]
		\\
		&= \mathbb{E} \left[\left( \ln(|DT_\omega(x)w(x,y)|) + \ln\left( \left| 1 +\frac{D^2T_\omega(\xi) w(x,y)}{2 DT_\omega (x)} \right| \right) \right)^2 \right].
	\end{align*}
	Expanding the square and comparing with the term $\mathbb{E}[\ln^2(|DT_\omega(x)w(x,y)|)]$ from $TP \tilde{\eta}$, we find that the difference is bounded by
	\begin{align*}
		\left| P_{(2)} \eta (x,y) - TP \tilde{\eta} (x,d(x,y)) \right| 
		&\leq \mathbb{E} \left[ \left| 2\ln(|DT_\omega(x)w(x,y)|) \ln\left( 1 +\frac{D^2T_\omega(\xi) w(x,y)}{2 DT_\omega (x)} \right) \right| \right] 
		\\
		&+ \mathbb{E} \left[ \ln^2\left( 1 +\frac{D^2T_\omega(\xi) w(x,y)}{2 DT_\omega (x)} \right) \right] + O(d(x,y)).
	\end{align*}
    Using the inequality $u^2/(1+u)^2\leq \ln^2(1+u) \leq u^2$ for small $u$ (as in \eqref{eq:2Bbdwithlog}), the second term is bounded by $O(d(x,y)^2)$.
    For the first term, since $R$ is small and $|\ln(|DT_\omega(x)|)|$ is bounded, we have $|\ln(|DT_\omega(x)w(x,y)|)| \leq C |\ln(d(x,y))|$ for some constant $C$. Thus,
	\begin{align*}
		 \left| P_{(2)} \eta (x,y) - TP \tilde{\eta} (x,d(x,y)) \right| 
		 &\leq 2 C |\ln(d(x,y))| \cdot \frac{|w(x,y)| \Vert D^2T_\omega\Vert }{2 a_1} + O(d(x,y))
		 \\
		 &\leq B d(x,y) |\ln(d(x,y))|.
	\end{align*}

	We can choose $B$ such that \eqref{eq:(TP-P)eta} holds for $(x,y) \in \Delta_{R}\setminus \Delta$ with $R$ small enough.
    }
    
		For \eqref{eq:TPW-PW}, take $(x,y) \in \Delta_{R}\setminus \Delta$ with $R$ small and assume, without loss of generality, $x<y$.
	Then
	\begin{multline*}
		\left|TP\tilde{W_q}(x,d(x,y))-P_{(2)}W_q(x,y)\right| \nonumber
		\\
		\leq 
		\left\Vert \phi_q\right\Vert \mathbb{E}  \left[ \left|\left| DT_\omega(x) w(x,y) \right|^{-q} -  \left|DT_\omega(x) w(x,y) +\frac{1}{2} D^2T_\omega(\xi) w(x,y)^2 \right|^{-q}  \right| \right].
		\nonumber 
	\end{multline*}
	 With \eqref{eq:taytay} and the mean value theorem for $a \mapsto a^{-q}$, we get
	\begin{align*}
		\left|TP\tilde{W_q}(x,w(x,y))-P_{(2)}W_q(x,y)\right| 
		&\leq C_q |q| \mathbb{E} \left[ |w(x,y)|^{-q-1}|D^2T_\omega(\xi)w(x,y)|^2 \right]
		\\
		&\leq C_q |q| \Vert D^2T \Vert^2 |w(x,y)|^{-q+1}
		\\
		&\leq B d(x,y)^{-q+1},
	\end{align*}
    for some positive constant $C_q$. As $C_q|q|$ depends continuous on $q$, restricting $q$ to $[-|\gamma|-1,|\gamma|+1]$ allows us to uniformly bound it with a constant $B$. This completes the proof.
\end{proof}

The following lemma is a key lemma that adapts the equalities and estimates in  Lemma~\ref{lem:TP-P2} for $TP$ to the setting for $P_{(2)}$, but only near the diagonal $\Delta$,
using that near $\Delta$, $P_{(2)}$ can be approximated by $TP$ (Lemma~\ref{lem:TP-P2}).

\begin{proposition}\label{prop:phieta}
	There exists a $R,K>0$, and continuous integrable functions
	\[
	\phi^\pm, \eta^\pm \in \mathcal{C}^0(\mathbb{T}^2\setminus \Delta, \R),
	\]
	such that the following holds for $(x,y) \in \Delta_R\setminus \Delta$.
	
	\noindent For $\phi^\pm$ we have
	\begin{align}
		\vert \phi^\pm(x,y) - \ln(d(x,y)) \vert &\leq K,
		\label{eq:phi_log}
		\\
		\left( P_{(2)}-I\right)\phi^- (x,y) &\leq \lambda \leq   \left( P_{(2)}-I\right) \phi^+ (x,y). \label{eq:phimartingale}
	\end{align}
	Assume $\lambda=0$. Then for $\eta^\pm$ we have
	\begin{align}
		\vert \eta^\pm(x,y) - \ln^2(d(x,y)) \vert &\leq K|\ln(d(x,y))| 
		\label{eq:eta_log},
		\\
		\left( P_{(2)}-I\right)\eta^- (x,y)  &\leq V \leq  \left( P_{(2)}-I\right)\eta^+ (x,y).
		\label{eq:etamartingale}
	\end{align}
\end{proposition}

\begin{proof}  
	We can bound $\left(P_{(2)}-I\right)\phi(x,y)$, by applying the triangle inequality and using \eqref{eq:TPphimartingale}, \eqref{eq:phi-tilde} and  \eqref{eq:(TP-P)phi}:
	{  \begin{multline*}
		\left\vert\left(P_{(2)}-I\right)\phi (x,y) - \lambda\right\vert \leq
		\\ \overbrace{\left\vert TP\tilde{\phi}(x,d(x,y))-P_{(2)}\phi(x,y)\right\vert}^{\leq Bd(x,y) \text{ by \eqref{eq:(TP-P)phi}}} 
		+ 
		\overbrace{\left\vert\left(TP-I\right)\tilde\phi(x,d(x,y))-\lambda\right\vert}^{=0 \text{ by \eqref{eq:TPphimartingale}}}
		+
		\overbrace{\left\vert \phi(x,y)- \tilde{\phi}(x,d(x,y))\right\vert}^{= 0 \text{ by \eqref{eq:phi-tilde}}},
	\end{multline*} }
	for  $(x,y) \in \Delta_{R}\setminus \Delta$.
	So
	\begin{equation}
		\left\vert\left(P_{(2)}-I\right)\phi(x,y) - \lambda\right\vert \leq
		Bd(x,y), 
		\label{eq:(P2-I)phi}
	\end{equation}
	for  $(x,y) \in \Delta_{R}\setminus \Delta$.
	
	This basically means that $\left(P_{(2)}-I\right)\phi$ near the diagonal is close to $\lambda$.
	To obtain functions $\phi^\pm$ so that $P_{(2)} - I$ applied to them, after subtracting $\lambda$, has a definite sign, 
	we add suitable functions to $\phi$:
	we take $\phi^\pm:\T^2\setminus \Delta \to \R$ of the form 
	\begin{equation}\phi^\pm(x,y) = \phi(x,y) \pm c_1 W_{-{q_0}}(x,y), \label{eq:phipm} \end{equation}
	for a small positive value $q_0\in (0,1/2)$, such that $\Lambda(q_0) \neq 0$,
	and $|c_1|$  large enough (will be defined below), with sign such that $c_1 (e^{\Lambda(q_0)}-1) > 0 $. 
	Note that $\phi^\pm \in \mathcal{C}^0(\mathbb{T}^2\setminus \Delta, \R)$.
	
	By Proposition~\ref{prop:specnew}, we obtain
	\begin{align}
		\left(TP-I\right)\tilde W_{-q_0}(x,u) = \left(e^{\Lambda(q_0)}-1\right)  u^{q_0}\phi_{-q_0}(x), 
		\label{eq:TPW-W}
	\end{align}
	for   $(x,u) \in {\T\times\R^+}$.	
		
	Combining \eqref{eq:(P2-I)phi}, \eqref{eq:TPW-W}, \eqref{eq:TPW-PW} and \eqref{eq:W-tilde}  allows us to  prove the first inequality of \eqref{eq:phimartingale}, by choosing $|c_1|$ in \eqref{eq:phipm} large enough. For $(x,y) \in \Delta_R\setminus \Delta $, 
	\begin{multline}
		\left(P_{(2)}-I\right)\phi^+(x,y) - \lambda
		= \left(P_{(2)}-I\right)\phi(x,y) - \lambda + c_1\left(P_{(2)}-I\right)W_{{-q_0}}(x,y)
		\\
		= \overbrace{\left(P_{(2)}-I\right)\phi(x,y) - \lambda}^{\geq - B d(x,y) \text{ by } \eqref{eq:(P2-I)phi} } \qquad + \overbrace{ c_1 (TP-I)\tilde W_{{-q_0}}(x,d(x,y))}^{\geq  c_1 \left(e^{\Lambda(q_0)}-1\right) d(x,y)^{q_0}/C \text{ by }\eqref{eq:TPW-W}}
		\\
		+  \overbrace{ c_1 \left(P_{(2)}W_{{-q_0}}(x,y)-TP\tilde W_{{-q_0}}(x,d(x,y))\right)}^{\geq - |c_1| B d(x,y)^{q_0+1} \text{ by }  \eqref{eq:TPW-PW}} \quad + \quad \overbrace{c_1 \left(W_{{-q_0}}(x,y)-\tilde W_{{-q_0}}(x,d(x,y)) \right) }^{=0 \text{ by }\eqref{eq:W-tilde}} 
		\\
		\geq  - B d(x,y) + c_1\left(e^{\Lambda(q_0)}-1\right)  d(x,y)^{q_0}/C - |c_1|B d(x,y)^{q_0+1} ,
		\label{eq:(P2-I)}
	\end{multline}
	so that
	\begin{equation*}
		\left(P_{(2)}-I\right)\phi^+(x,y) \geq \lambda,
	\end{equation*}
	if $R$ is small and $|c_1|$ is chosen large enough.
	The second inequality regarding $\phi^-$ is obtained using similar bounds.
	The bound \eqref{eq:phi_log} for suitable $K>0$ is immediate from the expressions for $\phi^\pm$.
	
	We take $\lambda = 0$ and proceed with the construction of $\eta^\pm$. 
	We can bound $\left(P_{(2)}-I\right)\eta(x,y)$ by applying the triangle inequality and using \eqref{eq:TPetamartingale}, \eqref{eq:(TP-P)eta} and \eqref{eq:eta-tilde}. This yields 
	\begin{multline*}
		\left\vert\left(P_{(2)}-I\right)\eta(x,y)-V\right\vert 
		\leq \overbrace{\left\vert\left(TP\tilde{\eta}(x,d(x,y))-P_{(2)}\eta(x,y)\right)\right\vert}^{\leq Bd(x,y)|\ln(d(x,y))| \text{ by } \eqref{eq:TPetamartingale}}+
		\\
		\overbrace{\left\vert\left(TP-I\right)\eta(x,d(x,y))-V\right\vert}^{= 0 \text{ by } \eqref{eq:(TP-P)eta}}
		+ \overbrace{\left\vert \eta(x,y)- \tilde{\eta}(x,d(x,y))\right\vert}^{= 0 \text{ by } \eqref{eq:eta-tilde}} 
		\\
	\end{multline*} 
	so that
	\begin{align}
		\left\vert\left(P_{(2)}-I\right)\eta(x,y)-V\right\vert 
		&\leq Bd(x,y)|\ln(d(x,y))|,
		\label{eq:(P2-I)eta}
	\end{align} 
	for all $(x,y) \in \Delta_{R}\setminus \Delta.$
	We may not have that $\left(P_{(2)}-I\right)\eta (x,y) - V$ has a definite sign.
	We therefore add functions to $\eta$ (as we did to obtain $\phi^\pm$ from $\phi$) and let
	\begin{equation} 
		\eta^\pm(x,y) = \eta(x,y) \pm c_2 W_{{-q_0}}(x,d(x,y)), \label{eq:etapm} 
	\end{equation}
	with $|c_2|$  large enough (will be defined below), with sign such that $c_2 (e^{\Lambda(q_0)}-1) > 0 $. We clearly have  that $\eta^\pm \in  \mathcal{C}^0(\mathbb{T}^2\setminus \Delta, \R) $.
	
	Combining \eqref{eq:TPetamartingale},\eqref{eq:TPW-PW} and \eqref{eq:TPW-W}  allows us to prove the first inequality of \eqref{eq:etamartingale}, by choosing $c_2$ in \eqref{eq:etapm} large enough. For $(x,y) \in \Delta_R\setminus \Delta$, we get through a straightforward combination of previous inequalities,
	\begin{multline*}
		\left(P_{(2)}-I\right)\eta^+(x,y) 
		\\
		= \overbrace{\left(P_{(2)}-I\right)\eta(x,y)}^{\geq V -Bd(x,y)|\ln(d(x,y))| \text{ by }\eqref{eq:(P2-I)eta} } +  \overbrace{c_2 \left(P_{(2)}-I\right)W_{{-q_0}}(x,y),}^{\geq c_2\left(e^{\Lambda(q_0)}-1\right)  d(x,y)^{q_0}/C  - |c_2| B d(x,y)^{q_0+1} \text{ by }\eqref{eq:(P2-I)}}
	\end{multline*} 
	so that
	\[
	\left(P_{(2)}-I\right)\eta^+(x,y) \geq V,
	\]		
	for $R$ small enough, by choosing $c_2$ large enough.
	Here we use the fact that for $0<R<1$, and $q_0\in(0,1)$, there exists a $C>0$, such that for all $x \in (0,R)$, $Cx^{q_0} > -x\ln(x)$. 
	The second inequality in \eqref{eq:etamartingale} for $\eta^-$ is obtained using similar bounds.

	Finally, \eqref{eq:eta_log} for $K$ large enough is clear from the expression for $\eta^\pm$, using Proposition~\ref{prop:specnew}.
\end{proof}

\begin{remark}\label{rem:subandsuper}
We can use \eqref{eq:phimartingale} to get 
\begin{equation*}
    \pm \mathbb E\left[\phi^+(T_\omega(x),T_\omega(y)) \right] -\lambda \geq \pm \phi^\pm(x,y),
\end{equation*}
for $d(x,y)$ small enough. Similarly if we assume $\lambda =0$, we get 
\begin{equation*}
    \pm \mathbb E\left[\eta^+(T_\omega(x),T_\omega(y)) \right] -V \geq \pm \eta^\pm(x,y),
\end{equation*}
for $d(x,y)$ small enough.
Applying Doob's stopping time theorem (we refer to standard references such as \cite{MR1368405} or \cite{MR3930614}) with suitable stopping times we build sub- and supermartingales for such random sequences from the functions $\phi^+$, $\phi^-$, $\eta^+$ and $\eta^-$. This is done in the proof of Lemma~\ref{lem:martingale}.  \hfill $\blacksquare$

%
\end{remark}

\begin{proposition}\label{prop:W}
	There exists a $R,K>0$, and a family of continuous functions, 
	\[ W_q^\pm \in \mathcal{C}^0(\mathbb{T}^2\setminus \Delta, \R),\]
	for $q \in [-|\gamma| - 1/2, |\gamma|+1/2]$, such that the following holds for all $(x,y) \in \Delta_R\setminus \Delta$.
	
	\noindent For $W_q^\pm$ we have
	\begin{align}
		\frac{1}{K} d(x,y)^{-q} \leq W^\pm_q(x,y)  &\leq K d(x,y)^{-q},
		\label{eq:psi_log}
		\\
		\left( P_{(2)}-e^{\Lambda(-q)}\right)W_q^- (x,y) &\leq 0 \leq  \left( P_{(2)}-e^{\Lambda(-q)}\right) W_q^+(x,y).
		\label{eq:psimartingale}
	\end{align}
\end{proposition}

\begin{proof}
	We can bound $\left( P_{(2)}-e^{\Lambda(-q)}\right)W_q (x,y)$, by applying the triangle inequality and using Lemma~\ref{lem:TP-P2}:
	\begin{multline*}
		\left\vert\left( P_{(2)}-e^{\Lambda(-q)}\right)W_q (x,y)\right\vert \leq
		\\ 
		\overbrace{	\left\vert P_{(2)}W_q (x,y) -TP\tilde W_q(x,d(x,y))  \right\vert}^{\leq B d(x,y)^{-q+1} \text{ by \eqref{eq:TPW-PW}}} 
		+ 
		\overbrace{\left\vert\left(TP-e^{\Lambda(-q)}\right)\tilde W_q(x,d(x,y))\right\vert}^{= 0 \text{ by \eqref{eq:TPW}}}\\
		+
		e^{\Lambda(-q)} \overbrace{\left\vert W_q(x,y)- \tilde{W}_q(x,d(x,y))\right\vert}^{= 0 \text{ by \eqref{eq:W-tilde}}},
	\end{multline*} 
	for  $(x,y) \in \Delta_{R}\setminus \Delta$.
	So
	\begin{equation}
		\left\vert\left( P_{(2)}-e^{\Lambda(-q)}\right)W_q (x,y)\right\vert \leq
		B d(x,y)^{-q+1}, 
		\label{eq:(P2-I)W}
	\end{equation}
	for  $(x,y) \in \Delta_{R}\setminus \Delta$.
	
	This basically means that $\left(P_{(2)}-e^{\Lambda(-p)}\right)W_q$ near the diagonal is close to 0.
	To obtain functions $W_q^\pm$ so that $P_{(2)} - e^{\Lambda(-p)}$ applied to them, has a definite sign, 
	we add suitable functions to $W_q$:
	we take $W_q^\pm:T^2\setminus \Delta \to \R$ of the form 
	\begin{equation*}
		W_q^\pm(x,y) = W_q(x,y) \pm c_3 W_{{q_1}}(x,y), 
	\end{equation*}
	for $q_1 \in (q-1/2, q)$, such that $\Lambda(-q) - \Lambda(-q_1) \neq 0$, 
	and $|c_3|$ large enough (will be defined below), with sign such that $c_3 (e^{\Lambda(-q)} - e^{\Lambda(-q_1)}) >0.$
	
	By applying \eqref{eq:(P2-I)W} we get for $(x,y) \in \Delta_R\setminus \Delta$,
	\begin{multline*}
		\left( P_{(2)}-e^{\Lambda(-q)}\right)W_q^+ (x,y)= \\
		\overbrace{\left(P_{(2)}-e^{\Lambda(-q)}\right)W_q(x,y)}^{\geq  -B d(x,y)^{-q+1} \text{ by }\eqref{eq:(P2-I)W} } +  \overbrace{c_3 \left(P_{(2)}-e^{\Lambda(-q_1)}\right)W_{{q_1}}(x,y)}^{\geq -c_3\left(B d(x,y)^{-q_1+1}\right) \text{ by }\eqref{eq:(P2-I)W}}\\
		+ \overbrace{c_3 (e^{\Lambda(-q_1)}-e^{\Lambda(-q)})W_{q_1}.}^{\geq c_3(e^{\Lambda(-q_1)}-e^{\Lambda(-q)})d(x,y)^{-q_1}/K \text{ by }\eqref{eq:Wdef}}
	\end{multline*}
	Picking $c_3$ large enough and $R$ small enough yields,
	\begin{equation*}
		\left( P_{(2)}-e^{\Lambda(-q)}\right)W_q^+ (x,y) \geq 0, \quad \text{ for } (x,y) \in \Delta_R\setminus \Delta.
	\end{equation*}
	Similarly we get
	\begin{equation*}
		\left( P_{(2)}-e^{\Lambda(-q)}\right)W_q^- (x,y) \leq 0, \quad \text{ for } (x,y) \in \Delta_R\setminus \Delta.
	\end{equation*}
	Now set $K$ again large enough such that \eqref{eq:psi_log} holds.
\end{proof}

\begin{remark}\label{rem:Wpmsubsuper}
	Suppose $\gamma$ is such that $e^{\Lambda(\gamma)} = 1$.  
We can use \eqref{eq:psimartingale} to get 
\begin{equation*}
    \pm \mathbb E\left[W_{-\gamma}^\pm(T_\omega(x),T_\omega(y)) \right] \geq \pm W_{-\gamma}^\pm(x,y),
\end{equation*}
for $d(x,y)$ small enough. Applying Doob's stopping time theorem  with suitable stopping times we build sub- and supermartingales for such random sequences from the functions $W^+_\gamma$ and $W^-_\gamma$. This is done in the proof of Lemma~\ref{lem:martingalepos}.  
\end{remark}

\section{Topological random dynamics}\label{s:trd}

In this section we prove {  Theorem \ref{thm:main1}} and construct tools to prove Theorem~\ref{thm:main2} in Section~\ref{s:stat}, for $\lambda \geq 0$. We look separately at cases with zero Lyapunov exponent,
positive Lyapunov exponent and negative Lyapunov exponent.  

%
%

\subsection{Zero Lyapunov exponent}\label{s:zero}

The next part of the analysis is to calculate escape probabilities and expected escape times for escape from neighborhoods of the diagonal and strips near the diagonal. The proofs in this section rely on an analysis of Koopman operators, which is developed in Section~\ref{s:koopman}. The statements can be read without reference to Section~\ref{s:koopman}, but for the proofs the reader has to familiarize with the results in Section~\ref{s:koopman}.

In order study escapes from strips near the diagonal, we define stopping times
\begin{align}
\label{eq:tau+}
\tau_{\delta,+}(x,y) &\coloneqq \min \{n\in \mathbb N \; ; \;  d(T_\omega^n(x), T_\omega^n(y))> \delta \},
\\
\label{eq:tau-}
\tau_{\varepsilon,-}(x,y) &\coloneqq \min \{n\in \mathbb N \; ; \; d(T_\omega^n(x), T_\omega^n(y))< \varepsilon \},
\end{align}
where  $(x,y) \in \Delta_R$ with $0<\varepsilon < d(x,y) < \delta<R $.
The following lemma addresses statistics of these stopping times.

\begin{lemma}\label{lem:martingale}
For $\lambda = 0$, there exists a sufficiently small $R>0$, and sufficiently large $K>0$, such that if $0<\varepsilon<d(x,y)<\delta< R$, then
			\begin{equation}
				\mathbb{P}(\min\{\tau_{\varepsilon,-}(x,y), \tau_{\delta,+}(x,y)\}< \infty) = 1, \label{eq:asfin}
			\end{equation}
as well as
\begin{equation}
				\frac{\ln\left( \frac{\delta}{d(x,y)}\right)-2K}{\ln\left(\frac{\delta}{\varepsilon}\right)} \leq \mathbb{P}\left(\tau_{\varepsilon,-}(x,y)< \tau_{\delta,+}(x,y)\right) \leq \frac{\ln\left( \frac{\delta}{d(x,y)}\right)+2K}{\ln\left(\frac{\delta}{\varepsilon}\right)} \label{eq:propbounds}
\end{equation}
and
\begin{multline}
			\frac{1}{V} \left( \ln\left( \frac{\delta}{d(x,y)} \right) \ln\left( \frac{d(x,y)}{\varepsilon} \right) - 6K |\ln (\varepsilon)| -2K^2 \right)  
\\
	\leq \mathbb{E} \left[\min\{\tau_{\varepsilon,-}(x,y), \tau_{\delta,+}(x,y)\}\right]  \label{eq:Ebounds}
\\
	\leq \frac{1}{V} \left( \ln\left( \frac{\delta}{d(x,y)} \right) \ln\left( \frac{d(x,y)}{\varepsilon} \right) +  6K |\ln (\varepsilon)| + 2K^2 \right). 
\end{multline}

{ 
Note that the estimates \eqref{eq:propbounds} and \eqref{eq:Ebounds} are only informative when both $\delta/d(x,y)$ and $d(x,y)/\varepsilon$ are sufficiently large.}
\end{lemma}
        
\begin{proof}
We fix $R>0$ small enough and $K>0$ large enough such that Proposition \ref{prop:phieta} and Proposition \ref{prop:W} hold simultaneously. We follow the reasoning of \cite{MR968817,MR1144097}.	
As indicated above, we will use statements from Section~\ref{s:koopman}. The functions $\phi^\pm$ and $\eta^\pm$ come directly from Proposition~\ref{prop:phieta}.
Denote $(x_n,y_n) = T_\omega^n(x_0,y_0),$ for $(x_0,y_0)\in \T^2\setminus\Delta$, such that $0<\varepsilon<d(x,y)<\delta< R$. Now $\eta^+(x_n,y_n)$ stopped at $\min\{\tau_{\varepsilon,-}(x,y), \tau_{\delta,+}(x,y)\}$ is a submartingale, by Proposition~\ref{prop:phieta}, Remark~\ref{rem:subandsuper} and applying Doob's stopping time theorem. 
So, with 
$$
\tilde n = \min \{n, \tau_{\varepsilon,-}(x,y),  \tau_{\delta,+}(x,y)\},
$$
we have, for a fixed $n\in \N$, by Doob's stopping time theorem with bounded stopping time,
\[
\eta^+(x,y) \leq \mathbb{E} \left[\eta^+(x_{\tilde n},y_{\tilde n})-\tilde{n}V\right].
\]
Rewriting, we obtain
\begin{align*}
	\mathbb E[\tilde n] V&\leq \mathbb{E} \left[\eta^+(x_{\tilde n},y_{\tilde n})\right]-  \eta^+(x,y)\\
	&\leq \ln^2(d(x,y))+\ln^2(\varepsilon a_1) + K|\ln(\varepsilon a_1)|<\infty.
\end{align*}
{ Taking $n\to\infty$ and by the monotone convergence theorem, we have $$\mathbb E[\min \{\tau_{\varepsilon,-}(x,y),  \tau_{\delta,+}(x,y)\}] <\infty.$$ This implies  \eqref{eq:asfin}. }

By Proposition~\ref{prop:phieta} and Remark~\ref{rem:subandsuper} and applying Doob's stopping time theorem { now for almost surely bounded stopping time},
 $\phi^+(x_n,y_n)$ stopped at $\min\{\tau_{\varepsilon,-}(x,y), \tau_{\delta,+}(x,y)\}$ is a submartingale.
Therefore we have
\[
 	\phi^+(x,y) \leq \mathbb{E} \left[\phi^+(x_{\tilde n},y_{\tilde n})\right].
\]
Letting $n$ go to infinity, { applying the dominated convergence theorem}, and conditioning separately on  $\tau_{\varepsilon,-}(x,y)<  \tau_{\delta,+}(x,y)$ or $\tau_{\varepsilon,-}(x,y)<  \tau_{\delta,+}(x,y)$ , we get
			\begin{multline}
				\phi^+(x,y) \leq \lim_{
                n\to\infty} \mathbb{E} \left[\phi^+(x_{\tilde n},y_{\tilde n})\right]
\\
				= \mathbb{P} \left(\tau_{\varepsilon,-}(x,y)<  \tau_{\delta,+}(x,y)\right)\mathbb{E} \left[ \phi^+(x_{\tau_{\varepsilon,-}(x,y)},y_{\tau_{\varepsilon,-}(x,y)})\mid\tau_{\varepsilon,-}(x,y)<  \tau_{\delta,+}(x,y)\right]
\\
				+ \mathbb{P} \left(\tau_{\varepsilon,-}(x,y)>  \tau_{\delta,+}(x,y)\right)\mathbb{E} \left[ \phi^+(x_{\tau_{\delta,+}(x,y)},y_{\tau_{\delta,+}(x,y)})\mid\tau_{\varepsilon,-}(x,y)> \tau_{\delta,+}(x,y)\right].
\label{eq:phi2Bbdd}
			\end{multline}
By \eqref{eq:asfin} we have
\begin{align}\label{eq:sumisone}
P \left(\tau_{\varepsilon,-}(x,y)>  \tau_{\delta,+}(x,y)\right) 
&=
 1 - \mathbb{P} \left(\tau_{\varepsilon,-}(x,y)<  \tau_{\delta,+}(x,y)\right).
\end{align}
From \eqref{eq:phi_log} we get
\begin{align}\label{eq:est-phi+} 
\ln(d(x,y)) -   K  &\leq \phi^+(x,y).
\end{align}
On the iterate $n$ when the distance $d(x_{n},y_{n})$ first leaves $(\varepsilon,\delta)$, the distance is either in $(\varepsilon a_1,\varepsilon ]$ or in $[\delta, \delta a_2)$ by Hypothesis~\ref{h:endo}.
Therefore we can bound the conditional expectation in \eqref{eq:phi2Bbdd} in the following manner,
			\begin{multline}\label{eq:est-E<}
			 \mathbb{E} \left[ \phi^+(x_{\tau_{\varepsilon,-}(x,y)},y_{\tau_{\varepsilon,-}(x,y)})\mid\tau_{\varepsilon,-}(x,y)<  \tau_{\delta,+}(x,y)\right] \leq
\\
 \sup_{r \in [a_1\varepsilon ,\varepsilon)} \{\ln(r) + K\} \leq \ln(\varepsilon) + K
			\end{multline}
			and 
			\begin{multline}\label{eq:est-E>}
			\mathbb{E} \left[ \phi^+(x_{\tau_{\delta,+}(x,y)},y_{\tau_{\delta,+}(x,y)})\mid\tau_{\varepsilon,-}(x,y)> \tau_{\delta,+}(x,y)\right] \leq 
\\
\sup_{r \in [\delta ,a_2\delta ]} \{\ln(r) + K\}\leq \ln(\delta) + 2K,
			\end{multline}
for $K$ chosen large enough.
			Using the bounds \eqref{eq:est-phi+}, \eqref{eq:est-E<}, \eqref{eq:est-E>} and the equality \eqref{eq:sumisone}
in \eqref{eq:phi2Bbdd}, we get
			\begin{align*} 
\mathbb{P} \left(\tau_{\varepsilon,-}(x,y)<  \tau_{\delta,+}(x,y)\right) 
&\leq
				\ddfrac{\ln\left(\frac{d(x,y)}{ \delta}\right) -4K}{\ln\left(\frac{\varepsilon}{ \delta}\right)}.
			\end{align*}
			In a similar fashion we get the counterpart of  \eqref{eq:phi2Bbdd} for $\phi^-$,
			\begin{multline*}
				\phi^-(x,y) \geq \lim_{n\to\infty} \mathbb{E} \left[\phi^-(x_{\tilde n},y_{\tilde n })\right]
\\
				= \mathbb{P} \left(\tau_{\varepsilon,-}(x,y)<  \tau_{\delta,+}(x,y)\right)\mathbb{E} \left[ \phi^-(x_{\tau_{\varepsilon,-}(x,y)},y_{\tau_{\varepsilon,-}(x,y)})\mid\tau_{\varepsilon,-}(x,y)<  \tau_{\delta,+}(x,y)\right]
\\
				+ \mathbb{P} \left(\tau_{\varepsilon,-}(x,y)>  \tau_{\delta,+}(x,y)\right)\mathbb{E} \left[ \phi^-(x_{\tau_{\delta,+}(x,y)},y_{\tau_{\delta,+}(x,y)})\mid\tau_{\varepsilon,-}(x,y)> \tau_{\delta,+}(x,y)\right],
			\end{multline*}
	from which we obtain a bound 
			\begin{align}
\label{eq:Pleq}
				\mathbb{P} \left(\tau_{\varepsilon,-}(x,y)<  \tau_{\delta,+}(x,y)\right) 
&\geq 
\ddfrac{\ln\left(\frac{d(x,y)}{\delta}\right) +4K}{\ln\left(\frac{\varepsilon}{ \delta}\right)},
			\end{align}
again assuming that $K$ is chosen large enough.
		This finishes the proof of \eqref{eq:propbounds}.
			
			For \eqref{eq:Ebounds} we use the submartingale property of $\eta^+(x_n,y_n) { - nV}$  stopped at $\tau_{\varepsilon,-}(x,y)<  \tau_{\delta,+}(x,y)$ or $\tau_{\varepsilon,-}(x,y)<  \tau_{\delta,+}(x,y)$  (see Remark~\ref{rem:subandsuper}  and apply Doob's stopping time theorem.) and we let time go to infinity. Denoting 
\[
\tau = \min  \{\tau_{\delta,+}(x,y), \tau_{\varepsilon,-}(x,y)\},
\] 
this yields
			\begin{align*}
				\eta^+(x,y) \leq \mathbb{E} \left[\eta^+(x_\tau,y_\tau)  -V \tau \right].
			\end{align*}
			Rewriting and conditioning on $\tau_{\varepsilon,-}(x,y)<  \tau_{\delta,+}(x,y)$ and $\tau_{\varepsilon,-}(x,y)<  \tau_{\delta,+}(x,y)$ separately, we obtain
			\begin{multline}\label{eq:VEleq}
				V\mathbb{E} \left[\tau\right] \leq 
\mathbb{P}\left(\tau_{\varepsilon,-}(x,y)<  \tau_{\delta,+}(x,y)\right)
\mathbb{E} \left[\eta^+(x_\tau,y_\tau)|\tau_{\varepsilon,-}(x,y)<  \tau_{\delta,+}(x,y) \right] 
\\
				+ \mathbb{P}\left(\tau_{\varepsilon,-}(x,y)> \tau_{\delta,+}(x,y)\right)
\mathbb{E} \left[\eta^+(x_\tau,y_\tau)|\tau_{\varepsilon,-}(x,y)> \tau_{\delta,+}(x,y) \right]
				 -  \eta^+(x,y).
			\end{multline}	
As above we have bounds
\begin{align*}
\mathbb{E} \left[\eta^+(x_\tau,y_\tau)|\tau_{\varepsilon,-}(x,y)<  \tau_{\delta,+}(x,y) \right] 
&\leq 
 \ln^2(a_1 \varepsilon) + K \left| \ln(a_1 \varepsilon) \right|
\leq
 \ln^2(\varepsilon) + 3K \left| \ln(\varepsilon) \right| +2K^2,
\\
\mathbb{E} \left[\eta^+(x_\tau,y_\tau)|\tau_{\varepsilon,-}(x,y)> \tau_{\delta,+}(x,y) \right]
&\leq
\ln^2 (\delta) + K \left| \ln(\delta) \right|,
\end{align*}
Now for a different, larger $K$, we have $\mathbb{E} \left[\eta^+(x_\tau,y_\tau)|\tau_{\varepsilon,-}(x,y)<  \tau_{\delta,+}(x,y) \right] 
 \leq \ln^2(\varepsilon)+ K|\log (\varepsilon)|$
and we have
\begin{align*}
-\eta^+(x,y) 
&\leq
-\ln^2(d(x,y)) + K \left| \ln (d(x,y)) \right|.
\end{align*}
Plugging this and \eqref{eq:sumisone},  \eqref{eq:Pleq}, \eqref{eq:eta_log} into  the estimate \eqref{eq:VEleq}, we obtain
	\begin{multline*}				
					V\mathbb{E} \left[\tau\right] \leq  \ddfrac{\ln\left(\frac{d(x,y)}{\delta}\right) +4K}{\ln\left(\frac{\varepsilon}{ \delta}\right)}\left( \ln^2\left({\varepsilon}\right)-\ln^2\left(\delta\right)  + K \left|\ln\left(\frac{\varepsilon}{\delta} \right)\right| \right)\\
				+ \ln^2(\delta) + 2K |\ln(\delta)| - \ln^2(d(x,y)) + 2K |\ln(d(x,y))|.
			\end{multline*}
Note that $\ln^2\left({\varepsilon}\right)-\ln^2\left(\delta\right)  = \ln\left(\frac{\varepsilon}{\delta} \right)\ln(\varepsilon\delta)$ and $\ln^2(\delta) - \ln^2(d(x,y)) = \ln \left(\frac{\delta}{d(x,y)}\right) \ln(d(x,y)\delta)$.
As $|\ln(x)|$ is a decreasing function for $0<x<1$, we have $|\ln(\varepsilon)| < |\ln(d(x,y))| < |\ln(\delta)|$.
Using these identities and estimates we get
	\begin{multline*}
				V\mathbb{E} \left[\tau\right] \leq  \ln\left(\frac{d(x,y)}{\delta}\right)\ln(\varepsilon\delta) - \ln\left(\frac{d(x,y)}{\delta}\right)\ln\left(d(x,y)\delta\right)  +  12K|\ln(\varepsilon)| + 8K^2
\\
				\leq  \ln\left(\frac{d(x,y)}{\delta}\right)\ln\left(\frac{\varepsilon}{d(x,y)}\right) +  12K|\ln(\varepsilon)| + 8K^2.
			\end{multline*}
			 In a similar fashion we get 
	\begin{multline*}				
					V\mathbb{E} \left[\tau\right] \geq  \ddfrac{\ln\left(\frac{d(x,y)}{\delta}\right) -4K}{\ln\left(\frac{\varepsilon}{ \delta}\right)}\left( \ln^2\left(\varepsilon\right)-\ln^2\left(\delta\right)  + 2K \left|\ln\left(\frac{\varepsilon}{\delta} \right)\right| \right)\\
				+ \ln^2(\delta) - 2K |\ln(\delta)| - \ln^2(d(x,y)) + 2K |\ln(d(x,y))|,
			\end{multline*}
and from this,
\[
V\mathbb{E} \left[\tau\right] \geq  \ln\left(\frac{d(x,y)}{\delta}\right)\ln\left(\frac{\varepsilon}{d(x,y)}\right) 
- 12K|\ln(\varepsilon)| - 8K^2.
\]
Replace $2K$ by $K$ to get the statement of the lemma. This finishes the proof.
\end{proof}

From here on,  $R,K$ are fixed as in Lemma~\ref{lem:martingale}. 
We now formulate a result stating that trajectories of points $(x,y) \in \Delta_\delta$ will almost surely escape from $\Delta_\delta$.
However, the expected escape time will be infinite for $d(x,y)$ sufficiently small. 

\begin{lemma}\label{lem:stop}
	Let $R,K$ be as in Lemma~\ref{lem:martingale}. Then for all $(x,y) \in \T^2$, with $0<d(x,y)< \delta <R$,
	\begin{equation}\label{eq:taubnd}
		\mathbb{P}\left(\tau_{\delta,+}(x,y)< \infty \right) = 1
	\end{equation}
	and 
	\[
	\mathbb{E} \left[ \tau_{\delta,+}(x,y)\right] = \infty  \text{ whenever } d(x,y) <  e^{-6K}\delta.
	\]
\end{lemma}

\begin{proof}
    The first statement follows from \eqref{eq:propbounds}. For any $\varepsilon \in (0, d(x,y))$, we have the inclusion of events
    \[
        \{ \tau_{\delta,+}(x,y) \leq \tau_{\varepsilon,-}(x,y) \} \subset \{ \tau_{\delta,+}(x,y) < \infty \}.
    \]
    Therefore,
    \[
        \mathbb{P}(\tau_{\delta,+}(x,y) < \infty) \geq \mathbb{P}(\tau_{\delta,+}(x,y) \leq \tau_{\varepsilon,-}(x,y)) = 1 - \mathbb{P}(\tau_{\varepsilon,-}(x,y) < \tau_{\delta,+}(x,y)).
    \]
    Recall the upper bound in \eqref{eq:propbounds}, 
    \[
        \mathbb{P}(\tau_{\varepsilon,-}(x,y) < \tau_{\delta,+}(x,y)) \leq \frac{\ln\left( \frac{\delta}{d(x,y)}\right)+2K}{\ln\left(\frac{\delta}{\varepsilon}\right)}.
    \]
    As $\varepsilon \to 0$, the denominator goes to infinity while the numerator remains fixed. Thus the term on the right hand side goes to 0, implying $\mathbb{P}(\tau_{\delta,+}(x,y) < \infty) = 1$.

    For the second statement, observe that for any $0 < \varepsilon < d(x,y)$,
    \[
        \tau_{\delta,+}(x,y) \geq \min\{\tau_{\varepsilon,-}(x,y), \tau_{\delta,+}(x,y)\}.
    \]
    Taking expectations and applying the lower bound from \eqref{eq:Ebounds} in Lemma \ref{lem:martingale}, we find
    \begin{align*}
        V \mathbb{E}[\tau_{\delta,+}(x,y)] &\geq V \mathbb{E}[\min\{\tau_{\varepsilon,-}(x,y), \tau_{\delta,+}(x,y)\}] \\
        &\geq \ln\left( \frac{\delta}{d(x,y)} \right) \ln\left( \frac{d(x,y)}{\varepsilon} \right) - 6K |\ln (\varepsilon)| - 2K^2.
    \end{align*}
    Recalling that for small $\varepsilon$, $\ln(d(x,y)/\varepsilon) = \ln(d(x,y)) + |\ln(\varepsilon)|$, we group the terms involving $|\ln(\varepsilon)|$:
    \begin{align*}
         V \mathbb{E}[\tau_{\delta,+}(x,y)] &\geq \ln\left( \frac{\delta}{d(x,y)} \right) \left( \ln(d(x,y)) + |\ln(\varepsilon)| \right) - 6K |\ln (\varepsilon)| - 2K^2 \\
         &= |\ln(\varepsilon)| \left[ \ln\left( \frac{\delta}{d(x,y)} \right) - 6K \right] + \ln\left( \frac{\delta}{d(x,y)} \right)\ln(d(x,y)) - 2K^2.
    \end{align*}
The hypothesis  $d(x,y) < e^{-6K}\delta$ gives $\ln\left( \frac{\delta}{d(x,y)} \right) - 6K > 0.$
Taking the limit $\varepsilon \to 0$ and applying the monotone convergence theorem,  yields
    \[
        \mathbb{E}[\tau_{\delta,+}(x,y)] = \lim_{\varepsilon \to 0} \mathbb{E}[\tau_{\delta,+}(x,y)\cdot \mathbbm{1}_{\{\tau_{\delta,+}(x,y)<\tau_{\epsilon,-}(x,y)\}}(\omega)] = \infty.
    \]
\end{proof}


From this lemma we know that for $(x_0,y_0) \in \Delta_\delta$, the expected number of orbit points 
\[(x_n,y_n) = \left( T^{(2)}_\omega \right)^n (x_0,y_0)
\] 
before escaping $\Delta_\delta$, with $(x_0,y_0) \in \Delta_{ e^{-6K}\delta}$,
is infinite.

The above results allow to conclude Theorem~\ref{thm:main1}\eqref{i:thm1:iter}.

\begin{proposition}\label{prop:SynchronOnAverage}
    If $\lambda=0$, 				
		for all $(x,y) \in \mathbb{T}^2\setminus \Delta$, 
		\[
			\lim_{n\to\infty} \frac{1}{n}\sum_{i=0}^{n-1}    d\left(T^i_\omega (x) ,T^i_\omega(y)\right)= 0
		~\textrm{and}~ 
			\limsup_{n\to\infty}d\left(T^n_\omega (x), T^n_\omega(y)\right) > 0,~\mathbb{P}-\mathrm{a.s.}
		\]
\end{proposition}

\begin{proof}
    Let $\varepsilon>0$. For $(x,y) \in \mathbb{T}^2\setminus \Delta$ and $\omega \in \Sigma_\vartheta$, 
    consider the empirical count of iterates in $\Delta_\varepsilon$,
    \[
    N_\varepsilon(x,y,\omega) = \lim_{n\to\infty}\frac{\#\{i\in\N, 0\leq i \leq n-1: d(T_\omega^i(x),T_\omega^i(y))<\varepsilon \}}{n}.
    \]
{ {
We follow reasoning in  \cite{MR2033186}, adapting to our setting.
Let $(x_0,y_0)$ be a point outside $\Delta_\varepsilon$  and consider the orbit $(x_i,y_i) = \left(T_\omega^i (x_0), T_\omega^i (y_0)\right)$.
For almost all $\omega \in \Sigma_\vartheta$ there is an infinite sequence of times $S_i$ and $T_i$, $i \geq 0$, with $0 = S_0 < T_0 < S_1 < T_1 < \cdots$, such that for $j \geq 0$,
\begin{align*}
T_j  &= \min \left\{i > S_j \; : \; d(x_i,y_i) < \varepsilon e^{-6K} \right\},
\\
S_{j+1} &= \min \left\{i >   T_j  \; : \;  d(x_i,y_i) > \varepsilon  \right\}.
\end{align*}
We can also assume that $d(x_i,y_i)$ is never exactly $\varepsilon$, as this holds for almost all $\omega \in \Sigma_\vartheta$.
Write 
\[
\sigma_{j+1} = T_j - S_j, \qquad \tau_{j+1} = S_{j+1} - T_j.
\]
Given $(x_0,y_0)$ outside $\Delta_\varepsilon$, write
\[
 N_\varepsilon(n,\omega) = \frac{\#\{i\in\N, 0\leq i \leq n-1: d(x_i,y_i)<\varepsilon \}}{n}.
\]
Let $N_\sigma (n)$ be the largest index $i$ with $S_i \leq n$ and let $N_\tau (n)$ be the largest index $i$ with $T_i < n$.
Consider, for definiteness,  the case that $d(x_n,y_n) > \varepsilon$. 
Calculate
 \begin{align*}
N_\varepsilon(n,\omega) &\geq \frac{1}{T_{N_\tau (n)}} \sum_{i=1}^{T_{N_{\tau (n)}-1}}  \tau_i / (\sigma_i + \tau_i) 
\end{align*}
The dependence of, for instance, $\sigma_i$ on the orbit point $(x_{S_i}, y_{S_i})$, makes that we can not apply 
the strong law of large numbers for independent random variables. However, the arguments that lead to the strong law of large numbers, see for instance \cite[Section~2.5]{MR3930614} and \cite{MR722846,MR1192096}, combined with
the uniform bounds on the expected escape times $\sigma_i$ do give that for some $K>0$,  for almost all $\omega \in \Sigma_\vartheta$,
\begin{align*}
\limsup_{n\to\infty} \frac{1}{n} \sum_{i=0}^{n-1} \sigma_i &\leq K.
\end{align*}
Likewise,
\begin{align*}
\liminf_{n\to\infty} \frac{1}{n} \sum_{i=0}^{n-1} \tau_i &= \infty,
\end{align*}
for almost all $\omega \in \Sigma_\vartheta$. The case that $d(x_n,y_n) > \varepsilon$ is treated similarly.
We get that $\mathbb{P}$-almost surely $N_\varepsilon(x,y,\omega)=1$, for all $\varepsilon>0$ implying the first part of the statement. 
}}

     The second part follows from Lemma \ref{lem:proximal}. 
\end{proof}

\subsection{Positive Lyapunov exponent}\label{s:pos}

In this section we consider positive Lyapunov exponent $\lambda > 0$. This will be assumed to hold throughout the section. Furthermore, we will denote the second zero of the moment Lyapunov function $\Lambda$ as $\gamma$.

Consider  $0 < \delta < R$ and $(x,y) \in \mathbb{T}^2\setminus \Delta$ with $d(x,y) < \delta$. 
Let $\tau (x,y,\omega)$ be the minimal time with 
\[
d\left( \left(T^{(2)}_\omega\right)^{\tau(x,y,\omega)} (x,y)  \right) > \delta,
\] 
with $\tau (x,y,\omega)=\infty$ if this does not exist. 

\begin{lemma}\label{lem:returntime>}
Let $R,K$ be as in Lemma~\ref{lem:martingale}. Suppose $\lambda > 0$. For $0 < \delta < R$ and $(x,y) \in \mathbb{T}^2\setminus \Delta$ with $d(x,y) < \delta$ we have
\[
\mathbb{P} \left(  d(T^n_\omega(x), T^n_\omega(y)) > \delta  \text{ for some } n \in \mathbb{N} \right) = 1.
\]
Moreover, for some $K>0$,
\begin{equation}\label{eq:escape>}
 \frac{1}{\lambda}  \left(  \ln \left( \frac{\delta}{d(x,y)} \right) - 2K \right)  \leq \mathbb{E}\left[ \tau (x,y,\omega)  \right]  \leq \frac{1}{\lambda}  \left(  \ln \left( \frac{\delta}{d(x,y)}\right) + 2K \right).
\end{equation}
\end{lemma}

\begin{proof}
Write $x_n = T^n_\omega (x)$ and $y_n = T^n_\omega (x)$. We also use shorthand notation $\tau$ for $\tau(x,y,\omega)$.
By Proposition~\ref{prop:phieta} and Remark~\ref{rem:subandsuper} we find, applying Doob's stopping time theorem,
\begin{align*}
\lambda \mathbb{E} [\tau ] \leq \mathbb{E} \left[  \phi^+  (x_\tau,y_\tau) - \phi^+ (x,y)   \right]
\leq \ln (\delta) - \ln(d(x,y)) + 2K 
\end{align*}
and 
\begin{align*}
	\lambda \mathbb{E} [\tau ] \geq \mathbb{E} \left[  \phi^-  (x_\tau,y_\tau) - \phi^- (x,y)   \right]
	\geq \ln (\delta) - \ln(d(x,y)) - 2K.
\end{align*}
The lemma follows.
\end{proof}

More detailed estimates can be obtained by analysing escape times from  strips $\Delta_\delta \setminus \Delta_\varepsilon$. The next lemma is the equivalent of Lemma~\ref{lem:martingale} for positive Lyapunov exponent.
Stopping times $\tau_{\delta,+}(x,y)$ and $\tau_{\varepsilon,-}(x,y)$ are defined as before in
\eqref{eq:tau+}, \eqref{eq:tau-}.

 \begin{lemma}\label{lem:martingalepos}
Let $R,K$ be as in Lemma~\ref{lem:martingale}. Suppose $\lambda > 0$.
Let $\gamma\neq 0$ be the negative value so that $\Lambda(\gamma) = 0$.  If $0<\varepsilon<d(x,y)<\delta< R$, then
 	\begin{equation}\label{eq:asfinpos}
 		\mathbb{P}(\min\{\tau_{\varepsilon,-}(x,y), \tau_{\delta,+}(x,y)\}< \infty) = 1. 
 	\end{equation}
 	Furthermore, there exists a $\kappa \in (0,1)$, such that if $0<\varepsilon<d(x,y)<\kappa\delta<\kappa R$, then
 	\begin{equation} \label{eq:propbounds>}
 		\frac{1}{K}\left(\frac{d(x,y)}{\varepsilon}\right)^{\gamma} \leq \mathbb{P}\left(\{\tau_{\varepsilon,-}(x,y)< \tau_{\delta,+}(x,y)\}\right) \leq {K}\left(\frac{d(x,y)}{\varepsilon}\right)^{\gamma}. 
 	\end{equation}
 \end{lemma}
 
 \begin{proof}
The proof is similar to the proof of Lemma~\ref{lem:martingale}. 
Let $(x,y) \in \Delta_{\delta}\setminus \Delta_\varepsilon$, so with $\varepsilon < d(x,y) < \delta$.
Let $\tau_{\varepsilon,-}$ and $\tau_{\delta,+}$ as in \eqref{eq:tau+} and \eqref{eq:tau-}. As before we get \eqref{eq:asfinpos}, that is, $\min\{\tau_{\varepsilon,-},\tau_{\delta,+}\}<\infty$ for almost all $\omega$.

We start with the upper bound in \eqref{eq:propbounds>}.
Using Remark~\ref{rem:Wpmsubsuper} and applying Doob's stopping time theorem, we get
			\begin{multline}
				W^-_{-\gamma}(x,y) \geq 
\\
				 \mathbb{P} \left(\tau_{\varepsilon,-}(x,y)<  \tau_{\delta,+}(x,y)\right)\mathbb{E} \left[ W^-_{-\gamma}(x_{\tau_{\varepsilon,-}(x,y)},y_{\tau_{\varepsilon,-}(x,y)})\mid\tau_{\varepsilon,-}(x,y)<  \tau_{\delta,+}(x,y)\right]
\\
				+ \mathbb{P} \left(\tau_{\varepsilon,-}(x,y)>  \tau_{\delta,+}(x,y)\right)\mathbb{E} \left[ W^-_{-\gamma}(x_{\tau_{\delta,+}(x,y)},y_{\tau_{\delta,+}(x,y)})\mid\tau_{\varepsilon,-}(x,y)> \tau_{\delta,+}(x,y)\right].
 \label{eq:W-stop>}
			\end{multline}
As before, since $\min\{\tau_{\varepsilon,-},\delta\}$ is almost surely finite (see Lemma~\ref{lem:martingale}), we have
\begin{align*}
\mathbb{P} \left(\tau_{\varepsilon,-}(x,y)>  \tau_{\delta,+}(x,y)\right) 
&=
 1 - \mathbb{P} \left(\tau_{\varepsilon,-}(x,y)<  \tau_{\delta,+}(x,y)\right).
\end{align*}
Furthermore, for some $K>1$ as in Proposition~\ref{prop:W}, we get
\begin{align*}
\mathbb{E} \left[ W^-_{-\gamma}(x_{\tau_{\varepsilon,-}(x,y)},y_{\tau_{\varepsilon,-}(x,y)})\mid\tau_{\varepsilon,-}(x,y)<  \tau_{\delta,+}(x,y)\right] &\geq  \frac{\varepsilon^{\gamma}}{K},
\\
\mathbb{E} \left[ W^-_{-\gamma}(x_{\tau_{\delta,+}(x,y)},y_{\tau_{\delta,+}(x,y)})\mid\tau_{\varepsilon,-}(x,y)> \tau_{\delta,+}(x,y)\right] &\geq \frac{\delta^{\gamma}}{K},
\end{align*}
and also 
$K d(x,y)^{\gamma} \geq  W^-_{-\gamma}(x,y)$.
Then \eqref{eq:W-stop>} yields
\begin{align*}  
 \mathbb{P} \left(\tau_{\varepsilon,-}(x,y) <  \tau_{\delta,+}(x,y)\right)
&\leq  \left( {K^2} d(x,y)^{\gamma} - \delta^{\gamma} \right) / \left(  \varepsilon^{\gamma} -  \delta^{\gamma}  \right).
\end{align*}
From this and a condition $d(x,y) \leq \kappa\delta$ for small enough $\kappa$, we get the stated bound. The lower bound in \eqref{eq:propbounds>} is derived similarly.
 \end{proof}
 
Using the above lemma we get the following statement in which we discuss, for orbit pieces starting in 
$\Delta_{\kappa\delta}\setminus\Delta_\varepsilon$ until escape from $\Delta_\delta$, the expected number of points in $\Delta_\varepsilon$. It is similar to Lemma~\ref{lem:gepsilonnetje}, but for positive Lyapunov exponent.
 
The results in this section easily allow to conclude Theorem~\ref{thm:main1}\eqref{i:thm1:chao}.

\begin{proposition}\label{prop:chaos}
If $\lambda > 0$, for all $(x,y) \in \mathbb{T}^2\setminus \Delta$, 
\[
\lim_{n\to\infty}\frac{1}{n}\sum_{i=0}^{n-1}   d \left(T^i_\omega (x), T^i_\omega(y)\right) > 0
~\textrm{and}~
\liminf_{n\to\infty}d\left(T^n_\omega (x) ,T^n_\omega(y)\right) = 0,~\mathbb{P}-\mathrm{a.s.}
\]	
\end{proposition}

\begin{proof}
The reasoning follows Proposition~\ref{prop:SynchronOnAverage}.
    Let $\varepsilon>0$. For $(x,y) \in \mathbb{T}^2\setminus \Delta$ and $\omega \in \Sigma_\vartheta$, 
    consider the empirical count of iterates in $\Delta_\varepsilon$,
    \[
    N_\varepsilon(x,y,\omega) = \lim_{n\to\infty}\frac{\#\{i\in\N, 0\leq i \leq n-1: d(T_\omega^i(x),T_\omega^i(y))<\varepsilon \}}{n}.
    \]
      By Lemma~\ref{lem:returntime>} and Lemma~\ref{lem:logbnd} and the strong law of large numbers, we get that $\mathbb{P}$-almost surely $N_\varepsilon(x,y,\omega)<1 $, for a $\varepsilon>0$ implying the first part of the statement.
     The second part follows from Lemma \ref{lem:proximal}. 
\end{proof}

\subsection{Negative Lyapunov exponent}\label{s:neg}

In this section we consider negative Lyapunov exponent $\lambda<0$. This will be assumed to hold throughout the section.
The following lemma can be obtained by the construction of local stable manifolds for $T^n_\omega$ of points in $\Delta$, which exists for almost all $\omega \in \Sigma_\vartheta$ \cite{MR1369243}. We provide a proof along the lines of our reasoning in previous sections, compare also \cite{MR872098}.

\begin{lemma}\label{lem:stablesetposprob}
Let $R,K$ be as in Lemma~\ref{lem:martingale}. Suppose $\lambda < 0$. 
There is $0 < \chi <1$ so that the following holds.
For each $0< \delta < R$ there is $0< \delta' < \delta$, so that for $(x,y)$ with $d(x,y) < \delta'$, 
\begin{align*}
\mathbb{P} \left(  d( T^n_\omega (x) , T^n_\omega (y)) < \delta \text{ for all } n \text{ and } \lim_{n\to \infty} d( T^n_\omega (x) , T^n_\omega (y))   = 0   \right) &> \chi.
\end{align*}
\end{lemma}

\begin{proof}
Let $\gamma>0$ given by Lemma~\ref{lem:second0} be so that $\Lambda(\gamma)=0$.
The first part of the proof is similar to the proof of Lemma~\ref{lem:martingale} and Lemma~\ref{lem:martingalepos}. 
Take $(x,y) \in \Delta_{\delta}\setminus \Delta_\varepsilon$, so with $\varepsilon \leq d(x,y) < \delta$.
Let $\tau_{\varepsilon,-}(x,y)$ and $\tau_{\delta,+}(x,y)$ as in \eqref{eq:tau+} and \eqref{eq:tau-}. As before we get \eqref{eq:asfinpos}, that is, $\min\{\tau_{\varepsilon,-}(x,y),\tau_{\delta,+}(x,y)\}<\infty$ for almost all $\omega$. Thus, as in Lemma~\ref{lem:martingale},
\begin{align}\label{eq:sumisone(b)}
\mathbb{P} \left(\tau_{\varepsilon,-}(x,y)>  \tau_{\delta,+}(x,y)\right) 
&=
 1 - \mathbb{P} \left(\tau_{\varepsilon,-}(x,y)<  \tau_{\delta,+}(x,y)\right).
\end{align}

Using Remark~\ref{rem:Wpmsubsuper} and applying Doob's stopping time theorem, we get 
			\begin{multline}
				W^-_{-\gamma}(x,y) \geq 
\\
				 \mathbb{P} \left(\tau_{\varepsilon,-}(x,y)<  \tau_{\delta,+}(x,y)\right)\mathbb{E} \left[ W^-_{-\gamma}(x_{\tau_{\varepsilon,-}(x,y)},y_{\tau_{\varepsilon,-}(x,y)})\mid\tau_{\varepsilon,-}(x,y)<  \tau_{\delta,+}(x,y)\right]
\\
				+ \mathbb{P} \left(\tau_{\varepsilon,-}(x,y)>  \tau_{\delta,+}(x,y)\right)\mathbb{E} \left[ W^-_{-\gamma}(x_{\tau_{\delta,+}(x,y)},y_{\tau_{\delta,+}(x,y)})\mid\tau_{\varepsilon,-}(x,y)> \tau_{\delta,+}(x,y)\right].
 \label{eq:W-stop}
			\end{multline}
For some $K>1$ as in Proposition~\ref{prop:W}, we get
\begin{align*}
\mathbb{E} \left[ W^-_{-\gamma}(x_{\tau_{\varepsilon,-}(x,y)},y_{\tau_{\varepsilon,-}(x,y)})\mid\tau_{\varepsilon,-}(x,y)<  \tau_{\delta,+}(x,y)\right] &\geq\frac{\varepsilon^\gamma}{K},
\\
\mathbb{E} \left[ W^-_{-\gamma}(x_{\tau_{\delta,+}(x,y)},y_{\tau_{\delta,+}(x,y)})\mid\tau_{\varepsilon,-}(x,y)> \tau_{\delta,+}(x,y)\right] &\geq\frac{\delta^\gamma}{K},
\end{align*}
and also 
${K} d(x,y)^{\gamma} \geq  W^-_{-\gamma}(x,y)$.
 Using \eqref{eq:sumisone(b)}, \eqref{eq:W-stop} yields 
\begin{align}
\lim_{\varepsilon\to 0} \mathbb{P} \left(\tau_{\varepsilon,-}(x,y)<  \tau_{\delta,+}(x,y)\right)&\geq \lim_{\varepsilon\to 0} \left( {K^2} d(x,y)^\gamma - \delta^\gamma \right) / \left(  \varepsilon^\gamma -  \delta^\gamma  \right)
\nonumber \\ 
&= 1 - {K^2} \left( \frac{d(x,y) }{ \delta} \right)^\gamma.
\label{eq:Plowerbound}
\end{align}
The limit exists because the probability is monotone decreasing as $\varepsilon \to 0$. This lower bound for $\mathbb{P}$ is strictly positive if $d(x,y) / \delta$ is small enough.
The computation means that there is a strictly positive probability 
\[
\mathbb{P} \left(  d(x_n,y_n) < \delta \text{ for all } n\in\mathbb{N} \text{ and } \liminf_{n\to\infty} d(x_n,y_n) = 0   \right) \geq 1 - {K^2} \left( \frac{d(x,y) }{ \delta} \right)^\gamma,
\]
at least for $d(x,y)/\delta$ small enough.
Note that $\mathbb{P}$ goes to one if $d(x,y) / \delta$ goes to zero.

We obtain the lemma from the observation that the above argument applies to any $\delta$.
To finish the argument, take $0 < d_2 < \delta_2 < d(x,y) < \delta$ and, using \eqref{eq:Plowerbound}, consider 
 \begin{multline*}
\mathbb{P} \left(\tau_{d_2,-}(x,y) <  \tau_{\delta,+}(x,y)  \text{ and } d(x_i,y_i)< \delta_2 \text{ for all } i \geq  \tau_{d_2,-}(x,y)   \right)
\\
\geq 
\left(
1 - {K^2} \left( \frac{d(x,y) }{ \delta} \right)^\gamma 
\right) 
\left( 
1 - {K^2} \left( \frac{d_2}{ \delta_2} \right)^\gamma 
\right).  
\end{multline*}
The lemma follows by taking $d_2$ and $\delta_2$ to zero, with $d_2/\delta_2$ small enough, and noting that the bound on the right hand side stays strictly positive.
\end{proof}

Synchronisation of typical orbits expressed by Theorem~\ref{thm:main1}\ref{i:thm1:syn} is a consequence of the above lemma and our hypotheses on the random dynamical system.

\begin{proposition}\label{prop:synchron}
Suppose $\lambda < 0$. 
For all $x,y \in \mathbb{T}$, for $\mathbb{P}$-almost all $\omega \in {\Sigma_\vartheta}$,		
\[
		\lim_{n\to\infty} d(T^n_\omega (x) ,T^n_\omega(y)) = 0.
\]
\end{proposition}

\begin{proof}
Take $0 < \delta' < \delta$ as in Lemma~\ref{lem:stablesetposprob}.
By Lemma~\ref{lem:proximal} there is a strictly positive probability for an orbit $T^n_\omega (x,y)$ to enter $\Delta_{\delta'}$ in finitely many steps.
That is, there exists $C>0$ and $N>0$ so that for $(x,y) \in \mathbb{T}^2 \setminus \Delta_{\delta'}$,
\[
\mathbb{P} \left(  (T^n_\omega (x) , T^n_\omega (y)) \in \Delta_{\delta'} \text{ for some } 0 < n < N  \right) > C.
\]
We find that for $(x,y) \in \mathbb{T}^2 \setminus \Delta_{\delta'}$, 
\[
\mathbb{P} \left(  (T^n_\omega (x) , T^n_\omega (y)) \not \in \Delta_{\delta'} \text{ for all } 0 \leq n < kN  \right) \geq (1-C)^k,
\]
so that $\mathbb{P} \left(  (T^n_\omega (x) , T^n_\omega (y)) \not \in \Delta_{\delta'} \text{ for all } n  \right) = 0$. There are therefore almost surely infinitely many orbit points from 
$(T^n_\omega (x) , T^n_\omega (y))$ in $\Delta_{\delta'}$.

As before, for $0 < \varepsilon < \delta$ and $(x,y) \in \mathbb{T}^2$ with $\varepsilon < d(x,y) < \delta$,
$(T^n_\omega (x) , T^n_\omega (y))$ will be outside $\Delta_{\delta} \setminus \Delta_{\varepsilon}$
for some $n >0$, almost surely. Combined with the above we see that almost surely, orbit points
$(T^n_\omega (x) , T^n_\omega (y))$ are in $\Delta_\varepsilon$ infinitely often. This holds for any $\varepsilon>0$, so that
\[
\liminf_{n\to \infty} d(T^n_\omega (x), T^n_\omega (y)) = 0
\]
almost surely.
That is, given a sequence $\varepsilon_k$ that is decreasing to zero, there is a sequence of first iterates 
$(T^{n_k}_\omega (x) , T^{n_k}_\omega (y))$ inside $\Delta_{\varepsilon_k}$.
Lemma~\ref{lem:stablesetposprob} now implies synchronisation.
\end{proof}

\section{Stationary measures for the two-point motion}\label{s:stat}

This section contains both the construction of stationary measures on $\mathbb{T}^2\setminus \Delta$ for the random two-point maps, in the case of zero and positive Lyapunov exponent, 
and the derivation of their asymptotics at the diagonal. The results apply to zero and positive Lyapunov exponent.

\subsection{Construction by inducing}\label{s:inducing}

We construct a stationary Radon measure on $\mathbb{T}^2 \setminus \Delta$ for the two-point random dynamical system with Lyapunov exponent greater than or equal to zero.
We do this by inducing: we construct a stationary measure for a first return map on a domain away from the diagonal $\Delta$ by the Krylov-Bogolyubov method. 
Following work by 
Deroin, Kleptsyn, Navas and Parwani \cite{MR3098067}, who study random walks on the real line, 
we introduce random stopping
times in order to be able to apply a Krylov-Bogolyubov method to find stationary measures.

A stationary measure is obtained as usual by pushing forward the stationary measure for the first return map.
For a zero Lyapunov exponent, Lemma~\ref{lem:stop} implies that it takes in expectation infinite time to get away from the diagonal and therefore the stationary measure is not finite. 
For a positive Lyapunov exponent, a stationary measure for the two-point motion can be constructed just as in the case of zero Lyapunov exponent. The finite expectation of the escape time from $\Delta_\delta$ as expressed by Lemma~\ref{lem:returntime>}, shows that the total measure is finite. Normalizing the measure, we derive the existence of a stationary probability measure. 

{ {An alternative way to construct stationary measures for the positive Lyapunov exponent  case $\lambda>0$, might be through a Lyapunov function as in  \cite{PierreHumbert} for random diffeomorphisms. Our case is different due to the possibility that points  $(x,y)\in\mathbb{T}^2$ can be mapped onto the diagonal $\Delta$ by $T^{(2)}_a$.
}}

}

\begin{proposition}\label{prop:infmeas}
	The random dynamical system \eqref{eq:rds2} has the following properties:
	\begin{itemize}
		\item if $\lambda=0$, then the two-point motion	admits an infinite stationary Radon measure $\mu^{(2)}$ on $\mathbb{T}^2 \setminus \Delta$, and 
		\item if $\lambda>0$, then the two-point motion	admits a stationary probability measure $\mu^{(2)}$ on $\mathbb{T}^2 \setminus \Delta$.
	\end{itemize} 
\end{proposition}

\begin{proof}
	For a fixed small $\varepsilon>0$ take the compact set
	\begin{align*} 
		\mathcal{K} &= \mathbb{T}^2 \setminus \Delta_\varepsilon.
	\end{align*} 
	Because $T$ is of degree two, each element in $\T$ has two distinct pre-images. The minimal distance between pre-images is smaller than $R_{\min}$ (recall \eqref{eq:Rmin}).  Take $\varepsilon< R_{\min}$,  so that
	points in $\T^2 \setminus (\mathcal{K} \cup \Delta)$ can not be mapped into $\Delta$ by $T^{(2)}_a$. 
	Note that there will be points in $\mathcal{K}$ that are mapped into $\Delta$ by some $T^{(2)}_a$. Namely,
	any $(x,y) \in \T^2$ with $x\ne y$ and $T_a (x) = T_a(y)$ lies in $\mathcal{K}$ and will be mapped into $\Delta$ by  $T^{(2)}_a$.
	By Lemma~\ref{lem:stop}, for any $(x_0,y_0)\in \mathcal{K}$, of 
	the random orbit { $(x_n,y_n) = \left(T^{(2)}_\omega \right)^n(x_0,y_0)$}
	has almost surely infinitely many points contained in $\mathcal{K}$. 
	
	Fix a smooth function $\xi: \T^2 \to [0,1]$ with support 
	disjoint from $\Delta$ and with $\xi \equiv 1$ on $\mathcal{K}$. For an initial point $(x_0,y_0)$ 
	consider the random stopping time $V(\omega) \geq 1$ (suppressing dependence on $(x_0,y_0)$) so that
	the probability $\mathbb{P} (V = n+1 \; | \; V \geq n)$ equals $\xi (x_{n+1},y_{n+1})$. So when the random orbit arrives at
	$(x_{n+1},y_{n+1})$ we stop with probability  $\xi (x_{n+1},y_{n+1})$ and continue with probability $1 - \xi (x_{n+1},y_{n+1})$. 
	
	We use $\mathbb{E}$ to denote the expectation both over ${\Sigma_\vartheta}$ and over  
	the random process defining the random stopping time.
	So $\mathbb{E} \left[ \delta_{(x_V ,y_V)} \right]$ is the distribution of the stopped point $(x_V,y_V)$. It is an element of the space $\mathcal{P} (\mathrm{supp}\, \xi)$  of probability measures on $ \mathrm{supp}\, \xi$, which we endow with the weak star topology.
	
	We claim that this distribution depends continuously on $(x_0,y_0)\in \mathcal{K}$ in the weak star topology.
	Namely, consider a sequence of distributions of stopped points for a converging sequence of initial points $(x_0^n,y_0^n)$. 
	From \eqref{eq:taubnd} we get that for $\zeta>0$ small there exists $N>0$ so that with probability at least $1 - \zeta$,
	the stopping time $V$ is smaller than $N$. As the maps $T_a$ depend continuously on $a$, the points $(x_V,y_V)$, $V<N$, for $(x_0^n,y_0^n)$ are close to those for $(x_0,y_0)$.
	Consequently,   the distribution $\mathbb{E} \left[ \delta_{(x_V ,y_V )} \right]$ of the stopped point
	depends continuously on $(x_0,y_0)$, for $(x_0,y_0) \in \mathcal{K}$, in the weak star topology.
	
	Define 
	\begin{align}\label{eq:Pxi}
		P_\xi \mu &= 
		\int_{\T^2} \mathbb{E}  \left[ \delta_{(x_V,y_V)} \right] \, d\mu (x_0,y_0)  
	\end{align}
	acting on $\mathcal{P} (\mathrm{supp}\, \xi)$.
	Then \eqref{eq:Pxi} is a continuous map from 
	$\mathcal{P} (\mathrm{supp}\, \xi)$ to itself.
	We can therefore apply the Krylov-Bogolyubov procedure of taking a converging subsequence of C\'esaro averages, to find a $P_\xi$ invariant probability measure $\varsigma_0$,
	\[
	P_\xi \varsigma_0 = \varsigma_0,
	\] 
	with ${ \text{supp}\,} \varsigma_0 \subset \text{supp}\, \xi$.
	
	We will use this to construct a stationary Radon measure on $\mathbb{T}^2 \setminus \Delta$.
	Consider the average measures 
	\begin{align}\label{eq:mx0y0}
		\overline{m}_{(x_0,y_0)} &\coloneqq \mathbb{E}  \left[  \sum_{j=0}^{V(\omega)-1} \delta_{(x_j,y_j)} \right].
	\end{align}
	Because $1 - \xi$ vanishes on $\mathcal{K}$, we can write , with $(x_j,y_j) = T^{(2)}_{\omega_{j-1}} \circ \cdots \circ T^{(2)}_{\omega_0} (x_0,y_0)$,
	\begin{align}\label{eq:formformbar}
		\overline{m}_{(x_0,y_0)} &= \sum_{n=0}^\infty \,\,\,\iint\displaylimits_{\overset{\omega_0,\ldots,\omega_{n-1}}{\in \T}}\, \prod_{j=1}^n  \left( 1 - \xi (x_j,y_j) \right) \delta_{(x_n,y_n)} \, d\mathbb{P}(\omega_0) \cdots d\mathbb{P}(\omega_{n-1}).   
	\end{align}
	Integrate over $(x_0,y_0)$ taken from the measure $\varsigma_0$ to obtain 
	\begin{align}\label{eq:defofm2}
		\mu^{(2)} &\coloneqq \int_{\T^2} \overline{m}_{(x_0,y_0)} \, d\varsigma_0 (x_0,y_0).
	\end{align}
	Written out this reads
	\begin{align*}
		\mu^{(2)}(A) &= \int_{\T^2} \mathbb{E}  \left[\sum_{j=0}^{V(\omega)-1} \mathbbm 1_A(T^j_\omega(x),T^j_\omega(y))\right] \, d\varsigma_0 (x_0,y_0)
	\end{align*}
	for Borel sets $A \subset \mathbb{T}^2 \setminus \Delta$.

	We claim that for any continuous function $\psi: \T^2 \to \mathbb{R}$ with support disjoint from $\Delta$,
	\begin{align}\label{eq:integral}
		\int_{\T^2} \psi \, d m &= \int_{\T^2} \mathbb{E}  \left[ \sum_{j=0}^{V(\omega)-1} \psi (x_j)\right] \, d\varsigma_0 (x_0,y_0)
	\end{align}
	is well defined and finite.
	We conclude from this that $\mu^{(2)}$ is a Radon measure that assigns finite measure to compact sets disjoint from $\Delta$.
	To establish the claim 
	note that there are $N>0$ and $q >0$ so that with probability at least $q$ an orbit starting in $\text{supp}\, \psi$ hits $\mathcal{K}$ in at most $N$ steps. 
	Note that iterates outside $\text{supp}\, \psi$ do not contribute 
	to the right hand side of \eqref{eq:integral}.
	We find an estimate 
	\[
	\mathbb{E} \left[\sum_{j=0}^{V(\omega)-1} \psi (x_j,y_j) \right]
	< \sum_{i=1}^\infty \max_{(x,y)\in \T^2} |\psi(x,y)| N  (1-q)^{i-1} < \infty.
	\]
	That is, $\mathbb{E} \left[ \sum_{j=0}^{V(\omega)-1} |\psi (x_j,y_j)| \right]$ is finite 
	and bounded uniformly on $\text{supp}\,\psi$.
	This implies that the right hand side of \eqref{eq:integral} is finite.
	
	We will establish that $\mu^{(2)}$ is stationary for the two-point maps.
	We must thus show $P \mu^{(2)} = \mu^{(2)}$ with
	\[
	P \mu^{(2)} \coloneqq
	\int_{\T}   \left( T_a^{(2)}\right)_\ast  \mu^{(2)} \, d\mathbb{P} (a).
	\]
	Applying $P$ yields, with $(x_j,y_j) = T^{(2)}_{\omega_{j-1}} \circ \cdots \circ T^{(2)}_{\omega_0} (x_0,y_0)$, 
	\begin{flalign*}
		P \overline{m}_{x_0,y_0} &=	
		\int_{\T} \left( T^{(2)}_a \right)_\ast  \overline{m}_{x_0,y_0} \, d\mathbb{P} (a)
		\\
		&= \sum_{n=0}^\infty \,\,\, \iint\displaylimits_{\overset{a,\omega_0,\ldots,\omega_{n-1}}{ \in \T}} \,
		\prod_{j=1}^n \left(1 - \xi( x_j,y_j) \right)
		\left( T^{(2)}_a \right)_\ast  \delta_{(x_n,y_n)}  \, d\mathbb{P}(a) d\mathbb{P}(\omega_0) \cdots d\mathbb{P}(\omega_{n-1}) 
		\\
		&= \sum_{n=0}^\infty \,\,\,\iint\displaylimits_{\overset{\omega_0,\ldots,\omega_{n}}{\in \T}} \,
		\prod_{j=1}^n \left(1 - \xi(x_j,y_j) \right) 
		\delta_{(x_{n+1},y_{n+1})} \, d\mathbb{P}(\omega_0) d\mathbb{P}(\omega_1) \cdots d\mathbb{P}(\omega_{n}) 
		\\
		&= \mathbb{E}   \left[ \sum_{j=1}^{V(\omega)} \delta_{(x_j,y_j)} \right]
	\end{flalign*}
	(again using that $1-\xi$ vanishes on $\mathcal{K}$).
	
	Compared to \eqref{eq:mx0y0}, the index $j$ is counting from $1$ to $V(\omega)$ instead of from $0$ to $V(\omega)-1$.
	We thus find
	\begin{align*}
		P  \overline{m}_{(x_0,y_0)} 
		&=
		\overline{m}_{(x_0,y_0)}  - \delta_{(x_0,y_0)} + \mathbb{E} \left[ \delta_{(x_{V(\omega)},y_{V(\omega)})}\right].
	\end{align*}	
	Integration over $\varsigma_0$ yields
	\begin{align*} P \mu^{(2)} &= P \int_{\T^2}   \overline{m}_{(x_0,y_0)} \, d\varsigma_0 (x_0,y_0) 
		\\
		&= \int_{\T^2} P \overline{m}_{(x_0,y_0)} \, d\varsigma_0 (x_0,y_0)  
		\\
		&= 	\int_{\T^2} 
		\overline{m}_{(x_0,y_0)}  \, d\varsigma_0 (x_0,y_0)   -  	\int_{\T^2}  \delta_{(x_0,y_0)}  \, d\varsigma_0 (x_0,y_0)  + 
		\int_{\T^2} \mathbb{E} \left[ \delta_{(x_{V(\omega)},y_{V(\omega)})} \right] \, d\varsigma_0 (x_0,y_0)  
		\\
		&= \mu^{(2)} - \varsigma_0 + P_\xi \varsigma_0
		\\
		&= \mu^{(2)},
	\end{align*}
	the last step by $P_\xi$ invariance of $\varsigma_0$.
	
	Note that the argument includes the observation that $\mu^{(2)}$
	assigns finite measure to compact sets disjoint from $\Delta$.
	
	For $\lambda = 0$, Lemma~\ref{lem:proximal} shows that iterates of points $(x,y) \in \mathbb{T}^2$ enter any small neighborhood of $\Delta$ with positive probability. 
	Combining this with Lemma~\ref{lem:stop} shows
	\begin{align*}
		\mu^{(2)}(\T^2) &= \int_{\T^2} \mathbb{E}  \left[\sum_{j=0}^{V(\omega)-1} \mathbbm 1_{\T^2}(T^j_\omega(x),T^j_\omega(y))\right] \, d\varsigma_0 (x_0,y_0)\\
		&= \infty.
	\end{align*}
	For $\lambda >0$,  by  Lemma~\ref{lem:returntime>}, we see that 
	
	\[
	X =  \int_{\T^2} \mathbb{E}  \left[\sum_{j=0}^{V(x,y,\omega)-1} \mathbbm 1_{\T^2}(T^j_\omega(x),T^j_\omega(y))\right] \, d\varsigma_0 (x,y)
	= \int_{\T^2} \mathbb{E} \left[ V(x,y,\omega) \right] \,d\varsigma_0 (x,y)<\infty.
	\]
	Now $\mu^{(2)}$ given by
	\begin{align*}  
		\mu^{(2)}(A) &= \frac{1}{X}\int_{\T^2} \mathbb{E}  \left[\sum_{j=0}^{V(x,y,\omega)-1} \mathbbm 1_{A}(T^j_\omega(x),T^j_\omega(y))\right] \, d\varsigma_0(x,y)
	\end{align*}
	is a stationary measure. 	Since $\mu^{(2)} (\T^2) = 1$, we find from this expression that $\mu^{(2)}$ is a probability measure. 
\end{proof}

%
%

\begin{remark}\label{rem:restricted}
	The stationary measure $\mu^{(2)}$ is obtained by pushing forward the measure $\varsigma_0$ on the support of a test function $\xi$. The test function is constant one on a suitable set $\mathcal{K}$. The construction shows that $\mu^{(2)}$ restricted to $\mathcal{K}$ equals $\varsigma_0$ restricted to $\mathcal{K}$ (see \eqref{eq:formformbar} and \eqref{eq:defofm2}). We can therefore also obtain $\mu^{(2)}$ from $\varsigma_0$ restricted to $\mathcal{K}$ by pushing forward. 
	That is, with
	\begin{align*}
		\widehat{m}_{(x_0,y_0)} &= \sum_{n=0}^\infty \,\,\,\iint\displaylimits_{\overset{\omega_0,\ldots,\omega_{n-1}}{\in \T}}\, \prod_{j=1}^n  \left( 1 - \mathbbm{1}_\mathcal{K} (x_j,y_j) \right) \delta_{(x_n,y_n)} \, d\lambda(\omega_0) \cdots d\lambda(\omega_{n-1})
	\end{align*}
	for $(x_0,y_0) \in \mathcal{K}$,
	we find
	\begin{align*}
		\mu^{(2)} &= \int_{\mathcal{K}} \widehat{m}_{(x_0,y_0)} \, d\varsigma_0 (x_0,y_0).
	\end{align*}
	\hfill $\blacksquare$
\end{remark}

\subsection{The support of the stationary measure}

Hypothesis~\ref{h:openaftertwo} and Lemma~\ref{lem:2pointexact} show that for all $\varepsilon>0$, there exists a $k\in \N$, such that
images under $\Theta^{(2)}$ of a set $\Sigma_\vartheta \times \{(x,y)\}$ for 
$(x,y) \in \mathbb{T}^2 \setminus \Delta_\varepsilon$, cover $\Sigma_\vartheta \times \mathbb{T}^2$. This implies that stationary measures for the two-point motion have full support, if they are obtained from an inducing scheme as in Section \ref{s:inducing}.

\subsection{The growth rate of the stationary measure at the diagonal for $\lambda=0$}

The following lemma discusses the expected number of such orbit points that lie inside strips $\Delta_\delta \setminus \Delta_\varepsilon$, in the limit of $\varepsilon$ going to zero.  The obtained bounds will be used below to derive the growth-rate near the diagonal of stationary measures for the two-point motion.   		

\begin{lemma}\label{lem:gepsilonnetje}
    Let $R,K$ be as in Lemma~\ref{lem:martingale}. Suppose $\lambda = 0$. 
	Assume	$0<\varepsilon<\delta<R$.  For $(x,y) \in \mathbb{T}^2$ with $0 < d(x,y) < \delta$, 
	define
	\[
	g_{\varepsilon,\delta}(x,y) = \mathbb{E} \left[ \sum_{i=0}^{\tau_{\delta,+}(x,y)} \mathbbm{1}_{[\varepsilon,\infty) }(d( T_\omega^i (x),T_\omega^i (y))) \right].
	\]
	Then
	\begin{multline*}
		\frac{1}{V}\left(\ln\left(\frac{\delta a_1}{d(x,y)}\right) -6K \right) \leq \liminf_{\varepsilon\to 0} \frac{g_{\varepsilon,\delta}(x,y)}{-\ln(\varepsilon)}\\
		\leq \limsup_{\varepsilon\to 0} \frac{g_{\varepsilon,\delta}(x,y)}{-\ln(\varepsilon)} \leq \frac{1}{V}\left(\ln\left(\frac{\delta a_2}{d(x,y)}\right) + 6K \right). 
	\end{multline*}
\end{lemma}

\begin{proof}
	We follow \cite[Proposition~5.6]{MR1144097}, with modifications needed for the discrete time setting. 
	Recall that $a_1,a_2$ are given in Hypothesis~\ref{h:endo}. Define, for $\varepsilon<r<\delta$,		
	\begin{align*}
		g_{\varepsilon,\delta}^-(r) &=  \inf \left\lbrace g_{\varepsilon,\delta}(x,y) \; ; \; ra_1\leq d(x,y) \leq r \right\rbrace,
		\\
		g_{\varepsilon,\delta}^+(r) &=  \sup \left\lbrace g_{\varepsilon,\delta}(x,y) \; ; \; r \leq d(x,y) \leq r a_2 \right\rbrace.
	\end{align*}
	Observe that $r \mapsto g^\pm_{\varepsilon,\delta}(r)$ is a monotone non-increasing function on $[\varepsilon,\delta]$, and constant on $(0,\varepsilon]$. 
	
	We first focus on  $g_{\varepsilon,\delta}^-$.
	Conditioning on the smallest stopping time $\tau_{\delta,+}$ or $\tau_{\varepsilon,-}$, which have been defined in \eqref{eq:tau+} and \eqref{eq:tau-}, we obtain
	\begin{align*}
		g_{\varepsilon,\delta}(x,y) &= \mathbb{P}\left(\tau_{\varepsilon,-}(x,y)> \tau_{\delta,+}(x,y)\right)\mathbb{E} [\tau_{\delta,+}(x,y)]
		\\
		&\qquad+ \mathbb{P}\left(\tau_{\varepsilon,-}(x,y)< \tau_{\delta,-}(x,y)\right) \mathbb{E} [\tau_{\varepsilon,-}(x,y)]
		\\
		&\qquad+ \mathbb{P}\left(\tau_{\varepsilon,-}(x,y)< \tau_{\delta,-}(x,y)\right)
		\mathbb{E} \left[g_{\varepsilon,\delta}(x_{\tau_{\varepsilon,-}},y_{\tau_{\varepsilon,-}})\mid \tau_{\varepsilon,-}(x,y)< \tau_{\delta,+}(x,y) \right] 
		\\
		&= \mathbb{E} \left[\min\{\tau_{\delta,+}(x,y),\tau_{\varepsilon,-}(x,y)\}\right]
		\\
		&\qquad+ \mathbb{P}\left(\tau_{\varepsilon,-}(x,y)< \tau_{\delta,+}(x,y)\right)\mathbb{E} \left[g_{\varepsilon,\delta}(x_{\tau_{\varepsilon,-}},y_{\tau_{\varepsilon,-}})\mid \tau_{\varepsilon,-}(x,y)< \tau_{\delta,+}(x,y) \right]
		\\
		&\geq \mathbb{E} \left[\min\{\tau_{\delta,+}(x,y),\tau_{\varepsilon,-}(x,y)\}\right] + g^-_{\varepsilon,\delta}(\varepsilon)\mathbb{P}\left(\tau_{\varepsilon,-}(x,y)< \tau_{\delta,+}(x,y)\right).
	\end{align*}
	Due to monotonicity of $g_{\varepsilon,\delta}^-(r)$ in $r$, we have for $\varepsilon <  r < \delta$,
	\begin{align*}
		g_{\varepsilon,\delta}^-(\varepsilon)&\geq g_{\varepsilon,\delta}^-(r)\\
		&\geq \inf \left\lbrace g_{\varepsilon,\delta}(x,y) \; ; \; a_1 r \leq d(x,y) \leq r \right\rbrace  \\
		&\geq\inf_{ra_1\leq d(x,y) \leq r}\left\lbrace \mathbb{E} \left[\min\{\tau_{\delta,+}(x,y),\tau_{\varepsilon,-}(x,y)\}\right] + g_{\varepsilon,\delta}^-(\varepsilon) \mathbb{P}\left(\tau_{\varepsilon,-}(x,y)< \tau_{\delta,+}(x,y)\right) \right\rbrace. 
	\end{align*}
	We get the lower bound for $g_{\varepsilon,\delta}(x,y)$ by setting $r = d(x,y)$, using  \eqref{eq:propbounds} and \eqref{eq:Ebounds} from Lemma~\ref{lem:martingale}, and $g_{\varepsilon,\delta}^-(\varepsilon)\geq  g_{\varepsilon,\delta}^-(r)$, 
	\begin{multline*}
		g_{\varepsilon,\delta}(x,y)\geq g_{\varepsilon,\delta}^-(r) \\
		\geq \inf_{ra_1\leq d(x,y)\leq r}\Bigg\{\frac{1}{V}\left( \ln\left( \frac{\delta}{d(x,y)} \right) \ln\left( \frac{d(x,y)}{\varepsilon} \right) 
		- 6K |\ln \varepsilon| + 2K^2 \right)
		\\
		+ g_{\varepsilon,\delta}^-(\varepsilon) \ddfrac{\ln\left( \frac{\delta}{d(x,y)}\right)-2K}{\ln\left(\frac{\delta}{\varepsilon}\right)}\Bigg\}
		\\ 
		\geq \frac{1}{V} \inf_{r/a\leq d(x,y)\leq r}\Bigg\{ \ln\left( \frac{\delta}{d(x,y)} \right) \ln\left( \frac{d(x,y)}{\varepsilon} \right)  \Bigg\}\\
		- 6K |\ln (\varepsilon)| + 2K^2 + g_{\varepsilon,\delta}^-(r) \ddfrac{\ln\left( \frac{\delta}{ra_1}\right)-2K}{\ln\left(\frac{\delta}{\varepsilon}\right)}.
	\end{multline*}
	Divide by $|\ln(\varepsilon)|$ and take $\liminf_{\varepsilon\to 0}$. This yields
	\[
	\liminf_{\varepsilon\to 0}\frac{g_{\varepsilon,\delta}^-(\delta)}{|\ln(\varepsilon)|}\geq\inf_{a_1 r<d(x,y)<r} \frac{\ln\left( \frac{\delta}{d(x,y)} \right)  - 6K }{V} \\
	\geq  \frac{\ln\left( \frac{\delta}{r/a} \right)  - 6K }{V}.
	\]
	The bound for $g^+_{\varepsilon,\delta} (r)$ is obtained by following a similar scheme. 
\end{proof}

%
%
%

The following proposition allows us to estimate on the growth-rate of the stationary measure $\mu^{(2)}$ near the diagonal, in the case $\lambda = 0$.

\begin{proposition}\label{prop:bounds}
	Suppose $\lambda = 0$. Given the stationary measure $\mu^{(2)}$, there exist $\alpha,\beta \in (0,\infty)$, such that 
	\begin{equation}\label{eq:measuremainprop}
		\alpha \leq \liminf_{\varepsilon \to 0} \frac{\mu^{(2)}(\mathbb{T}^2\setminus\Delta_\varepsilon)}{-\ln(\varepsilon)} \leq \limsup_{\varepsilon \to 0} \frac{\mu^{(2)}(\mathbb{T}^2\setminus\Delta_\varepsilon)}{-\ln(\varepsilon)} \leq \beta.
	\end{equation}
\end{proposition}

\begin{proof}  	
	Temporary fix $\varepsilon>0$ small. 
	For $0 < \varepsilon < \delta$ for suitable small $\delta$, set $\mathcal K = \mathbb T^2 \setminus \Delta_\delta$.
	For $(x,y) \in \mathcal{K}$, write
	\[
	N (x,y,\omega) =  \min \{ n > 0 \; ; \; (T^n_\omega (x) , T^n_\omega (y)) \in \mathcal{K} \}
	\]
	for the first return time to $\mathcal{K}$. 
	By Remark~\ref{rem:restricted} we have that for all stationary Radon measures $\mu^{(2)}$ on $\T^2 \setminus \Delta$, the restricted measure $\mu_{\mathcal K} = \mu^{(2)}\vert_{\mathcal K}$ is a stationary measure for the induced process 
	\[
	(x_{n+1},y_{n+1}) = \left(T_\omega^{(2)}\right)^{N(x_n,y_n,\omega)} (x_n,y_n) 
	\]
	on  $\mathcal K$. For convenience we rescale $\mu_{\mathcal{K}}$ so that it becomes a probability measure,
	\[
	\mu_{\mathcal{K}} (\mathcal{K}) = 1.
	\]
	
	Denote the set $\mathcal{G} \subset \mathcal{K}\times {\Sigma_\vartheta}$ as the union of $(x,y,\omega) \in \mathcal{K}\times {\Sigma_\vartheta}$, such that there exists  $\tau(x,y,\omega)\in \N$ with the following properties,
	\begin{enumerate}
		\item for  $0< i <\tau$,  $(T_\omega^i(x),T_\omega^i(y)) \notin \mathcal{K}$,
		\item $(T_\omega^{\tau(x,y,\omega)}(x),T_\omega^{\tau(x,y,\omega)}(y)) \in \Delta_{e^{-7K} a_1 \delta}$.
	\end{enumerate} 
	Here $K$ is as in Lemma~\ref{lem:martingale}.
	By Lemma~\ref{lem:proximal}, we have $(\mu_{\mathcal K} \times \mathbb P) (\mathcal G) > 0$.
	
	Now we have the ingredients to prove the lower bound for the $\liminf$.  For all measurable sets $A \subset \mathbb T^2$, from Remark~\ref{rem:restricted} we obtain
	\begin{equation} \label{eq:induced}
		\mu^{(2)}(A) =  \int_{\mathcal K  \times {\Sigma_\vartheta}}  \sum_{j=0}^{N(x,y,\omega)-1} \mathbbm 1_{A}(T^j_\omega(x),T^j_\omega(y)) \, d(\mu_{\mathcal K}\times \mathbb P) (x,y,\omega).
	\end{equation}
	By \eqref{eq:induced} and Lemma~\ref{lem:gepsilonnetje},
	\begin{align*}
		\mu^{(2)}(\mathbb T^2 \setminus \Delta_{\varepsilon}) &=  \int_{\mathcal K  \times {\Sigma_\vartheta}}  \sum_{j=0}^{N(x,y,\omega)-1} \mathbbm 1_{\mathbb T^2 \setminus \Delta_{\varepsilon}}(T^j_\omega(x),T^j_\omega(y)) \, d(\mu_{\mathcal K}\times \mathbb P) (x,y,\omega)\\
		&\geq   \int_{\mathcal G}  \sum_{j=0}^{N(x,y,\omega)-1} \mathbbm 1_{\mathbb T^2 \setminus \Delta_{\varepsilon}}(T^j_\omega(x),T^j_\omega(y)) \, d(\mu_{\mathcal K}\times \mathbb P) (x,y,\omega)\\
		&\geq   \int_{\mathcal G}  \sum_{j=\tau(x,y,\omega)}^{N(x,y,\omega)-1} \mathbbm 1_{\mathbb T^2 \setminus \Delta_{\varepsilon}}(T^j_\omega(x),T^j_\omega(y)) \, d(\mu_{\mathcal K}\times \mathbb P) (x,y,\omega)\\
		&\geq   \int_{\mathcal G}  g_{\varepsilon,\delta}\left(  \left(T_\omega^{(2)}\right)^{\tau(x,y,\omega)}(x,y)\right) \, d(\mu_{\mathcal K}\times \mathbb P) (x,y,\omega).
	\end{align*}
	To get the first inequality of \eqref{eq:measuremainprop} we divide both sides by $-\ln(\varepsilon)$, take the $\liminf$, and apply Lemma~\ref{lem:gepsilonnetje}. This yields
	\begin{align*}
		\liminf_{\varepsilon\to 0}\frac{\mu^{(2)}(\mathbb T^2 \setminus \Delta_{\varepsilon})}{-\ln(\varepsilon)} 	&\geq \frac{(K + \ln (a_1)) (\mu_K\times\mathbb P) (\mathcal G)}{V}.
	\end{align*}
	This proves the lower bound for the $\liminf$.
	
	To get the upper bound for the  $\limsup$ we use a similar technique. Here we have to incorporate 
	the possibility that points in $\mathcal{K}$ are mapped onto, or close to, $\Delta$ by a single iterate of $T^{(2)}_\omega$. When pushing forward $\mu_{\mathcal{K}}$, this moves mass directly to small neighborhoods of $\Delta$. 
	Write 
	$
	\mathcal{K} \times \Sigma_\vartheta = \mathcal{R}_0 \cup \mathcal{R}_+
	$
	as a union of sets on which $N$ is either equal to $1$ or is bigger than $1$,
	\begin{align*}
		\mathcal{R}_0 = \left\{ (x,y,\omega) \in \mathcal{K} \times \Sigma_\vartheta \; ; \; N(x,y,\omega) = 1\right\},
		\\
		\mathcal{R}_+ = \left\{ (x,y,\omega) \in \mathcal{K} \times \Sigma_\vartheta \; ; \; N(x,y,\omega) > 1\right\}.
	\end{align*}
	The set $\mathcal{R}_+$ is a disjoint union of a set 
	\[
	\mathcal{G}_d  = \mathcal{R}_+ \cap \left( \Delta_{R_\text{min}} \times \Sigma_\vartheta\right)
	\]
	(recall \eqref{eq:Rmin} for the definition of $R_{\min}$) and its complement $\mathcal{G}_c$, which is contained in
	$\left(\Delta_{a_2 \delta}\setminus \Delta_\delta\right) \times \Sigma_\vartheta$.
	For $(x,y,\omega) \in \mathcal{G}_c$ we find $T^{(2)}_\omega (x,y) \subset \Delta_{\delta}\setminus \Delta_{a_1 \delta}$.
	We have 
	\begin{multline}
		\mu^{(2)}(\mathbb T^2 \setminus \Delta_{\varepsilon}) =  
		\int_{\mathcal K}  \int_{{\Sigma_\vartheta}}
		\sum_{j=0}^{N(x,y,\omega)-1} \mathbbm 1_{\mathbb T^2 \setminus \Delta_{\varepsilon}}(T^j_\omega(x),T^j_\omega(y)) \, 
		d\mathbb P(\omega)d\mu_{\mathcal K} (x,y)
		\\
		=  \int_{\mathcal K}  \int_{{\Sigma_\vartheta}}  \sum_{j=0}^{N(x,y,\omega)-1} \mathbbm 1_{\mathbb T^2 \setminus \Delta_{\varepsilon}}(T^j_\omega(x),T^j_\omega(y)) \, d\mathbb P(\omega)d\mu_{\mathcal K} (x,y) 
		\\
		\leq
		\iint_{\mathcal{R}_0}\,  d\mathbb P(\omega)d\mu_{\mathcal K} (x,y)
		+ \iint_{\mathcal{R}_+}   \sum_{j=1}^{N(x,y,\omega)-1} \mathbbm 1_{\mathbb T^2 \setminus \Delta_{\varepsilon}}(T^j_\omega(x),T^j_\omega(y)) \, d\mathbb P(\omega)d\mu_{\mathcal K} (x,y) 
		\\
		\leq \mu_\mathcal{K} (\mathcal{K}) + \iint_{\mathcal{R}_+}   g_{\varepsilon,\delta}( T_\omega^{(2)}(x,y) ) \, d\mathbb P(\omega)d\mu_{\mathcal K} (x,y) \label{eq:m2up}.
	\end{multline}
	
	Recall that by Lemma~\ref{lem:logbnd}, we have the existence of           $C_1>0$ with 
	\begin{align}
		\mathbb{E} \left[  -\ln(d(T_\omega^{(2)}(x,y))) \right] < C_1,
	\end{align}
	for all $(x,y) \in \mathcal{K}$. 
	
	To  conclude the proposition we
	divide \eqref{eq:m2up} by $-\ln(\varepsilon)$ and take the $\limsup$. Doing this, applying Fatou's lemma and using Lemma~\ref{lem:gepsilonnetje} and \eqref{eq:resulthypo}, yields
	\begin{multline*}
		\limsup_{\varepsilon \to 0} \frac{\mu^{(2)}(\mathbb T^2 \setminus \Delta_{\varepsilon})}{-\ln(\varepsilon)}
		\\
		\leq 
		\limsup_{\varepsilon \to 0} \frac{\mu_\mathcal{K} (\mathcal{K})}{-\ln(\varepsilon)}+ \limsup_{\varepsilon \to 0}   \iint_{\mathcal{G}_c \cup \mathcal{G}_d}  
		\frac{g_{\varepsilon,\delta}(T_\omega^{(2)}(x,y) )}{-\ln(\varepsilon)} \, d\mathbb P(\omega)d\mu_{\mathcal K} (x,y)
		\\
		\leq 
		\iint_{\mathcal{G}_c \cup \mathcal{G}_d}  
		\limsup_{\varepsilon \to 0}\frac{g_{\varepsilon,\delta}(T_\omega^{(2)}(x,y) )}{-\ln(\varepsilon)} \, d\mathbb P(\omega)d\mu_{\mathcal K} (x,y)
		\\
		\leq  
		\iint_{\mathcal{G}_c \cup \mathcal{G}_d}  \frac{1}{V}\left(\ln\left(\frac{\delta a_2}{d(T_\omega^{(2)}(x,y) )}\right) + 6K \right) \, d\mathbb P(\omega)d\mu_{\mathcal K} (x,y)
		\\
		\leq 
		\frac{\ln(\delta a_2) + 6K + C_1}{V}, 
	\end{multline*}
	which finishes the proof.
\end{proof}

\subsection{The growth rate of the stationary measure at the diagonal for $\lambda >0$}

The following proposition allows us to estimate on the growth-rate of the stationary measure $\mu^{(2)}$ near the diagonal, in the case of $\lambda >0$.

\begin{proposition}\label{prop:probmeas}
	Suppose $\lambda >0$ and let $\gamma$ be the negative value so that $\Lambda(\gamma)=0$. Suppose $\gamma \in (-1/2,0)$. Given the stationary measure $\mu^{(2)}$, there exist $\alpha,\beta \in (0,\infty)$, such that 
	\begin{equation}\label{eq:measuremainprop>}
		\alpha \leq \liminf_{\varepsilon \to 0} \frac{\mu^{(2)}(\Delta_\varepsilon)}{\varepsilon^{-\gamma}} \leq \limsup_{\varepsilon \to 0} \frac{\mu^{(2)}(\Delta_\varepsilon)}{\varepsilon^{-\gamma}} \leq \beta.
	\end{equation}
\end{proposition} 

To prove the above proposition, we first consider orbits near the diagonal. 

 \begin{lemma}\label{lem:spendtimepos}
Let $R$ be as in Lemma~\ref{lem:martingale}. Suppose $\lambda > 0$. 
Assume $0 < \varepsilon<\delta<R$. For $(x,y) \in \mathbb{T}^2$ with $0 < d(x,y) < \delta$, define
\[
f_{\varepsilon,\delta} (x,y) = \mathbb{E}\left[\sum_{i=0}^{\tau_{\delta,+}(x,y,\omega)-1}\mathbbm{1}_{(0,\varepsilon]}(d(T_\omega^i(x),T_\omega^i(y))) \right].
\]
There exist $K < \infty$, $\kappa \in(0,1)$, such that for $(x,y) \in \mathbb T^2$, if $0<\varepsilon<d(x,y)<\kappa\delta<\kappa R$, then
 	\[
 	\frac{1}{K}\left(\frac{d(x,y)}{\varepsilon}\right)^\gamma 
\leq 
f_{\varepsilon,\delta} (x,y) 
	\leq {K}\left(\frac{d(x,y)}{\varepsilon}\right)^\gamma.  
 	\]
 \end{lemma}
 
 \begin{proof}
	Let $\kappa$ be as in Lemma~\ref{lem:martingalepos}.  Let $K_0$ be $K$ from Lemma~\ref{lem:martingalepos} and set $K_1 = e^{2 \lambda K_0}$. By a straightforward computation for the lower bound, comparable to the proof of Lemma~\ref{lem:gepsilonnetje}, we obtain
 	\begin{align*}
 		f_{\varepsilon,\delta}(x,y) &\geq \mathbb{P}\left(\tau_{\varepsilon/K_1,-}(x,y)< \tau_{\delta,+}(x,y)\right) \mathbb{E}\left[\tau_{\varepsilon,+}(T^{(2)\tau_{\varepsilon/K_1}}_\omega (x,y)) \mid \tau_{\varepsilon/K_1,-}(x,y)< \tau_{\delta,+}(x,y) \right]\\
 		&\geq \frac{1}{K_0}\left(\frac{K_1 d(x,y)a_1}{\varepsilon}\right)^\gamma\left(\frac{1}{\lambda}\left(\ln\left(\frac{a_1}{K_1 }\right) -2K_0\right) \right).
 	\end{align*}
 	Set
 	\begin{equation}\label{eq:feps}
 		f_{\varepsilon,\delta}^+(r) = \sup_{r a_1 < d(x,y)\leq r} \mathbb{E}\left[\sum _{i=0}^{\tau_{\delta,+}(x,y,\omega)-1} \mathbbm{1}_{(0,\varepsilon]}(d(T^{i}_\omega(x),T^i_\omega(y)))\right]
 	\end{equation}
 	and observe that $f_{\varepsilon,\delta}^+$ is monotone decreasing in $r$. Let $K_2 = \left(\frac{1}{2K_0}
    \right)^{1/\gamma}/a_2$. When $K_2\varepsilon < \kappa R$, we have
 	\begin{align*}
 		f_{\varepsilon,\delta}^+(\varepsilon) &\leq \sup_{\varepsilon a_1 < d(x,y)\leq \varepsilon} \mathbb{E}\left[\tau_{(K_2\varepsilon), +}(x,y)\right] + \sup_{\varepsilon K_2  < d(x,y)\leq \varepsilon K_2 a_2} \mathbb{P}\left(\tau_{\varepsilon,-}(x,y)< \tau_{\delta,+}(x,y)\right)f_{\varepsilon,\delta}^+(\varepsilon)\\
 		&\leq \frac{1}{\lambda}  \left(  \ln \left( \frac{K_2}{a_1}\right) + 2K_0 \right)+ K_0\left({K_2a_2} \right)^\gamma f_{\varepsilon,\delta}^+(\varepsilon).
 	\end{align*}
 	Therefore $f_{\varepsilon,\delta}^+(\varepsilon) < \infty$.
 	So, using $f_{\varepsilon,\delta}(x,y) \leq \mathbb{P}\left(\tau_{\varepsilon,-}(x,y)< \tau_{\delta,+}(x,y)\right) f_{\varepsilon,\delta}^+(\varepsilon) $, we get the desired result. 
 \end{proof}

\begin{proof}[Proof of Proposition \ref{prop:probmeas}]
	We follow reasoning of Proposition~\ref{prop:bounds}, using Lemma~\ref{lem:returntime>} and Lemma~\ref{lem:spendtimepos}.  Let $R,K>0$ be as in Lemma~\ref{lem:spendtimepos}. 
	For $0<\delta<R$ small, take $\mathcal{K} = \mathbb{T}^2 \setminus \Delta_\delta$.
	As in Proposition~\ref{prop:infmeas},  a finite measure $\sigma_0$ is constructed and from the restriction $\mu_\mathcal{K}$ of $\sigma_0$ to $\mathcal{K}$, the measure $\mu^{(2)}$ is obtained by pushing forward $\mu_\mathcal{K}$ by the two-point maps. See Remark~\ref{rem:restricted}.
	
	 By rescaling $\mu_\mathcal{K}$ we may assume 
	\[
	\int_{\T^2} \mathbb{E} \left[ V(x,y,\omega) \right] \, d\mu_\mathcal{K} (x,y)=1.
	\]
	Fix $\varepsilon > 0$ small enough and let $\delta <R$. For the stationary measure of the two-point motion we have by Proposition~\ref{prop:infmeas} that
	\begin{align*}
		\mu^{(2)}(\Delta_\varepsilon) &= \int_{\T^2} \mathbb{E}  \left[\sum_{j=0}^{V(x,y,\omega)-1} \mathbbm 1_{\Delta_\varepsilon}(T^j_\omega(x),T^j_\omega(y))\right] \, d\mu_\mathcal{K} (x_0,y_0).
	\end{align*}

    We remark that the condition $|\gamma|<1/2$ is to bound the mass that is transported to neighborhoods of the diagonal from outside $\Delta_{R_\text{min}}$ by this construction (recall \eqref{eq:Rmin} for the definition of $R_{\min}$). 
	
	As before, see the proof of Proposition~\ref{prop:bounds}, denote the set of $\mathcal{G} \subset(\T^2\setminus \Delta_\delta) \times {\Sigma_\vartheta}$ as the union of $(x,y,\omega) \in (\T^2\setminus \Delta_\delta)\times {\Sigma_\vartheta}$, such that there exists an  $\tau(x,y,\omega)\in \N$, with the following properties:
	\begin{enumerate}
		
		\item for all $i \in \N, 0< i <\tau$,  $(T_\omega^i(x),T_\omega^i(y)) \notin (\T^2\setminus \Delta_\delta)$,
		
		\item $(T_\omega^{\tau(x,y,\omega)}(x),T_\omega^{\tau(x,y,\omega)}(y)) \in \Delta_{\kappa\delta}$.
		
	\end{enumerate}  	
	Here $\kappa$ as in the proof of Lemma~\ref{lem:spendtimepos}.
	
	For the lower bound in \eqref{eq:measuremainprop>} we consider only orbit pieces near $\Delta$.
	From Lemma~\ref{lem:proximal} we know that $(\mu_{\delta} \times \mathbb P) (\mathcal G) > 0$.	
	\begin{align*}
		\mu^{(2)}(\Delta_\varepsilon) 
		&\geq  
		\int_{\mathbb T^2 \setminus \Delta_\delta \times \Sigma_\vartheta }
		\sum_{i=0}^{\tau_{\delta,+}(T_\omega(x),T_\omega(y),\omega)-1}\mathbbm{1}_{(0,\varepsilon]}(d(T^i_\omega(x),T^i_\omega(y))) \, d(\mu_\mathcal{K} \times \mathbb P) (x,y,\omega)
		\\
		&\geq  
		\int_{\mathcal  G }
		\sum_{i=0}^{\tau_{\delta,+}(T_\omega(x),T_\omega(y),\omega)-1}\mathbbm{1}_{(0,\varepsilon]}(d(T^i_\omega(x),T^i_\omega(y))) \, d(\mu_\mathcal{K} \times \mathbb P) (x,y,\omega)
		\\
		&\geq  
		\int_{\mathcal  G }
		\sum_{i=\tau(x,y,\omega)}^{\tau_{\delta,+}(T_\omega(x),T_\omega(y),\omega)-1}\mathbbm{1}_{(0,\varepsilon]}(d(T^i_\omega(x),T^i_\omega(y))) \, d(\mu_\mathcal{K} \times \mathbb P) (x,y,\omega)
		\\
		&\geq  
		\frac{\varepsilon^{-\gamma}}{K(\kappa \delta )^{-\gamma}}  \left(\mu_\mathcal{K} \times \mathbb P\right)  (\mathcal G).
	\end{align*}
	
	We proceed with an upper bound. Here we must include orbits that start outside $\Delta_{R_\text{min}}$ and jump 
	directly into $\Delta_\delta$, and also directly into $\Delta_\varepsilon$. This can be seen as pushing forward mass from $\mathcal{K}$ into $\Delta_\varepsilon$ in a single iterate.
	We divide the  set $\mathcal{G}$ into two subsets, $\mathcal{G} = \mathcal{G}_d \cup \mathcal{G}_c$, with a set 
	\[
	\mathcal{G}_d = \mathcal{G} \cap \left(\mathbb{T}^2 \setminus \Delta_{R_\text{min}} \times \Sigma_\vartheta\right)
	\]
	and its complement  
	$\mathcal{G}_c \subset \left(\Delta_{\delta/a_1}\setminus \Delta_\delta\right) \times \Sigma_\vartheta$.
    
	Using this division, we get
	\begin{multline*}
		\mu^{(2)}(\Delta_\varepsilon) \leq 
		\int_{\mathbb T^2 \setminus \Delta_\delta \times \Sigma_\vartheta }
		\sum_{i = 0}^{\tau_{\delta,+}(x,y,\omega)-1}\mathbbm{1}_{(0,\varepsilon]}(d(T_\omega^i(x),T_\omega^i(y))) \, d(\mu_\mathcal{K} \times \mathbb P) (x,y,\omega)
		\\
		\leq 
		\int_{\mathcal{G}_c}
		\sum_{i = 0}^{\tau_{\delta,+}(x,y,\omega)-1}\mathbbm{1}_{(0,\varepsilon]}(d(T_\omega^i(x),T_\omega^i(y))) \, d(\mu_\mathcal{K} \times \mathbb P) (x,y,\omega)
		\\
		+ 
		\int_{\mathcal{G}_d}
		\sum_{i = 0}^{\tau_{\delta,+}(x,y,\omega)-1}\mathbbm{1}_{(0,\varepsilon]}(d(T_\omega^i(x),T_\omega^i(y))) \, d(\mu_\mathcal{K} \times \mathbb P) (x,y,\omega).
	\end{multline*}
	
	The first integral can be  bounded from above by using Lemma~\ref{lem:spendtimepos}. Recall that $f^+_{\varepsilon,\delta}$ defined in 
	\eqref{eq:feps} is a monotone function. Now
	\begin{multline*}
		\int_{\mathcal{G}_c}
		\sum_{i = 0}^{\tau_{\delta,+}(x,y,\omega)-1} 
		\mathbbm{1}_{(0,\varepsilon]}(d(T^i_\omega(x),T^i_\omega(y))) \, d(\mu_\mathcal{K} \times \mathbb P) (x,y,\omega)
		\\
		\leq \int_{\mathcal{G}_c} f^+_{\varepsilon,\delta}(d(T^{(2)}_\omega(x,y))) d(\mu_\mathcal{K} \times \mathbb P) (x,y,\omega) 
		\leq  f^+_{\varepsilon,\delta}(\kappa \delta)
		\leq K \left(\frac{\varepsilon}{a_1 \kappa \delta}\right)^{-\gamma}.  
	\end{multline*}
	
	To handle the second integral we consider separately the subset $\mathcal{G}_\varepsilon \subset \mathcal{G}_d$ of points
	that are mapped directly in $\Delta_\varepsilon$, thus
	\[
	\mathcal{G}_\varepsilon = \left\{  (x,y,\omega) \in \mathcal{G}_d \; ; \;  T^{(2)}_\omega (x,y) \in \Delta_\varepsilon  \right\}.  
	\]
	For $(x,y) \in \Delta_\varepsilon$ let, similar to  \eqref{eq:tau+}, 
	\begin{align*}
		\tau_{\varepsilon,+}(x,y) &= \min \{n\in \mathbb N \; ; \;  d(T_\omega^n(x), T_\omega^n(y))> \varepsilon \},
		\\
		\tau_{\delta,+}(x,y) &= \min \{n\in \mathbb N \; ; \;  d(T_\omega^n(x), T_\omega^n(y))> \delta \}.
	\end{align*}
	By Lemma~\ref{lem:returntime>} we have an estimate
	\[
	\frac{1}{\lambda}  \left(  \ln \left( \frac{\varepsilon}{d(x,y)} \right) - 2K \right)  \leq \mathbb{E}\left[ \tau _{\varepsilon,+}(x,y)  \right]  \leq \frac{1}{\lambda}  \left(  \ln \left( \frac{\varepsilon}{d(x,y)}\right) + 2K \right)
	\]
	for some constant $K>0$.
	For $(x,y,\omega) \in \mathcal{G}_\varepsilon$ we find $\tau_{\varepsilon,+}(x,y,\omega) < \tau_{\delta,+}(x,y,\omega)$ almost surely.
	This allows us to write
	\begin{multline*}
		\int_{\mathcal{G}_d}
		\sum_{i = 0}^{\tau_{\delta,+}(x,y,\omega)-1}\mathbbm{1}_{(0,\varepsilon]}(d(T_\omega^i(x),T_\omega^i(y))) \, d(\mu_\mathcal{K} \times \mathbb P) (x,y,\omega) 
		\\
		\leq
		\int_{\mathcal{G}_\varepsilon }\left(\sum_{i = 0}^{\tau_{\varepsilon,+}(x,y,\omega)-1} + 
		\sum_{i = \tau_{\varepsilon,+}(x,y,\omega) }^{\tau_{\delta,+}(x,y,\omega)-1}\right)\mathbbm{1}_{(0,\varepsilon]}(d(T_\omega^i(x),T_\omega^i(y))) \, d(\mu_\mathcal{K} \times \mathbb P) (x,y,\omega)
		\\
		+ \int_{\mathcal{G}_d \setminus \mathcal{G}_\varepsilon} \sum_{i = 0 }^{\tau_{\delta,+}(x,y,\omega)-1}
		\mathbbm{1}_{(0,\varepsilon]}(d(T_\omega^i(x),T_\omega^i(y))) \, d(\mu_\mathcal{K} \times \mathbb P) (x,y,\omega)
		\\
		\leq  \int_{\mathcal{G}_\varepsilon} \tau_{\varepsilon, +}(T^{(2)}_\omega(x,y)) 
		\, d(\mu_\mathcal{K} \times \mathbb P) (x,y,\omega)
		+ \int_{\mathcal{G}_\varepsilon}  f^+_{\varepsilon,\delta} (\varepsilon) \, d(\mu_\mathcal{K} \times \mathbb P) (x,y,\omega)
		\\
		+ \int_{\mathcal{G}_d\setminus \mathcal{G}_\varepsilon} f^+_{\varepsilon,\delta}( d( T^{(2)}_\omega (x,y) ) )  \, d(\mu_\mathcal{K} \times \mathbb P) (x,y,\omega).
	\end{multline*}
	To bound the first integral in the final expression we use Lemma~\ref{lem:returntime>}, Lemma~\ref{lem:logbnd} and Hypothesis \ref{h:genericnoise} to get or any $0 < s < 1$ the existence of a constant $\tilde{C} > 0$, such that
	\[
	\int_{\mathcal{G}_d}\tau_{\delta, +}(T^{(2)}_\omega(x,y))  \, d(\mu_\mathcal{K} \times \mathbb P) (x,y,\omega) 
	\leq 
	\tilde{C} \varepsilon^s. 
	\] 
	Hypothesis~\ref{h:genericnoise} shows the existence of a constant $C>0$ with 
	$(\mu_{\mathcal{K}}\times\mathbb P) (\mathcal{G}_\varepsilon) < C \sqrt{\varepsilon}$. 
	Lemma~\ref{lem:returntime>} shows that
	\[
	f^+_{\varepsilon,\delta} (\varepsilon) \leq C \frac{1}{\lambda}
	\]
	for some $C>0$.
	We get therefore a bound
	\[
	\int_{\mathcal{G}_\varepsilon}  f^+_{\varepsilon,\delta} (\varepsilon) \, d(\mu_\mathcal{K} \times \mathbb P) (x,y,\omega) \leq C \sqrt{\varepsilon}
	\]
	for some $C>0$.
	The third integral is bounded by
	\begin{multline*}
		\int_{\mathcal{G}_d\setminus \mathcal{G}_\varepsilon} f_{\varepsilon,\delta}^+(d(T_\omega^{(2)}(x,y)))\,d(\varsigma_\delta \times \mathbb P) (x,y,\omega) 
		\\
		\leq \int_{\mathcal{G}_d\setminus \mathcal{G}_\varepsilon}  K \left(\frac{\varepsilon}{d(T_\omega^{(2)}(x,y))}\right)^{-\gamma} \, d(\varsigma_\delta \times \mathbb P)
		\\
		\leq K \varepsilon^{-\gamma} \int_{\mathcal{G}_d\setminus \mathcal{G}_\varepsilon} \, d(T_\omega^{(2)}(x,y))^{\gamma}d(\varsigma_\delta \times \mathbb P)(x,y,\omega) 
		\leq C  \varepsilon^\gamma,
	\end{multline*}
	for some $C>0$.
	Here we use Lemma~\ref{lem:spendtimepos} and Hypothesis~\ref{h:genericnoise}, and $\gamma < 1$.
	Combining the above analysis yields
	\begin{align*}
		\limsup_{\varepsilon\to 0 }\frac{\mu^{(2)}(\Delta_\varepsilon)}{\varepsilon^{-\gamma}} &\leq C
	\end{align*}
	for some $C>0$, if $\gamma \in (-1/2,0)$.
\end{proof}

\subsection*{Acknowledgements}
The authors gratefully acknowledge useful discussions with Dennis Chemnitz, Gr\'egory Desplats, Maximilian Engel, Rainer Klages, Julian Newman, Dmitry Turaev and David Villringer. We thank the anonymous referees for their thorough reading and insightful suggestions, which led to a number of improvements in the exposition. VPHG and JSWL have been supported by the EPSRC
Centre for Doctoral Training in Mathematics of Random Systems: Analysis, Modelling and Simulation (EP/S023925/1). VPHG thanks the Mathematical Institute of Leiden University and the Institut de Recherche Mathématiques de Rennes for their hospitality. VPHG also thanks the Scientific High Level Visiting Fellowships (SSHN) of the Higher Education, Research and Innovation Department of the French Embassy in the United Kingdom for their support.
JSWL has been supported by the EPSRC (EP/Y020669/1) and thanks JST (Moonshot R \& D Grant Number JPMJMS2021 - IRCN, University of Tokyo) and GUST (Kuwait) for their research support.
	\bibliographystyle{plain} 
		\bibliography{exp-contr-biblio_VG} 

@article {MR4404792,
    AUTHOR = {Bedrossian, J. and Blumenthal, A. and Punshon-Smith, S.},
     TITLE = {Lagrangian chaos and scalar advection in stochastic fluid
              mechanics},
   JOURNAL = {J. Eur. Math. Soc. (JEMS)},
  FJOURNAL = {Journal of the European Mathematical Society (JEMS)},
    VOLUME = {24},
      YEAR = {2022},
    NUMBER = {6},
     PAGES = {1893--1990},
      ISSN = {1435-9855,1435-9863},
   MRCLASS = {60H15 (35Q30 37D25 37L55 76F20)},
  MRNUMBER = {4404792},
MRREVIEWER = {Xuhui\ Peng},
       DOI = {10.4171/jems/1140},
       URL = {https://doi.org/10.4171/jems/1140},
}

@article {MR4415776,
    AUTHOR = {Bedrossian, J. and Blumenthal, A. and Punshon-Smith, S.},
     TITLE = {The {B}atchelor spectrum of passive scalar turbulence in
              stochastic fluid mechanics at fixed {R}eynolds number},
   JOURNAL = {Comm. Pure Appl. Math.},
  FJOURNAL = {Communications on Pure and Applied Mathematics},
    VOLUME = {75},
      YEAR = {2022},
    NUMBER = {6},
     PAGES = {1237--1291},
      ISSN = {0010-3640,1097-0312},
   MRCLASS = {76F55},
  MRNUMBER = {4415776},
MRREVIEWER = {Silvia\ Lorenzani},
       DOI = {10.1002/cpa.22022},
       URL = {https://doi.org/10.1002/cpa.22022},
}

@article {MR576270,
    AUTHOR = {Pomeau, Y. and Manneville, P.},
     TITLE = {Intermittent transition to turbulence in dissipative dynamical
              systems},
   JOURNAL = {Comm. Math. Phys.},
  FJOURNAL = {Communications in Mathematical Physics},
    VOLUME = {74},
      YEAR = {1980},
    NUMBER = {2},
     PAGES = {189--197},
      ISSN = {0010-3616,1432-0916},
   MRCLASS = {58F13 (65L20)},
  MRNUMBER = {576270},
       URL = {http://projecteuclid.org/euclid.cmp/1103907981},
}

@article{lamb_horseshoes_2025,
	title = {Horseshoes for a {Class} of {Nonuniformly} {Expanding} {Random} {Circle} {Maps}},
	issn = {1424-0661},
	url = {https://doi.org/10.1007/s00023-025-01555-1},
	doi = {10.1007/s00023-025-01555-1},
	journal = {Annales Henri Poincaré},
	author = {Lamb, J. S. W. and Tenaglia, G. and Turaev, D.},
	year = {2025},
        doi = {https://doi.org/10.1007/s00023-025-01555-1}
}

@article {MR1900136,
    AUTHOR = {Blanchard, F. and Glasner, E. and Kolyada, S.
               and Maass, A.},
     TITLE = {On {L}i-{Y}orke pairs},
   JOURNAL = {J. Reine Angew. Math.},
  FJOURNAL = {Journal f\"ur die Reine und Angewandte Mathematik. [Crelle's
              Journal]},
    VOLUME = {547},
      YEAR = {2002},
     PAGES = {51--68},
      ISSN = {0075-4102,1435-5345},
   MRCLASS = {37B05 (37A25 37B40)},
  MRNUMBER = {1900136},
MRREVIEWER = {Jan\ Kwiatkowski},
       DOI = {10.1515/crll.2002.053},
       URL = {https://doi.org/10.1515/crll.2002.053},
}

@article {MR385028,
    AUTHOR = {Li, T. Y. and Yorke, J. A.},
     TITLE = {Period three implies chaos},
   JOURNAL = {Amer. Math. Monthly},
  FJOURNAL = {American Mathematical Monthly},
    VOLUME = {82},
      YEAR = {1975},
    NUMBER = {10},
     PAGES = {985--992},
      ISSN = {0002-9890,1930-0972},
   MRCLASS = {26A18},
  MRNUMBER = {385028},
MRREVIEWER = {B.\ Crstici},
       DOI = {10.2307/2318254},
       URL = {https://doi.org/10.2307/2318254},
}

@book {MR1070361,
    AUTHOR = {Kunita, H.},
     TITLE = {Stochastic flows and stochastic differential equations},
    SERIES = {Cambridge Studies in Advanced Mathematics},
    VOLUME = {24},
 PUBLISHER = {Cambridge University Press, Cambridge},
      YEAR = {1990},
     PAGES = {xiv+346},
      ISBN = {0-521-35050-6},
   MRCLASS = {60H10 (35R60 60H15)},
  MRNUMBER = {1070361},
MRREVIEWER = {Yves\ Le Jan},
}

@article {MR629208,
    AUTHOR = {Eckmann, J.-P.},
     TITLE = {Roads to turbulence in dissipative dynamical systems},
   JOURNAL = {Rev. Modern Phys.},
  FJOURNAL = {Reviews of Modern Physics},
    VOLUME = {53},
      YEAR = {1981},
    NUMBER = {4},
     PAGES = {643--654},
      ISSN = {0034-6861,1539-0756},
   MRCLASS = {58F13 (70F99 82A05)},
  MRNUMBER = {629208},
MRREVIEWER = {P.\ L.\ Sulem},
       DOI = {10.1103/RevModPhys.53.643},
       URL = {https://doi.org/10.1103/RevModPhys.53.643},
}

@misc{lamb2025nonuniformexpansionadditive,
      title={Non uniform expansion and additive noise imply random horseshoe}, 
      author={J. S. W. Lamb and G. Tenaglia and D. V. Turaev},
      year={2025},
      eprint={2501.11656},
      archivePrefix={arXiv},
      primaryClass={math.DS},
      url={https://arxiv.org/abs/2501.11656}, 
      howpublished = {Preprint, {arXiv}:2501.11656},
}

@article {MR4242626,
    AUTHOR = {Bedrossian, J. and Blumenthal, A. and Punshon-Smith, S.},
     TITLE = {Almost-sure enhanced dissipation and uniform-in-diffusivity
              exponential mixing for advection-diffusion by stochastic
              {N}avier-{S}tokes},
   JOURNAL = {Probab. Theory Related Fields},
  FJOURNAL = {Probability Theory and Related Fields},
    VOLUME = {179},
      YEAR = {2021},
    NUMBER = {3-4},
     PAGES = {777--834},
      ISSN = {0178-8051,1432-2064},
   MRCLASS = {60H15 (37A25 37H15 47D07 76D06 76F20)},
  MRNUMBER = {4242626},
MRREVIEWER = {Chengfeng\ Sun},
       DOI = {10.1007/s00440-020-01010-8},
       URL = {https://doi.org/10.1007/s00440-020-01010-8},
}

@article {MR4385127,
    AUTHOR = {Bedrossian, J. and Blumenthal, A. and Punshon-Smith,
              S.},
     TITLE = {Almost-sure exponential mixing of passive scalars by the
              stochastic {N}avier-{S}tokes equations},
   JOURNAL = {Ann. Probab.},
  FJOURNAL = {The Annals of Probability},
    VOLUME = {50},
      YEAR = {2022},
    NUMBER = {1},
     PAGES = {241--303},
      ISSN = {0091-1798,2168-894X},
   MRCLASS = {37H10 (37A25 37N10 47D07 60H30 76D06)},
  MRNUMBER = {4385127},
MRREVIEWER = {Yanguang\ (Charles)\ Li},
       DOI = {10.1214/21-aop1533},
       URL = {https://doi.org/10.1214/21-aop1533},
}

@article {MR1165510,
    AUTHOR = {Pikovsky, A. S.},
     TITLE = {Statistics of trajectory separation in noisy dynamical
              systems},
   JOURNAL = {Phys. Lett. A},
  FJOURNAL = {Physics Letters. A},
    VOLUME = {165},
      YEAR = {1992},
    NUMBER = {1},
     PAGES = {33--36},
      ISSN = {0375-9601,1873-2429},
   MRCLASS = {58F08 (58F30)},
  MRNUMBER = {1165510},
MRREVIEWER = {M.\ L.\ Blank},
       DOI = {10.1016/0375-9601(92)91049-W},
       URL = {https://doi.org/10.1016/0375-9601(92)91049-W},
}

@article {MR800244,
	AUTHOR = {Carverhill, A.},
	TITLE = {A formula for the {L}yapunov numbers of a stochastic flow.
	{A}pplication to a perturbation theorem},
	JOURNAL = {Stochastics},
	FJOURNAL = {Stochastics},
	VOLUME = {14},
	YEAR = {1985},
	NUMBER = {3},
	PAGES = {209--226},
	ISSN = {0090-9491},
	MRCLASS = {93E15 (58G32)},
	MRNUMBER = {800244},
	MRREVIEWER = {Jean-Pierre\ Lepeltier},
	DOI = {10.1080/17442508508833339},
	URL = {https://doi.org/10.1080/17442508508833339},
}

@article{molcanov_structure_1978,
	author = {Molchanov, S. A.},
	doi = {10.1070/IM1978v012n01ABEH001841},
	issn = {0025-5726},
	journal = {Math. USSR Izv.},
	number = {1},
	pages = {69–101},
	title = {The structure of eigenfunction of the one-dimensional unordered structures},
	url = {https://www.mathnet.ru/eng/im1692},
	urldate = {2024-12-14},
	volume = {12},
	year = {1978}
}

@article{PierreHumbert,
	author = {Blumenthal, A. and {Coti Zelati}, M. and Gvalani, R. S.},
	doi = {10.1214/23-AOP1627},
	fjournal = {The Annals of Probability},
	journal = {Ann. Probab.},
	number = {4},
	pages = {1559 – 1601},
	publisher = {Institute of Mathematical Statistics},
	title = {{Exponential mixing for random dynamical systems and an example of Pierrehumbert}},
	url = {https://doi.org/10.1214/23-AOP1627},
	volume = {51},
	year = {2023}
}

@article{PhysRevLett.65.2935,
	author = {Yu, L. and Ott, E. and Chen, Q.},
	doi = {10.1103/PhysRevLett.65.2935},
	issue = {24},
	journal = {Phys. Rev. Lett.},
	numpages = {0},
	pages = {2935–2938},
	publisher = {American Physical Society},
	title = {Transition to chaos for random dynamical systems},
	url = {https://link.aps.org/doi/10.1103/PhysRevLett.65.2935},
	volume = {65},
	year = {1990}
}

@article{StayHomble,
	author = {Homblé, P.},
	doi = {10.1080/07362999408809355},
	fjournal = {Stochastic Analysis and Applications},
	issn = {0736-2994,1532-9356},
	journal = {Stochastic Anal. Appl.},
	mrclass = {60J05 (34D08 60F10 93E15)},
	mrreviewer = {Wolfgang\ Kliemann},
	number = {3},
	pages = {329–354},
	title = {Sample and moment {L}yapunov exponents of discrete linear dynamical systems},
	url = {https://doi.org/10.1080/07362999408809355},
	volume = {12},
	year = {1994}
}

@article{ArnoldStabFormula,
	author = {Arnold, L.},
	doi = {10.1137/0144057},
	fjournal = {SIAM Journal on Applied Mathematics},
	issn = {0036-1399},
	journal = {SIAM J. Appl. Math.},
	mrclass = {34F05 (60H10 93E15)},
	mrreviewer = {S.\ M.\ Khrī sanov},
	number = {4},
	pages = {793–802},
	title = {A formula connecting sample and moment stability of linear stochastic systems},
	url = {https://doi.org/10.1137/0144057},
	volume = {44},
	year = {1984}
}

@article{MR3826118,
	author = {Abbasi, N. and Gharaei, M. and Homburg, A. J.},
	doi = {10.1088/1361-6544/aac637},
	fjournal = {Nonlinearity},
	issn = {0951-7715},
	journal = {Nonlinearity},
	mrclass = {37E05 (28D05 37H10)},
	mrreviewer = {Byungik Kahng},
	number = {8},
	pages = {3880–3913},
	title = {{Iterated function systems of logistic maps: synchronization and intermittency}},
	url = {https://doi-org.vu-nl.idm.oclc.org/10.1088/1361-6544/aac637},
	volume = {31},
	year = {2018}
}

@article {MR756386,
    AUTHOR = {Antonov, V. A.},
     TITLE = {Modeling of processes of cyclic evolution type.
              {S}ynchronization by a random signal},
   JOURNAL = {Vestnik Leningrad. Univ. Mat. Mekh. Astronom.},
  FJOURNAL = {Vestnik Leningradskogo Universiteta. Matematika, Mekhanika,
              Astronomiya},
      YEAR = {1984},
     PAGES = {67--76},
      ISSN = {0024-0850},
   MRCLASS = {85A05 (70F15 86A15)},
  MRNUMBER = {756386},
}

@article{MR2033186,
	author = {Athreya, K. B. and Schuh, H.-J.},
	doi = {10.1023/B:JOTP.0000011994.90898.81},
	fjournal = {Journal of Theoretical Probability},
	issn = {0894-9840},
	journal = {J. Theoret. Probab.},
	mrclass = {60J05 (37H10 60F05 60F15)},
	mrreviewer = {Michele Zito},
	number = {4},
	pages = {813–830},
	title = {{Random logistic maps. {II}. {T}he critical case}},
	url = {https://doi.org/10.1023/B:JOTP.0000011994.90898.81},
	volume = {16},
	year = {2003}
}

@article{BB16,
	author = {Bahsoun, W. and Bose, C.},
	fjournal = {Nonlinearity},
	journal = {Nonlinearity},
	number = {4},
	pages = {1417–1433},
	title = {{Mixing rates and limit theorems for random intermittent maps}},
	volume = {29},
	year = {2016}
}

@incollection{MR872098,
	author = {Baxendale, P. H.},
	booktitle = {{Stochastic processes and their applications ({N}agoya, 1985)}},
	doi = {10.1007/BFb0076869},
	mrclass = {60B99 (58D05 58F11 58G32 60H10 60J60)},
	mrreviewer = {Hiroyuki Matsumoto},
	pages = {1–19},
	publisher = {Springer, Berlin},
	series = {{Lecture Notes in Math.}},
	title = {{Asymptotic behaviour of stochastic flows of diffeomorphisms}},
	url = {https://doi-org.vu-nl.idm.oclc.org/10.1007/BFb0076869},
	volume = {1203},
	year = {1986}
}

@article{MR968817,
	author = {Baxendale, P. H. and Stroock, D. W.},
	doi = {10.1007/BF00356102},
	fjournal = {Probability Theory and Related Fields},
	issn = {0178-8051},
	journal = {Probab. Theory Related Fields},
	mrclass = {58G32 (58D25 60F10 60H99)},
	mrreviewer = {Yu. E. Gliklikh},
	number = {2},
	pages = {169–215},
	title = {{Large deviations and stochastic flows of diffeomorphisms}},
	url = {https://doi.org/10.1007/BF00356102},
	volume = {80},
	year = {1988}
}

@incollection{MR1144097,
	author = {Baxendale, P. H.},
	booktitle = {{Spatial stochastic processes}},
	mrclass = {60H20 (34D08 58G32)},
	mrreviewer = {R. W. R. Darling},
	pages = {189–218},
	publisher = {Birkhäuser Boston, Boston, MA},
	series = {{Progr. Probab.}},
	title = {{Statistical equilibrium and two-point motion for a stochastic flow of diffeomorphisms}},
	volume = {19},
	year = {1991}
}

@book{MR3930614,
	author = {Durrett, R.},
	doi = {10.1017/9781108591034},
	isbn = {978-1-108-47368-2},
	mrclass = {60-01 (37A30)},
	pages = {xii+419},
	publisher = {Cambridge University Press, Cambridge},
	series = {{Cambridge Series in Statistical and Probabilistic Mathematics}},
	title = {{Probability—theory and examples}},
	url = {https://doi-org.vu-nl.idm.oclc.org/10.1017/9781108591034},
	volume = {49},
	year = {2019}
}

@article{MR3600645,
	author = {Gharaei, M. and Homburg, A. J.},
	doi = {10.3934/dcdss.2017012},
	fjournal = {Discrete and Continuous Dynamical Systems. Series S},
	issn = {1937-1632},
	journal = {Discrete Contin. Dyn. Syst. Ser. S},
	mrclass = {37E05 (37H10 37H15)},
	mrreviewer = {Romain Aimino},
	number = {2},
	pages = {241–272},
	title = {{Random interval diffeomorphisms}},
	url = {https://doi-org.vu-nl.idm.oclc.org/10.3934/dcdss.2017012},
	volume = {10},
	year = {2017}
}

@article{MR968818,
	author = {Ledrappier, F. and Young, L.-S.},
	doi = {10.1007/BF00356103},
	fjournal = {Probability Theory and Related Fields},
	issn = {0178-8051},
	journal = {Probab. Theory Related Fields},
	mrclass = {58F11 (60J05)},
	mrnumber = {968818},
	mrreviewer = {Paweł Góra},
	number = {2},
	pages = {217–240},
	title = {{Entropy formula for random transformations}},
	url = {https://doi-org.vu-nl.idm.oclc.org/10.1007/BF00356103},
	volume = {80},
	year = {1988}
}

@book{MR1368405,
	author = {Shiryaev, A. N.},
	doi = {10.1007/978-1-4757-2539-1},
	edition = {Second},
	isbn = {0-387-94549-0},
	mrclass = {60-01},
	pages = {xvi+623},
	publisher = {Springer-Verlag, New York},
	series = {{Graduate Texts in Mathematics}},
	title = {{Probability}},
	url = {https://doi.org/10.1007/978-1-4757-2539-1},
	volume = {95},
	year = {1996}
}

@article{MR722776,
	author = {Pelikan, S.},
	doi = {10.2307/2000087},
	fjournal = {Transactions of the American Mathematical Society},
	issn = {0002-9947},
	journal = {Trans. Amer. Math. Soc.},
	mrclass = {58F11 (28D05 58F13)},
	mrreviewer = {Hitoshi Nakada},
	number = {2},
	pages = {813–825},
	title = {{Invariant densities for random maps of the interval}},
	url = {https://doi.org/10.2307/2000087},
	volume = {281},
	year = {1984}
}

@article{MR3663624,
	author = {Flandoli, F. and Gess, B. and Scheutzow, M.},
	doi = {10.1007/s00440-016-0716-2},
	fjournal = {Probability Theory and Related Fields},
	issn = {0178-8051},
	journal = {Probab. Theory Related Fields},
	mrclass = {37H15 (37B25 37G35 60H10)},
	mrreviewer = {Min Zhao},
	number = {3-4},
	pages = {511–556},
	title = {{Synchronization by noise}},
	url = {https://doi-org.vu-nl.idm.oclc.org/10.1007/s00440-016-0716-2},
	volume = {168},
	year = {2017}
}

@article{MR3862799,
	author = {Homburg, A. J.},
	doi = {10.1007/s00574-018-0073-0},
	fjournal = {Bulletin of the Brazilian Mathematical Society. New Series. Boletim da Sociedade Brasileira de Matemática},
	issn = {1678-7544},
	journal = {Bull. Braz. Math. Soc. (N.S.)},
	mrclass = {37C05 (37D30)},
	mrreviewer = {Abbas Fakhari},
	number = {3},
	pages = {615–635},
	title = {{Synchronization in minimal iterated function systems on compact manifolds}},
	url = {https://doi-org.vu-nl.idm.oclc.org/10.1007/s00574-018-0073-0},
	volume = {49},
	year = {2018}
}

@article{MR3820004,
	author = {Newman, J.},
	doi = {10.1017/etds.2016.109},
	fjournal = {Ergodic Theory and Dynamical Systems},
	issn = {0143-3857},
	journal = {Ergodic Theory Dynam. Systems},
	mrclass = {37H10 (60J35)},
	mrreviewer = {Maximilian Engel},
	number = {5},
	pages = {1857–1875},
	title = {{Necessary and sufficient conditions for stable synchronization in random dynamical systems}},
	url = {https://doi-org.vu-nl.idm.oclc.org/10.1017/etds.2016.109},
	volume = {38},
	year = {2018}
}

@book{MR787404,
	author = {Deimling, K.},
	doi = {10.1007/978-3-662-00547-7},
	isbn = {3-540-13928-1},
	mrclass = {47-01 (47Hxx 55M25 58-01 58Cxx)},
	mrreviewer = {Joachim Naumann},
	pages = {xiv+450},
	publisher = {Springer-Verlag, Berlin},
	title = {{Nonlinear functional analysis}},
	url = {https://doi-org.vu-nl.idm.oclc.org/10.1007/978-3-662-00547-7},
	year = {1985}
}

@article{MR3098067,
	author = {Deroin, B. and Kleptsyn, V. and Navas, A. and Parwani, K.},
	doi = {10.1214/12-AOP784},
	fjournal = {The Annals of Probability},
	issn = {0091-1798},
	journal = {Ann. Probab.},
	mrclass = {60J99 (06F15 37H99 60J05)},
	mrreviewer = {C. R. E. Raja},
	number = {3B},
	pages = {2066–2089},
	title = {Symmetric random walks on {${\rm Homeo}^+({\bf R})$}},
	url = {https://doi.org/10.1214/12-AOP784},
	volume = {41},
	year = {2013}
}

@article{MR4434534,
	author = {Brofferio, S. and Buraczewski, D. and Szarek, T.},
	doi = {10.1017/etds.2021.31},
	fjournal = {Ergodic Theory and Dynamical Systems},
	issn = {0143-3857},
	journal = {Ergodic Theory Dynam. Systems},
	mrclass = {60J05 (37A05 37H99)},
	number = {7},
	pages = {2207–2238},
	title = {On uniqueness of invariant measures for random walks on {$\rm HOMEO^+(\mathbb{R})$}},
	url = {https://doi.org/10.1017/etds.2021.31},
	volume = {42},
	year = {2022}
}

@article{MR1428518,
	author = {Babillot, M. and Bougerol, P. and Elie, L.},
	doi = {10.1214/aop/1024404297},
	fjournal = {The Annals of Probability},
	issn = {0091-1798},
	journal = {Ann. Probab.},
	mrclass = {60J10 (60B15 60J15 60K05)},
	mrreviewer = {Charles M. Goldie},
	number = {1},
	pages = {478–493},
	title = {The random difference equation {$X_n=A_nX_{n-1}+B_n$} in the critical case},
	url = {https://doi.org/10.1214/aop/1024404297},
	volume = {25},
	year = {1997}
}

@article{MR3342668,
	author = {Brofferio, S. and Buraczewski, D.},
	doi = {10.1214/13-AOP903},
	fjournal = {The Annals of Probability},
	issn = {0091-1798},
	journal = {Ann. Probab.},
	mrclass = {37H10 (37B35 60B15 60J05 60K05)},
	mrreviewer = {Christophe Cuny},
	number = {3},
	pages = {1456–1492},
	title = {On unbounded invariant measures of stochastic dynamical systems},
	url = {https://doi.org/10.1214/13-AOP903},
	volume = {43},
	year = {2015}
}

@article{https://doi.org/10.48550/arxiv.2207.09987,
    AUTHOR = {Homburg, A. J. and Kalle, C.},
     TITLE = {Iterated function systems of affine expanding and contracting
              maps on the unit interval},
   JOURNAL = {Adv. Math.},
  FJOURNAL = {Advances in Mathematics},
    VOLUME = {482},
      YEAR = {2025},
     PAGES = {Paper No. 110605},
      ISSN = {0001-8708,1090-2082},
   MRCLASS = {37H20 (28A80 37A25 37E05 37H15)},
  MRNUMBER = {4975394},
       DOI = {10.1016/j.aim.2025.110605},
       URL = {https://doi.org/10.1016/j.aim.2025.110605},
}

@article{MR4142337,
	author = {Matias, E. and Melo, Í.},
	doi = {10.1142/S021919971950086X},
	fjournal = {Communications in Contemporary Mathematics},
	issn = {0219-1997,1793-6683},
	journal = {Commun. Contemp. Math.},
	mrclass = {37C70 (37C40 37H05)},
	mrnumber = {4142337},
	mrreviewer = {M.\ L.\ Blank},
	number = {8},
	pages = {1950086, 18},
	title = {Random products of maps synchronizing on average},
	url = {https://doi.org/10.1142/S021919971950086X},
	volume = {22},
	year = {2020}
}

@article{MR3519583,
	author = {Gorodetski, A. and Kleptsyn, V.},
	doi = {10.1007/s00220-016-2678-8},
	fjournal = {Communications in Mathematical Physics},
	issn = {0010-3616,1432-0916},
	journal = {Comm. Math. Phys.},
	mrclass = {37E30 (37H10)},
	mrnumber = {3519583},
	mrreviewer = {M.\ L.\ Blank},
	number = {3},
	pages = {781–796},
	title = {Synchronization properties of random piecewise isometries},
	url = {https://doi.org/10.1007/s00220-016-2678-8},
	volume = {345},
	year = {2016}
}

@article{MR4228028,
	author = {Matias, E. and Silva, E.},
	doi = {10.1088/1361-6544/abe17b},
	fjournal = {Nonlinearity},
	issn = {0951-7715,1361-6544},
	journal = {Nonlinearity},
	mrclass = {37H12 (37C70 60F17 60J05)},
	mrnumber = {4228028},
	mrreviewer = {Xuewei\ Ju},
	number = {3},
	pages = {1577–1597},
	title = {Random iterations of maps on {$\mathbb{R}^k$}: asymptotic stability, synchronisation and functional central limit theorem},
	url = {https://doi.org/10.1088/1361-6544/abe17b},
	volume = {34},
	year = {2021}
}

@article{MR3719548,
	author = {Malicet, D.},
	doi = {10.1007/s00220-017-2996-5},
	fjournal = {Communications in Mathematical Physics},
	issn = {0010-3616,1432-0916},
	journal = {Comm. Math. Phys.},
	mrclass = {60B15 (60B10 60G50)},
	mrnumber = {3719548},
	mrreviewer = {Michael\ Voit},
	number = {3},
	pages = {1083–1116},
	title = {Random walks on {${\rm Homeo}(S^1)$}},
	url = {https://doi.org/10.1007/s00220-017-2996-5},
	volume = {356},
	year = {2017}
}

@article{MR2117507,
	author = {Kleptsyn, V. A. and Nalskiĭ, M. B.},
	doi = {10.1007/s10688-005-0005-9},
	fjournal = {Funktsional\cprime nyĭ Analiz i ego Prilozheniya},
	issn = {0374-1990,2305-2899},
	journal = {Funktsional. Anal. i Prilozhen.},
	mrclass = {37H99 (37A15 37A50 37E10 60G50)},
	mrnumber = {2117507},
	mrreviewer = {Vadim\ A.\ Kaĭmanovich},
	number = {4},
	pages = {36–54, 95–96},
	title = {Convergence of orbits in random dynamical systems on a circle},
	url = {https://doi.org/10.1007/s10688-005-0005-9},
	volume = {38},
	year = {2004}
}

@article{MR2425065,
	author = {Zmarrou, H. and Homburg, A. J.},
	doi = {10.3934/dcdsb.2008.10.719},
	fjournal = {Discrete and Continuous Dynamical Systems. Series B. A Journal Bridging Mathematics and Sciences},
	issn = {1531-3492,1553-524X},
	journal = {Discrete Contin. Dyn. Syst. Ser. B},
	mrclass = {37E10 (37E45 37G35 37H20)},
	mrnumber = {2425065},
	mrreviewer = {Paweł\ Góra},
	number = {2-3},
	pages = {719–731},
	title = {Dynamics and bifurcations of random circle diffeomorphisms},
	url = {https://doi.org/10.3934/dcdsb.2008.10.719},
	volume = {10},
	year = {2008}
}

@article{MR2954260,
	author = {Brofferio, S. and Buraczewski, D. and Damek, E.},
	doi = {10.1214/10-AIHP406},
	fjournal = {Annales de l'Institut Henri Poincaré Probabilités et Statistiques},
	issn = {0246-0203,1778-7017},
	journal = {Ann. Inst. Henri Poincaré Probab. Stat.},
	mrclass = {60J10 (60B15 60G50)},
	mrnumber = {2954260},
	mrreviewer = {Hélène\ Airault},
	number = {2},
	pages = {377–395},
	title = {On the invariant measure of the random difference equation {$X_n=A_n X_{n-1}+B_n$} in the critical case},
	url = {https://doi.org/10.1214/10-AIHP406},
	volume = {48},
	year = {2012}
}

@article{MR1671091,
	author = {Baxendale, P. H. and Khasminskii, R. Z.},
	doi = {10.1239/aap/1035228203},
	fjournal = {Advances in Applied Probability},
	issn = {0001-8678,1475-6064},
	journal = {Adv. in Appl. Probab.},
	mrclass = {60H25 (60J05 92D25)},
	mrreviewer = {Andrej\ A.\ Dorogovtsev},
	number = {4},
	pages = {968–988},
	title = {Stability index for products of random transformations},
	url = {https://doi.org/10.1239/aap/1035228203},
	volume = {30},
	year = {1998}
}

@article{MR1233850,
	author = {Baladi, V. and Young, L.-S.},
	fjournal = {Communications in Mathematical Physics},
	issn = {0010-3616,1432-0916},
	journal = {Comm. Math. Phys.},
	mrclass = {58F19 (28D05 58F11 58F30)},
	mrreviewer = {M.\ L.\ Blank},
	number = {2},
	pages = {355–385},
	title = {On the spectra of randomly perturbed expanding maps},
	url = {http://projecteuclid.org/euclid.cmp/1104253631},
	volume = {156},
	year = {1993}
}

@article{PhysRevE.49.1140,
	author = {Heagy, J. F. and Platt, N. and Hammel, S. M.},
	doi = {10.1103/PhysRevE.49.1140},
	issue = {2},
	journal = {Phys. Rev. E},
	month = {Feb},
	numpages = {0},
	pages = {1140–1150},
	publisher = {American Physical Society},
	title = {Characterization of on-off intermittency},
	url = {https://link.aps.org/doi/10.1103/PhysRevE.49.1140},
	volume = {49},
	year = {1994}
}

@article{MR806250,
	author = {Mañé, R.},
	fjournal = {Communications in Mathematical Physics},
	issn = {0010-3616,1432-0916},
	journal = {Comm. Math. Phys.},
	mrclass = {58F15 (28D05 58F11)},
	mrnumber = {806250},
	mrreviewer = {Hans\ G.\ Bothe},
	number = {4},
	pages = {495–524},
	title = {Hyperbolicity, sinks and measure in one-dimensional dynamics},
	url = {http://projecteuclid.org/euclid.cmp/1104114003},
	volume = {100},
	year = {1985}
}

@book{MR1369243,
	author = {Liu, P.-D. and Qian, M.},
	doi = {10.1007/BFb0094308},
	isbn = {3-540-60004-3},
	mrclass = {58F11 (60H10)},
	mrnumber = {1369243},
	mrreviewer = {Vadim\ A.\ Kaĭmanovich},
	pages = {xii+221},
	publisher = {Springer-Verlag, Berlin},
	series = {Lecture Notes in Mathematics},
	title = {Smooth ergodic theory of random dynamical systems},
	url = {https://doi.org/10.1007/BFb0094308},
	volume = {1606},
	year = {1995}
}

@article{MR805125,
	author = {Carverhill, A.},
	doi = {10.1080/17442508508833343},
	fjournal = {Stochastics},
	issn = {0090-9491},
	journal = {Stochastics},
	mrclass = {58F11 (58F25 60H20)},
	mrreviewer = {R.\ W. R. Darling},
	number = {4},
	pages = {273–317},
	title = {Flows of stochastic dynamical systems: ergodic theory},
	url = {https://doi.org/10.1080/17442508508833343},
	volume = {14},
	year = {1985}
}

@article{MR953818,
	author = {Ledrappier, F. and Young, L.-S.},
	fjournal = {Communications in Mathematical Physics},
	issn = {0010-3616,1432-0916},
	journal = {Comm. Math. Phys.},
	mrclass = {58F11},
	number = {4},
	pages = {529–548},
	title = {Dimension formula for random transformations},
	url = {http://projecteuclid.org/euclid.cmp/1104161815},
	volume = {117},
	year = {1988}
}

@article{MR1192096,
	author = {Chandra, T. K. and Goswami, A.},
	fjournal = {Sankhyā. The Indian Journal of Statistics. Series A},
	issn = {0581-572X},
	journal = {Sankhyā{} Ser. A},
	mrclass = {60F15},
	mrnumber = {1192096},
	mrreviewer = {Eric\ Rieders},
	number = {2},
	pages = {215–231},
	title = {Cesaro uniform integrability and the strong law of large numbers},
	volume = {54},
	year = {1992}
}

@article{MR722846,
	author = {Csorgo, S. and Tandori, K. and Totik, V.},
	doi = {10.1007/BF01956779},
	fjournal = {Acta Mathematica Hungarica},
	issn = {0236-5294,1588-2632},
	journal = {Acta Math. Hungar.},
	mrclass = {60F15},
	mrnumber = {722846},
	mrreviewer = {O.\ I.\ Klosov},
	number = {3-4},
	pages = {319–330},
	title = {On the strong law of large numbers for pairwise independent random variables},
	url = {https://doi.org/10.1007/BF01956779},
	volume = {42},
	year = {1983}
}

@article {MR1695915,
    AUTHOR = {Liverani, C. and Saussol, B. and Vaienti,
              S.},
     TITLE = {A probabilistic approach to intermittency},
   JOURNAL = {Ergodic Theory Dynam. Systems},
  FJOURNAL = {Ergodic Theory and Dynamical Systems},
    VOLUME = {19},
      YEAR = {1999},
    NUMBER = {3},
     PAGES = {671--685},
      ISSN = {0143-3857,1469-4417},
   MRCLASS = {37D25 (28D05 37A99 37C40 37E05)},
  MRNUMBER = {1695915},
MRREVIEWER = {Makoto\ Mori},
       DOI = {10.1017/S0143385799133856},
       URL = {https://doi.org/10.1017/S0143385799133856},
}

@article {MR4024530,
    AUTHOR = {Bahsoun, W. and Bose, C. and Ruziboev, M.},
     TITLE = {Quenched decay of correlations for slowly mixing systems},
   JOURNAL = {Trans. Amer. Math. Soc.},
  FJOURNAL = {Transactions of the American Mathematical Society},
    VOLUME = {372},
      YEAR = {2019},
    NUMBER = {9},
     PAGES = {6547--6587},
      ISSN = {0002-9947,1088-6850},
   MRCLASS = {37A05 (37A50 37E05)},
  MRNUMBER = {4024530},
MRREVIEWER = {Erika\ Alejandra\ Rada-Mora},
       DOI = {10.1090/tran/7811},
       URL = {https://doi.org/10.1090/tran/7811},
}

@article {MR4879452,
    AUTHOR = {Blessing, A. and Blumenthal, A. and Breden, M.
              and Engel, M.},
     TITLE = {Detecting random bifurcations via rigorous enclosures of large
              deviations rate functions},
   JOURNAL = {Phys. D},
  FJOURNAL = {Physica D. Nonlinear Phenomena},
    VOLUME = {476},
      YEAR = {2025},
     PAGES = {Paper No. 134617, 22},
      ISSN = {0167-2789,1872-8022},
   MRCLASS = {37H20 (60F10)},
  MRNUMBER = {4879452},
       DOI = {10.1016/j.physd.2025.134617},
       URL = {https://doi.org/10.1016/j.physd.2025.134617},
}

@article{MR727698,
	author = {Palis, J. and Takens, F.},
	doi = {10.2307/2006976},
	fjournal = {Annals of Mathematics. Second Series},
	issn = {0003-486X,1939-8980},
	journal = {Ann. of Math. (2)},
	mrclass = {58F14 (58F10)},
	mrnumber = {727698},
	mrreviewer = {Michael\ J.\ Field},
	number = {3},
	pages = {383–421},
	title = {Stability of parametrized families of gradient vector fields},
	url = {https://doi.org/10.2307/2006976},
	volume = {118},
	year = {1983}
}

@article{MR279827,
	author = {Brunovsky, P.},
	fjournal = {Commentationes Mathematicae Universitatis Carolinae},
	issn = {0010-2628,1213-7243},
	journal = {Comment. Math. Univ. Carolinae},
	mrclass = {57.47},
	mrnumber = {279827},
	mrreviewer = {K.\ Lamotke},
	pages = {559–582},
	title = {On one-parameter families of diffeomorphisms},
	volume = {11},
	year = {1970}
}

@misc{galatolo2025efficientcomputationstationarymeasures,
      title={Efficient computation of stationary measures and the Lyapunov Landscape for families random dynamical systems with smooth additive noise}, 
      author={S. Galatolo and C. Lopez Vereau and L. Marangio and I. Nisoli},
      year={2025},
      howpublished = {Preprint, {arXiv}:2508.03895},
      eprint={2508.03895},
      archivePrefix={arXiv},
      primaryClass={math.DS},
      url={https://arxiv.org/abs/2508.03895}, 
}

@article {MR576460,
    AUTHOR = {Pianigiani, G.},
     TITLE = {First return map and invariant measures},
   JOURNAL = {Israel J. Math.},
  FJOURNAL = {Israel Journal of Mathematics},
    VOLUME = {35},
      YEAR = {1980},
    NUMBER = {1-2},
     PAGES = {32--48},
      ISSN = {0021-2172},
   MRCLASS = {58F11 (28D99)},
  MRNUMBER = {576460},
MRREVIEWER = {N.\ G.\ Markley},
       DOI = {10.1007/BF02760937},
       URL = {https://doi.org/10.1007/BF02760937},
}

@incollection{MR339280,
	author = {Sotomayor, J.},
	booktitle = {Dynamical systems ({P}roc. {S}ympos., {U}niv. {B}ahia, {S}alvador, 1971)},
	mrclass = {58F99},
	mrnumber = {339280},
	mrreviewer = {Zbigniew\ Nitecki},
	pages = {561–582},
	publisher = {Academic Press, New York-London},
	title = {Generic bifurcations of dynamical systems},
	year = {1973}
}

@article {MR599464,
    AUTHOR = {Thaler, M.},
     TITLE = {Estimates of the invariant densities of endomorphisms with
              indifferent fixed points},
   JOURNAL = {Israel J. Math.},
  FJOURNAL = {Israel Journal of Mathematics},
    VOLUME = {37},
      YEAR = {1980},
    NUMBER = {4},
     PAGES = {303--314},
      ISSN = {0021-2172},
   MRCLASS = {28D05 (10K10)},
  MRNUMBER = {599464},
MRREVIEWER = {F.\ Schweiger},
       DOI = {10.1007/BF02788928},
       URL = {https://doi.org/10.1007/BF02788928}
}

@book {MR977274,
    AUTHOR = {Barnsley, Michael},
     TITLE = {Fractals everywhere},
 PUBLISHER = {Academic Press, Inc., Boston, MA},
      YEAR = {1988},
     PAGES = {xii+396},
      ISBN = {0-12-079062-9},
   MRCLASS = {58F11 (00A05 26A18 28A75 58-01 58F08 58F15 65C99)},
  MRNUMBER = {977274},
MRREVIEWER = {Jaroslav\ Stark},
}

@article{PhysRevLett.70.279,
  title = {On-off intermittency: A mechanism for bursting},
  author = {Platt, N. and Spiegel, E. A. and Tresser, C.},
  journal = {Phys. Rev. Lett.},
  volume = {70},
  issue = {3},
  pages = {279--282},
  numpages = {0},
  year = {1993},
  month = {Jan},
  publisher = {American Physical Society},
  doi = {10.1103/PhysRevLett.70.279},
  url = {https://link.aps.org/doi/10.1103/PhysRevLett.70.279}
}

@article {MR4706371,
    AUTHOR = {Yan, J. and Majumdar, M. and Ruffo, S. and Sato,
              Y. and Beck, C. and Klages, R.},
     TITLE = {Transition to anomalous dynamics in a simple random map},
   JOURNAL = {Chaos},
  FJOURNAL = {Chaos. An Interdisciplinary Journal of Nonlinear Science},
    VOLUME = {34},
      YEAR = {2024},
    NUMBER = {2},
     PAGES = {Paper No. 023128, 17},
      ISSN = {1054-1500,1089-7682},
   MRCLASS = {37E05 (37H10)},
  MRNUMBER = {4706371},
       DOI = {10.1063/5.0176310},
       URL = {https://doi.org/10.1063/5.0176310}
}

@misc{chalhoub2025analyticdependencelyapunovmoment,
      title={Analytic Dependence of the Lyapunov Moment Function and the Projective Stationary Measure for Random Matrix Products}, 
      author={C. Chalhoub and V. P. H. Goverse and J. S. W. Lamb and M. Rasmussen},
      year={2025},
      eprint={2512.05034},
      archivePrefix={arXiv},
      primaryClass={math.DS},
      url={https://arxiv.org/abs/2512.05034}, 
      howpublished = {Preprint, {arXiv}:2512.05034},
}

\end{document}